\documentclass[10pt,reqno]{amsart}
\usepackage{fullpage}
\usepackage{amsfonts,amsmath,amssymb,amsthm,bbm,enumerate,mathabx,verbatim}
\usepackage[a4paper]{geometry}

\newtheorem{theorem}{Theorem}[section]
\newtheorem{lemma}[theorem]{Lemma}
\newtheorem{proposition}[theorem]{Proposition}
\newtheorem{corollary}[theorem]{Corollary}

\theoremstyle{definition}

\newtheorem{definition}[theorem]{Definition}
\newtheorem{notation}[theorem]{Notation}

\newtheorem{remark}[theorem]{Remark}

\begin{document}
\title[Large Data Well-posedness in the Energy Space of Chern-Simons-Schr\"odinger]{Large Data Well-posedness in the Energy Space of the Chern-Simons-Schr\"odinger System}
\author{Zhuo Min Lim}
\address{Cambridge Centre for Analysis,
University of Cambridge. Centre for Mathematical Sciences, Wilberforce Road, Cambridge, CB3 0WB, United Kingdom.}
\email{Z.M.Lim@maths.cam.ac.uk}
\date{\today}

\begin{abstract}
	We consider the initial-value problem for the Chern-Simons-Schr\"{o}dinger system,
	which is a gauge-covariant Schr\"{o}dinger system in $\mathbb{R}_t\times\mathbb{R}^2_x$
	with a long-range electromagnetic field.
	We show that, in the Coulomb gauge, it is locally well-posed in $H^s$ for $s\ge 1$,
	and the solution map satisfies a local-in-time weak Lipschitz bound.
	By energy conservation, we also obtain a global regularity result.
	The key is to retain the non-perturbative part of the derivative nonlinearity in the principal operator,
	and exploit the dispersive properties of the resulting paradifferential-type principal operator
	using adapted $U^p$ and $V^p$ spaces.
\end{abstract}

\maketitle

\section{Introduction}

	The Chern-Simons-Schr\"{o}dinger system \cite{jackiw_pi_PHYS.1990, jackiw_pi_PHYS.1992}
	is a gauge-covariant version of the familiar cubic nonlinear Schr\"{o}dinger system in 2 spatial dimensions.
	Precisely, it has the form
	\begin{equation}\label{ZEQN:CSS}
		\left\{
		\begin{aligned}
			\mathbf{D}_t\phi - \mathrm{i}\,\mathbf{D}_i\mathbf{D}_i\phi &= -\mathrm{i}\kappa|\phi|^2\phi
		\;,\\
			\partial_1A_2 - \partial_2A_1 &= -\mbox{$\frac{1}{2}$}|\phi|^2
		\;,\\
			\partial_tA_i - \partial_iA_0 &= -\epsilon_{ij}\mathrm{Im}\left(\overline{\phi}\;\mathbf{D}_j\phi\right)
		\end{aligned}
		\right.
	\end{equation}
	where $\mathbf{D}_\alpha$ are the covariant derivative operators defined by
	\[
		\mathbf{D}_\alpha := \partial_\alpha + \mathrm{i}A_\alpha
	\;.\]
	Here and in the rest of the article, the index $\alpha=0$ refers to the time variable $t$, and $\alpha=1,2$
	refers to the spatial variables $x_1,x_2$.
	Latin indices are assumed to refer only to the spatial variables.
	Repeated indices are always summed over.
	Finally, $\epsilon_{ij}$ denotes the standard anti-symmetric 2-form with $\epsilon_{12}=1$.

	The Chern-Simons-Schr\"{o}dinger system (\ref{ZEQN:CSS}) is a non-relativistic Lagrangian field theory
	whose action is given by
	\[
		\iint_{\mathbb{R}_t\times\mathbb{R}_x^2} \left(\frac{1}{2}\mathrm{Im}\left(\overline{\phi}\mathbf{D}_t\phi\right)
			+ \frac{1}{2}\left|\mathbf{D}_x\phi\right|^2 + \frac{\kappa}{4}|\phi|^4\right)\mathrm{d}x\;\mathrm{d}t
			+ \frac{1}{2}\iint_{\mathbb{R}_t\times\mathbb{R}_x^2} A\wedge\mathrm{d}A
	\]
	where $A=A_0\mathrm{d}t+A_1\mathrm{d}x_1+A_2\mathrm{d}x_2$ is the electromagnetic potential 1-form.
	The system (\ref{ZEQN:CSS}) describes the effective dynamics of a large system of non-relativistic charged quantum particles,
	interacting with each other and also with the self-generated electromagnetic field.
	It has been proposed as a theoretical model for various condensed matter phenomena
	such as the quantum Hall effect and high temperature superconductivity.
	The real-valued parameter $\kappa$ is called the coupling strength and measures the strength of the binary interactions.

	The Chern-Simons-Schr\"{o}dinger system (\ref{ZEQN:CSS}) enjoys the following two conservation laws:
	that of the total mass,
	\[
		\mathcal{M}(t) := \frac{1}{2}\int_{\mathbb{R}^2} |\phi(t,x)|^2\;\mathrm{d}x = \mathcal{M}(0)
	\;,\]
	and that of the energy
	\[
		\mathcal{E}(t) := \int_{\mathbb{R}^2} \left(\frac{1}{2}\left|\mathbf{D}_x\phi(t,x)\right|^2 + \frac{\kappa}{4}|\phi(t,x)|^4\right)\mathrm{d}x
			= \mathcal{E}(0)
	\;.\]

	In this article, we are concerned with the issue of well-posedness of (\ref{ZEQN:CSS}) for large initial data, on or above the energy regularity.

	Before we can address the initial-value problem, however,
	note that (\ref{ZEQN:CSS}) is gauge-invariant in the sense that
	if $(\phi,A)$ is a solution, then so is
	\[
		\left(\mathrm{e}^{\mathrm{i}\chi}\phi, A+\mathrm{d}\chi\right)
	\]
	for any sufficiently well-behaved function $\chi:\mathbb{R}_t\times\mathbb{R}^2_x\rightarrow\mathbb{R}$.
	Thus, in order that the evolution of (\ref{ZEQN:CSS}) be well-defined,
	this gauge-invariance must be eliminated by imposing an additional constraint equation,
	that is, by fixing a gauge.

	In this article, we will work only in the {\em Coulomb gauge}, which is defined by
	\[
		\mathrm{div}\,A_x := \partial_1A_1 + \partial_2A_2 = 0
	\;.\]
	With the Coulomb gauge condition, straightforward manipulations reduce (\ref{ZEQN:CSS}) to the following equivalent system
	\begin{equation}\label{ZEQN:CSSCoul}
		\left\{
		\begin{aligned}
			\left(\partial_t-\mathrm{i}\triangle\right)\phi &= -2A_x\cdot\nabla\phi - \mathrm{i}A_0\phi  - \mathrm{i}|A_x|^2\phi - \mathrm{i}\kappa|\phi|^2\phi \;,\\
			-\triangle A_i &= -\mbox{$\frac{1}{2}$}\epsilon_{ij}\partial_j\left(|\phi|^2\right) \;,\\
			-\triangle A_0 &= -\mathrm{Im}\left(\nabla\overline{\phi}\wedge\nabla\phi\right) - \mathrm{rot}\left(A|\phi|^2\right) \;.
		\end{aligned}
		\right.
	\end{equation}
	In the above, we have denoted the cross product by $a\wedge b := a_1b_2 - a_2b_1$,
	and we have also denoted by $A_x=(A_1,A_2)$ the spatial components of $A$,
	and by $\nabla=(\partial_1,\partial_2)$ the spatial derivatives.
	We will use these conventions throughout this article.
	Observe that, in the Coulomb gauge, the electromagnetic potentials $A_\alpha$ are no longer dynamical variables,
	but are uniquely determined at each time $t$ by solving a Poisson equation.
	In particular, for the initial value problem (\ref{ZEQN:CSSCoul}), one need only prescribes $\phi(0)$ as initial data.

	Our goal in this article is to prove that, for $s\ge 1$, the system (\ref{ZEQN:CSSCoul}) is locally well-posed in $H^s$,
	and that a $H^s$ solution can be continued so long as its $H^1$ norm does not blow up.
	In particular, global well-posedness holds in the defocusing case $\kappa>0$, and also for initial data having sufficiently small $H^1$ norm when $\kappa\le 0$.
	
	In a forthcoming article, we will use our global well-posedness result to establish scattering in weighted spaces of large-data solutions to (\ref{ZEQN:CSSCoul})
	when $\kappa>0$.

	To our knowledge, the first well-posedness result for (\ref{ZEQN:CSSCoul}) was established by Berg\'{e}-de\,Bouard-Saut
	in \cite{berge_debouard_saut_Nonlinearity.1995}.
	With a regularisation argument, they also established, in the same paper,
	global existence of $H^1$ solutions for $H^1$ initial data having sufficiently small total mass,
	but they did not prove that such solutions are unique.
	Unconditional uniqueness in $L^\infty_tH^1$ of solutions for (\ref{ZEQN:CSSCoul}) was later demonstrated by Huh in \cite{huh_AbstrApplAnal.2013}
	using clever energy estimates,
	but the continuous dependence of these $H^1$ solutions on their initial data remains open.
	We note that neither of these approaches require exploiting the dispersive features of (\ref{ZEQN:CSSCoul}).

	On the issue of low-regularity well-posedness of (\ref{ZEQN:CSS}),
	the best result at the present moment is due to Liu-Smith-Tataru \cite{liu_smith_tataru_IMRN.2014}.
	They proved in the local well-posedness of (\ref{ZEQN:CSS}) for small data
	in all subcritical Sobolev spaces $H^\sigma$ with any $\sigma>0$.
	Their work differs from the present article in the use of the heat gauge $A_0=\mathrm{div}\,A_x$ rather than the Coulomb gauge,
	and the small data assumption is then crucial to construct the heat gauge.
	Their proof relies on various technical local smoothing and maximal function spaces 
	originally developed in the analysis of the Schr\"{o}dinger map system
	\cite{ionescu_kenig_DifferentialIntegralEquations.2006, ionescu_kenig_CommunMathPhys.2007, bejenaru_ionescu_kenig_AdvMath.2007};
	see also \cite{smith_JMathPhys.2014} for a more thorough comparison between (\ref{ZEQN:CSS}) and the Schr\"{o}dinger map system.
	We remark that our approach is very much technically simpler than theirs.

	On the issue of global well-posedness of (\ref{ZEQN:CSSCoul}), there have been at least two results.
	The first result is due to Oh-Pusateri \cite{oh_pusateri_IMRN.2015} who proved that, given initial data which are small both in $H^2$
	and in some weighted Sobolev spaces,
	the corresponding solution to (\ref{ZEQN:CSSCoul}) exists globally, and moreover scatters to a linear Schr\"{o}dinger solution in a weaker topology.
	The second result is due to Liu-Smith \cite{liu_smith_RevMatIberoamer.2016}, who studied (\ref{ZEQN:CSSCoul}) under equivariant symmetry.
	They proved global well-posedness and linear scattering in the critical space $L^2_x$ for equivariant solutions.

\subsection{Statement of Results}

	In this article, we show that the Chern-Simons-Schr\"{o}dinger system in the Coulomb gauge, (\ref{ZEQN:CSSCoul}),
	is locally well-posed for large initial data in $H^s$, $s\ge 1$.
	Denoting by $\mathbb{B}_{H^s}(D)$ the closed ball in $H^s$ of radius $D$, we state our main result as follows.

	\begin{theorem}[Main Theorem] \label{ZTHM:MainThm}
		Let $s\ge 1$.
		\begin{enumerate}[(i)]
		\item
			For any $D>0$, there exists $T=T(s,D)>0$ such that, given any initial data $\phi^{\mathrm{in}}\in\mathbb{B}_{H^s}(D)$,
			there exists 
			a unique solution $\phi\in C_b((-T,T),H^s)$ to (\ref{ZEQN:CSSCoul}) with $\phi(0)=\phi^{\mathrm{in}}$, which is the uniform limit of smooth solutions.

		\item
			With $D>0$ and $T=T(s,D)$ as above, the solution map
			\[
				\mathbb{B}_{H^s}(D)\ni \phi(0) \mapsto \phi\in C_b((-T,T),H^s)
			\]
			is continuous, and satisfies the local-in-time weak Lipschitz bound
			\begin{equation} \label{ZEQN:WkLipBd}
				\left\|\phi-\phi'\right\|_{L^\infty_t((-T,T),H^{s-1})} \le C\left\|\phi(0)-\phi'(0)\right\|_{H^{s-1}}
			\;.\end{equation}
		\end{enumerate}
		
		Moreover, persistence of regularity holds: for any $D_1>0$, there exists $T_\star=T_\star(s,D_1)>0$ and $C_\star=C_\star(s,D_1)>0$
		such that any $H^s$ solution $\phi$, whose initial data satisfy $\|\phi(0)\|_{H^1}\le D_1$,
		can be continued to $(-T_\star,T_\star)$, and
		\begin{equation} \label{ZEQN:NormGrowthEstmt}
			\left\|\phi\right\|_{L^\infty_t((-T_\star,T_\star),H^s)} \le C_\star\left\|\phi(0)\right\|_{H^s}
		\;.\end{equation}
		In particular we have the blow-up criterion:
		A maximal-in-time $H^s$ solution $\phi$ to (\ref{ZEQN:CSSCoul}) is global if and only if $\|\phi(t)\|_{H^1}$ does not blow up in finite time.
	\end{theorem}

	Using energy conservation, we can then obtain the following global well-posedness result as an easy corollary of Theorem \ref{ZTHM:MainThm}.

	\begin{corollary}[Global well-posedness in energy space] \label{ZCOR:MainThmCor}
		Let $s\ge 1$ and let $\phi$ be a local-in-time $H^s$ solution to (\ref{ZEQN:CSSCoul}).
		Either assume $\kappa>0$, or, assume $\kappa\le 0$ and $\|\phi(0)\|_{L^2}$ is sufficiently small depending on $\kappa$.
		Then $\|\phi(t)\|_{H^1}$ is controlled by the conserved quantities,
		\[
			\|\phi(t)\|_{H^1} \le C\left(\mathcal{M}(0), \mathcal{E}(0)\right)
		\;.\]
		Consequently, $\phi$ can be continued a global solution $\phi\in C_{\mathrm{b}}(\mathbb{R},H^s)$.
	\end{corollary}

\subsection{Overview of the proof}
	We now outline the main ideas of the proof of Theorem \ref{ZTHM:MainThm}.
	Observe that (\ref{ZEQN:CSSCoul}) is time-reversible,
	therefore we will, in the rest of the article, focus exclusively on proving well-posedness forward in time.

	The primary difficulty in establishing a well-posedness result for (\ref{ZEQN:CSSCoul}) at limited regularity,
	when energy methods alone are insufficient,
	is the presence
	of the nonlinear term $2A_x\cdot\nabla\phi$, involving a derivative of $\phi$,
	in the right-hand side of the first equation of (\ref{ZEQN:CSSCoul}).
	Indeed, the application of standard dispersive estimates, such as the Strichartz estimates, in the naive iteration scheme
	incurs a loss of derivatives on the right-hand side,
	and the estimates will fail to close.
	
	To make matters worse, the electromagnetic interaction is long-range in the sense that $A_x$ does not decay more quickly than $1/|x|$ for large $|x|$.
	This slow decay can be seen from the representation formula
	\[
		A_i(t,x) = \frac{1}{4\pi}\epsilon_{ij}\int_{\mathbb{R}^2} \frac{x_j-y_j}{|x-y|^2}\left|\phi(t,y)\right|^2\mathrm{d}y
	\]
	given by the Biot-Savart law.
	The slow decay causes severe difficulty in using local smoothing estimates \cite{kenig_ponce_vega_InventMath.2004}
	to recover the loss of derivatives by performing estimates in appropriate weighted function spaces.

	The above considerations suggest that the difficult nonlinearity $2A_x\cdot\nabla\phi$ is non-perturbative,
	and motivates the strategy in the present work.
	Our strategy is primarily inspired by the proof, due to Bejenaru-Tataru,
	of global well-posedness in the energy space of the Maxwell-Schr\"{o}dinger system \cite{bejenaru_tataru_CommunMathPhys.2009}.

	We perform a paraproduct decomposition on this derivative nonlinearity $2A_x\cdot\nabla\phi$.
	For a time-dependent spatial 1-form $B:[0,T)\times\mathbb{R}^2\rightarrow\mathbb{R}^2$, define the operators $\mathfrak{P}_B$ and $\mathfrak{Q}_B$ by
	\begin{equation*}\begin{split}
		\mathfrak{P}_Bw &:= 
			\sum_{\lambda\ge 1} \left[\mathrm{P}_{\le 2^{-5}\lambda}B_i\,\mathrm{P}_\lambda\partial_iw
				+ \mathrm{P}_\lambda\left(\mathrm{P}_{\le 2^{-5}\lambda}B_i\;\partial_iw\right)\right]
	\;,\\
		\mathfrak{Q}_Bw &:=
			\sum_{\lambda\ge 1} \left[\mathrm{P}_\lambda B_i\,\mathrm{P}_{< 2^5\lambda}\partial_iw
				+ \mathrm{P}_{<2^5\lambda}\left(\mathrm{P}_\lambda B_i\;\partial_iw\right)\right]
	\;,\end{split}\end{equation*}
	where $\mathrm{P}_\lambda$ are inhomogeneous Littlewood-Paley frequency restriction operators,
	i.e. $\mathrm{P}_1$ restricts to all low frequencies,
	and the sum above is taken over dyadic frequencies.
	We refer the reader to the next section for an explanation of the notations.
	We can then write
	\[
		2A_x\cdot\nabla\phi = \mathfrak{P}_{A_x}\phi + \mathfrak{Q}_{A_x}\phi
	\;.\]
	Heuristically, the term $\mathfrak{Q}_{A_x}\phi$ is well-behaved pertubatively.
	Indeed,
	because the derivative acts on a low frequency term in
	the term $\mathfrak{Q}_{A_x}\phi$,
	we expect to this term to obey better bounds than $\phi\nabla A_x$.
	Now,
	from the second equation in (\ref{ZEQN:CSSCoul}), we expect $\nabla A_x$ to have the regularity of $|\phi|^2$.
	Therefore, the term $\mathfrak{Q}_{A_x}\phi$ should be better behaved than
	the standard power nonlinearity $|\phi|^2\phi$,
	and in particular should be amenable to a perturbative treatment.

	The term $\mathfrak{P}_{A_x}\phi$ is the truly non-perturbative part of the derivative nonlinearity $2A_x\cdot\nabla\phi$.
	Therefore, we retain it in our principal operator and rewrite the first equation of (\ref{ZEQN:CSSCoul}) as the quasilinear evolution equation,
	\begin{equation}\label{ZEQN:IntroQuasilinearFormln}
		\left(\partial_t-\mathrm{i}\triangle+\mathfrak{P}_{A_x}\right)\phi
		=
		-\mathfrak{Q}_{A_x}\phi - \mathrm{i}A_0\phi  - \mathrm{i}|A_x|^2\phi - \mathrm{i}\kappa|\phi|^2\phi
	\;.\end{equation}

	An essential feature of the present work, then, is the understanding of principal operators of the form
	$(\partial_t-\mathrm{i}\triangle+\mathfrak{P}_B)$.
	At the very least, we require that the homogeneous linear equation
	\begin{equation}\label{ZEQN:IntroNewPrOpHomg}
		\left(\partial_t-\mathrm{i}\triangle+\mathfrak{P}_B\right)u = 0
	\end{equation}
	should be well-posed in Sobolev spaces, and the solutions should moreover satisfy appropriate dispersive estimates.
	To this end, we impose the conditions that $B\in L^\infty_t([0,T),L^\infty_x)$, $\mathrm{div}\,B=0$ and $\nabla B\in L^1_t([0,T),L^\infty_x)$,
	and call such time-dependent spatial 1-forms {\em admissible forms}.
	Note that the condition $\mathrm{div}\,B=0$ formally guarantees that the evolution of (\ref{ZEQN:IntroNewPrOpHomg}) conserves the $L^2_x$ norm.
	We show that, provided $B$ is an admissible form, (\ref{ZEQN:IntroNewPrOpHomg}) can be uniquely solved in Sobolev spaces on the time interval $[0,T)$,
	and the solutions satisfy Strichartz estimates with a loss of derivatives.

	In order to utilise this functional framework for solving the inhomogeneous equation
	\begin{equation}\label{ZEQN:IntroNewPrOpInhomg}
		\left(\partial_t-\mathrm{i}\triangle+\mathfrak{P}_B\right)u = f
	\end{equation}
	in an appropriate Sobolev space $H$,
	we define the associated $U^p$ and $V^p$ spaces
	\cite{koch_tataru_CommPureApplMath.2005, koch_tataru_IMRN.2007, hadac_herr_koch_AnnIHP.2009},
	namely $U^p_BH$ and $V^p_BH$,
	which are adapted to the principal operator $(\partial_t-\mathrm{i}\triangle+\mathfrak{P}_B)$.
	This gives us a functional calculus for solving (\ref{ZEQN:IntroNewPrOpInhomg})
	in the spaces $U^2_BH$.
	The construction of our functional framework is accomplished in Section \ref{ZSECT:TheModPrOp}.

	We can now apply our functional calculus to solve (\ref{ZEQN:CSSCoul}) using the following iteration scheme
	\begin{equation}\label{ZEQN:IntroItrtnScheme}
		\left\{
		\begin{aligned}
			\left(\partial_t-\mathrm{i}\triangle+\mathfrak{P}_{A_x^{[n-1]}}\right)\phi^{[n]} &=
				-\mathfrak{Q}_{A_x^{[n]}}\phi^{[n]} - \mathrm{i}A_0^{[n]}\phi^{[n]} 
				- \mathrm{i}\left|A^{[n]}_x\right|^2\phi^{[n]} - \mathrm{i}\kappa\left|\phi^{[n]}\right|^2\phi^{[n]} \;,\\
			-\triangle A_i^{[n]} &= -\mbox{$\frac{1}{2}$}\epsilon_{ij}\partial_j\left(\left|\phi^{[n]}\right|^2\right) \;,\\
			-\triangle A_0^{[n]} &= -\mathrm{Im}\left(\nabla\overline{\phi^{[n]}}\wedge\nabla\phi^{[n]}\right)
				- \mathrm{rot}\left(A^{[n]}\left|\phi^{[n]}\right|^2\right) \;,\\
			\phi^{[n]}(0) &= \phi^{\mathrm{in}} \;,\\
		\end{aligned}
		\right.
	\end{equation}
	which is initialised with $A_x^{[0]}=0$.
	Our functional calculus now allows us to solve (\ref{ZEQN:IntroItrtnScheme}) at each iteration $n$
	via a contraction mapping argument in the function space $U^2_{A_x^{[n-1]}}H$ where $H$ is a generalised Sobolev space containing $H^s$.
	The key point is that every $A_x^{[n]}$ generated by this iterative scheme will be an admissible form
	whose size depends only on the size $D$ of the initial data $\phi^{\mathrm{in}}$.
	As a consequence, the existence time of (\ref{ZEQN:IntroItrtnScheme}) is bounded below independently of $n$,
	and the $L^\infty_tH$ norm of the iterates $\phi^{[n]}$ are also bounded above independently of $n$.
	These are accomplished in Section \ref{ZSECT:ConstructItrtnScheme}.

	The convergence of the iteration scheme (\ref{ZEQN:IntroItrtnScheme}) is addressed in Section \ref{ZSECT:CnvgItrtnScheme}.
	We are able to obtain a weak Lipschitz bound between the iterates,
	\[
		\left\|\phi^{[n+1]}-\phi^{[n]}\right\|_{L^\infty_tH^{s-1}} \le \frac{1}{2}\left\|\phi^{[n]}-\phi^{[n-1]}\right\|_{L^\infty_tH^{s-1}}
	\]
	which shows that the iterates $\{\phi^{[n]}\}_{n=1}^\infty$ converge in $L^\infty_tH^{s-1}$.
	On the other hand, $\{\phi^{[n]}\}_{n=1}^\infty$ are bounded uniformly in $L^\infty_tH^{\mathfrak{m}}$
	for some generalised Sobolev space $H^{\mathfrak{m}}$,
	such that the embedding $H^s\hookrightarrow H^{\mathfrak{m}}$ is compact.
	Thus, by interpolation, the iterates $\{\phi^{[n]}\}_{n=1}^\infty$ converge in $L^\infty_tH^s$ as well,
	and it is straightforward to check that the limit is the desired solution to the system (\ref{ZEQN:CSSCoul}).
	The same arguments also prove the continuity of the solution map, and the weak Lipschitz bound between two solutions.

	Finally, we remark that our strategy uses only linear dispersive estimates and not bilinear or multilinear Strichartz estimates.
	Consequently we are not able to address well-posedness of (\ref{ZEQN:CSSCoul}) below $H^1$.
	The reason is that $A_0$ exhibits very bad $\mbox{high}\times\mbox{high}\rightarrow\mbox{low}$ interactions,
	and the proof of Lemma \ref{ZLEM:DiffEstmt_N2t} breaks down when $s<1$.

\subsection{Acknowledgements}
	The author gratefully acknowledges generous financial support from St. John's College, Cambridge.
	This work, which forms part of the author's PhD thesis, is also supported by the UK
	Engineering and Physical Sciences Research Council (EPSRC) grant EP/H023348/1 for the
	University of Cambridge Centre for Doctoral Training, the Cambridge Centre for Analysis.

\section{Notations and Preliminaries} \label{ZSECT:Preliminaries}

	We fix $s\ge 1$ once and for all.
	All constants in this article are allowed to depend on the coupling strength $\kappa$ but,
	unless otherwise stated, not on any other parameters.
	If $A$ and $B$ are nonnegative quantities, we write $A\lesssim B$ if there is a constant $C$ such that $A\le CB$.
	We write $A\approx B$ if $A\lesssim B$ and $B\lesssim A$.

	Throughout this article we will use the standard Lebesgue spaces $L^r_x:=L^r(\mathbb{R}^2_x)$,
	mixed space-time Lebesgue spaces $L^q_tL^r_x$,
	and spaces $C_b([0,T),X)$ of continuous bounded functions where $X$ is a Banach space of functions on $\mathbb{R}^2_x$.
	Almost always in this article, the time interval is not taken to be all of $\mathbb{R}$,
	but rather a finite time interval $[0,T)$ for some $T>0$.
	For ease of notation, we therefore denote $L^q_tL^r_x[T] := L^q_t([0,T),L^r_x)$
	and $C_{\mathrm{b}}X[T] := C_{\mathrm{b}}([0,T),X)$.

\subsection{Fourier analysis}
	We will occasionally take Fourier transforms over the spatial variables $x$, but never over the time variable $t$.
	Our convention for the Fourier transform will be
	\[
		\widehat{u}(\xi) := \mathcal{F}u(\xi) := \int_{\mathbb{R}^2} \mathrm{e}^{-\mathrm{i}x\cdot\xi}u(x)\,\mathrm{d}x
	\;.\]
	We denote the Riesz transform by
	\[
		\mathfrak{R}_i := \frac{\partial_i}{|\nabla|}
	\;.\]
	It is a standard fact in harmonic analysis that the Riesz transforms are bounded linear maps $L^p(\mathbb{R}^2_x)\rightarrow L^p(\mathbb{R}^2_x)$
	for every $p\in (1,\infty)$.

	We will very often make use of the Biot-Savart law,
	\[
		\frac{\partial_i}{(-\triangle)}F(x) = -\frac{1}{2\pi}\int_{\mathbb{R}^2} \frac{x_i-y_i}{|x-y|^2}F(y)\,\mathrm{d}y
	\;.\]
	This representation formula is amenable to the Hardy-Littlewood-Sobolev inequality
	for functions supported at low frequencies,
	when Bernstein's inequality does not directly apply due to the presence of the singular Fourier multiplier $|\mathrm{D}_x|^{-1}$.

	We now recall the inhomogeneous Littlewood-Paley decomposition.
	Denote by
	\[
		\mathfrak{D} := \left\{2^k\;\big|\; k\in\mathbb{Z}_{\ge 0}\right\}
	\]
	the set of all dyadic frequencies.
	Fix, once and for all, a smooth, radial, non-increasing function $\varphi_1:\mathbb{R}^2_\xi\rightarrow\mathbb{R}$
	such that $\varphi_1(\xi)\equiv 1$ on $|\xi|\le 1$,
	and $\varphi_1(\xi)\equiv 0$ on $|\xi|\ge 2$.
	For $\lambda\in\mathfrak{D}, \lambda\ge 2$, set
	\[
		\varphi_\lambda(\xi) := \varphi_1\left(\mbox{$\frac{1}{\lambda}$}\xi\right)
			- \varphi_1\left(\mbox{$\frac{2}{\lambda}$}\xi\right)
	\;.\]
	For all $\lambda\in\mathfrak{D}$, we define $\mathrm{P}_\lambda:=\varphi_\lambda(\mathrm{D}_x)$
	the standard Littlewood-Paley restriction.
	Equivalently,
	\[
		\mathrm{P}_\lambda u(x) = \int_{\mathbb{R}^2} \widecheck{\varphi_\lambda}(x-y)u(y)\,\mathrm{d}y
	\;.\]
	Henceforth, we will reserve the letters $\lambda,\mu,\nu,\rho$ for dyadic frequencies, i.e. elements of $\mathfrak{D}$.
	When summing over $\lambda,\mu,\nu,\rho$, the summation is implicitly taken over all of $\mathfrak{D}$
	unless otherwise stated.

	We define
	\[
		\mathrm{P}_{\le\lambda} := \sum_{\mu\le\lambda} \mathrm{P}_\mu
	\;,\quad
		\mathrm{P}_{<\lambda} := \mathrm{P}_{\le\frac{1}{2}\lambda}
	\;.\]
	We will also, for ease of exposition, abuse notation in using the following operators
	\[
		\mathrm{P}_{\ll\lambda} := \mathrm{P}_{\le 2^{-m}\lambda} \;,\quad
		\mathrm{P}_{\lesssim\lambda} := \mathrm{P}_{\le 2^{m}\lambda} \;,\quad
		\mathrm{P}_{\approx\lambda} := \mathrm{P}_{\lesssim\lambda}-\mathrm{P}_{\ll\lambda}
	\;,\]
	where $m$ denotes fixed universal positive integers, whose values may change from line to line and can be appropriately chosen by the reader if so desired.

	In this article, we will equip the Sobolev space $H^\sigma$ with the equivalent Besov space norm,
	\[
		\left\|w\right\|_{H^\sigma}^2 := \sum_\lambda \lambda^{2\sigma}\left\|\mathrm{P}_\lambda w\right\|_{L^2_x}
	\;.\]
	These norms will be consistent with those of the following family of function spaces.

	\begin{definition}
		A {\em Sobolev weight} is a function $\mathfrak{m}:\mathfrak{D}\rightarrow(0,\infty)$ such that $\mathfrak{m}(1)=1$, and
		there exist constants $c\le 1$ and $C\ge 1$ such that
		\[
			c\le \frac{\mathfrak{m}(2\lambda)}{\mathfrak{m}(\lambda)} \le C \quad\mbox{for all } \lambda\in\mathfrak{D}
		\;.\]
		Given a Sobolev weight $\mathfrak{m}$, define the {\em generalised Sobolev space}
		$H^{\mathfrak{m}}\subset\mathcal{S}'(\mathbb{R}^2_x)$ to be the Hilbert space whose inner product is given by
		\[
			(v,w)_{H^{\mathfrak{m}}}
			:=
			\sum_\lambda \mathfrak{m}(\lambda)^2 \int_{\mathbb{R}^2} \mathrm{P}_\lambda v(x)\,\overline{\mathrm{P}_\lambda w(x)}\,\mathrm{d}x
		\;.\]
		Moreover, for a Sobolev weight $\mathfrak{m}$, define the quantities $[\mathfrak{m}]_\star, [\mathfrak{m}]^\star, [\mathfrak{m}]$ by
		\[
			[\mathfrak{m}]_\star := \inf_\lambda \log_2\left(\frac{\mathfrak{m}(2\lambda)}{\mathfrak{m}(\lambda)}\right) \;,\quad
			[\mathfrak{m}]^\star := \sup_\lambda \log_2\left(\frac{\mathfrak{m}(2\lambda)}{\mathfrak{m}(\lambda)}\right)
		\;,\]
		\[
			[\mathfrak{m}] := \max\left(-[\mathfrak{m}]_\star, [\mathfrak{m}]^\star\right)
		\;.\]
	\end{definition}
	
	\begin{remark}
		By definition, we have $[\mathfrak{m}]_\star\le 0$ while $[\mathfrak{m}]^\star, [\mathfrak{m}]\ge 0$.
		Furthermore,
		\begin{equation} \label{ZEQN:SobWtRmk_EQN01}
			\mathfrak{m}(\lambda)\le 2^{k[\mathfrak{m}]}\mathfrak{m}(\mu)
			\quad\mbox{whenever }
			\left|\log_2\left(\frac{\lambda}{\mu}\right)\right|\le k
		\;.\end{equation}
	\end{remark}
	
	\begin{lemma} \label{ZLEM:HmDuality}
		Let $\mathfrak{m}$ be a Sobolev weight.
		Let $(H^{\mathfrak{m}})^\ast$ be the dual space of $H^{\mathfrak{m}}$
		extending the $L^2_x$ self-duality.
		Then $(H^{\mathfrak{m}})^\ast$ is isomorphic to $H^{\mathfrak{m}^{-1}}$
		with equivalent norms. More precisely,
		\begin{equation} \label{ZEQN:HmDuality_EQN01}
			C_1\left([\mathfrak{m}]\right)\left\|v\right\|_{H^{\mathfrak{m}^{-1}}}
			\le
			\left\|v\right\|_{(H^{\mathfrak{m}})^\ast}
			\le
			C_2\left([\mathfrak{m}]\right)\left\|v\right\|_{H^{\mathfrak{m}^{-1}}}
		\end{equation}
		for some constants $C_2\ge C_1>0$ depending only on $[\mathfrak{m}]$.
	\end{lemma}
	\begin{proof}
		Given $v\in H^{\mathfrak{m}^{-1}}$, we may set
		\[
			w := \sum_\mu \mathfrak{m}(\mu)^{-2}\mathrm{P}_\mu v
		\;.\]
		Clearly $w\in H^{\mathfrak{m}}$, and from (\ref{ZEQN:SobWtRmk_EQN01}) we have
		\[
			\left\|w\right\|_{H^\mathfrak{m}} \le C\left([\mathfrak{m}]\right)\left\|v\right\|_{H^{\mathfrak{m}^{-1}}}
		\;.\]
		Therefore,
		\begin{equation*}\begin{split}
			\left(w,v\right)_{L^2_x}
		=\;&
			\sum_\mu \mathfrak{m}(\mu)^{-2}\left(\mathrm{P}_\mu v,v\right)_{L^2_x}
		\ge
			\sum_\mu \mathfrak{m}(\mu)^{-2}\left\|\mathrm{P}_\mu v\right\|_{L^2_x}^2
		\ge
			C\left([\mathfrak{m}]\right)\left\|v\right\|_{H^{\mathfrak{m}^{-1}}}\left\|w\right\|_{H^{\mathfrak{m}}}
		\;.\end{split}\end{equation*}
		This verifies the first inequality in (\ref{ZEQN:HmDuality_EQN01}).
		
		We turn to the second inequality in (\ref{ZEQN:HmDuality_EQN01}).
		Given $v\in H^{\mathfrak{m}^{-1}}, w\in H^{\mathfrak{m}}$, we have from the Cauchy-Schwarz inequality that
		\begin{equation*}\begin{split}
			\left|(w,v)_{L^2_x}\right|
		\le\;&
			\sum_\lambda \left|\left(\mathrm{P}_\lambda w, \mathrm{P}_{\approx\lambda}v\right)_{L^2_x}\right|
		\le
			\left\|w\right\|_{H^{\mathfrak{m}}}
			\left(\sum_\lambda \mathfrak{m}(\lambda)^{-2}\left\|\mathrm{P}_{\approx\lambda}v\right\|_{L^2_x}^2\right)^{\frac{1}{2}}
		\;.\end{split}\end{equation*}
		Since (\ref{ZEQN:SobWtRmk_EQN01}) gives us
		\[
			\sum_\lambda \mathfrak{m}(\lambda)^{-2}\left\|\mathrm{P}_{\approx\lambda}v\right\|_{L^2_x}^2
			\le
			C\left([\mathfrak{m}]\right)\left\|v\right\|_{H^{\mathfrak{m}^{-1}}}^2
		\;,\]
		we deduce the second inequality in (\ref{ZEQN:HmDuality_EQN01}).
	\end{proof}

\subsection{Strichartz estimates}
	We now recall the well-known Strichartz estimates \cite{kato_AnnIHP.1987, cazenave_weissler_ManuscriptaMath.1988}
	for the Schr\"{o}dinger equation in two space dimensions.

	\begin{definition} \label{ZDEFN:StrichartzPair}
		We say $(q,r)\in[2,\infty]^2$ is a {\em Strichartz pair} if $\frac{2}{q}+\frac{2}{r}=1$ and $r<\infty$.
	\end{definition}

	\begin{lemma}[Strichartz estimates]
		Suppose $(q_1,r_1)$ and $(q_2,r_2)$ are Strichartz pairs.
		Assume $u:[0,T)\times\mathbb{R}^2_x\rightarrow\mathbb{C}$ is an $L^2_x$ solution to the Schr\"{o}dinger equation,
		\[
			\left(\partial_t-\mathrm{i}\triangle\right) u = f
		\;.\]
		Then the estimate
		\[
			\left\|u\right\|_{L^{q_1}_tL^{r_1}_x[T]} \lesssim \|u(0)\|_{L^2_x} + \|f\|_{L^{q_2'}_tL^{r_2'}_x[T]}
		\]
		holds with the implicit constant depending on $q_1,q_2$ but not on $T$.
		Here $q_2',r_2'$ denotes the H\"{o}lder conjugates of $q_2,r_2$ respectively, i.e. $1=\frac{1}{q_2}+\frac{1}{q_2'}=\frac{1}{r_2}+\frac{1}{r_2'}$.
	\end{lemma}

\subsection{$U^p$ and $V^p$ spaces}
	As mentioned in the Introduction,
	we need a functional framework built on spaces of the $U^p$ and $V^p$ type
	\cite{koch_tataru_CommPureApplMath.2005, koch_tataru_IMRN.2007, hadac_herr_koch_AnnIHP.2009}.
	We now recall the definitions and basic properties of these spaces.

	\begin{definition}
		Let $T>0$ and let $X$ be a separable Banach space over $\mathbb{C}$. Let $p\in[1,\infty)$.
		We define a $U^pX[T]$ {\em atom} to be a function $a:[0,T)\rightarrow X$ of the form
		\[
			a(t) = \sum_{k=0}^{K-1} \mathbbm{1}_{[t_k,t_{k+1})}(t)a_k
		\]
		where $K\in\mathbb{Z}_{>0}$, $0=t_0<t_1<\ldots<t_K=T$, and $\sum_{k=0}^{K-1} \|a_k\|_X^p = 1$.

		The Banach space $U^pX[T]$ is defined to be the atomic space over the $U^pX[T]$ atoms.
		More precisely, $U^pX[T]$ consists of all functions $a:[0,T)\rightarrow X$ admitting a representation
		\[
			u = \sum_{j=1}^\infty c_ja_j \;,\quad a_j \mbox{ are $U^pX[T]$ atoms }\;,\quad \{c_j\}_{j=1}^\infty\in \ell^1
		\;,\]
		equipped with the norm
		\[
			\left\|u\right\|_{U^pX[T]} := \inf\left\{\sum_{j=1}^\infty |c_j|
				\;\Big|\; u = \sum_{j=1}^\infty c_ja_j \;,\; \{c_j\}_{j=1}^\infty\in\ell^1 \;,\;
					a_j\mbox{ are $U^pX[T]$ atoms}
			\right\}
		\;.\]

		We define $\mathrm{D}U^pX[T]$ to be the space of distributional derivatives of functions in $U^pX[T]$, equipped with the norm
		\[
			\left\|f\right\|_{\mathrm{D}U^pX[T]}
			:=
			\left\|\int_0^t f(t')\,\mathrm{d}t'\right\|_{U^pX[T]}
		\;.\]
	\end{definition}

	Observe that, if $0<T_1<T_2$ then the restriction map
	\[
		u \mapsto \mathbbm{1}_{[0,T_1)}(t)u
	\]
	is continuous linear $U^pX[T_2]\rightarrow U^pX[T_1]$ and satisfies
	\[
		\left\|u\right\|_{U^pX[T_1]} := \left\|\mathbbm{1}_{[0,T_1)}(t)u\right\|_{U^pX[T_1]} \le \left\|u\right\|_{U^pX[T_2]}
	\;.\]
	
	\begin{definition}
		Let $T>0$ and let $X$ be a separable Banach space over $\mathbb{C}$. Let $p\in[1,\infty)$.
		We define $V^pX[T]$ to be the Banach space of functions $v:[0,T)\rightarrow X$ with the norm
		\[
			\left\|v\right\|_{V^pX[T]} := \sup_{\mathfrak{t}} \left(\sum_{k=0}^{K-1} \left\|v(t_{k+1})-v(t_k)\right\|_X^p\right)^{\frac{1}{p}}
		\]
		where the supremum is taken over all partitions $\mathfrak{t}=\{t_k\}_{k=0}^K$ with $0=t_0<t_1<\ldots<t_K=T$,
		and we define $v(T):=0$.

		Observe that a $V^pX[T]$ function possesses left and right limits at every $t\in[0,T)$.
		We define $V^p_{\mathrm{rc}}X[T]$
		to be the closed subspace of $V^pX[T]$ of right-continuous functions $[0,T)\rightarrow X$.
	\end{definition}

	We will require the following two crucial properties of the $U^p$ and $V^p$ spaces.
	We refer to \cite{hadac_herr_koch_AnnIHP.2009, koch_NOTES.2014} for their proofs.
	
	\begin{lemma}[Embeddings] \label{ZLEM:UpVpEmbedding}
		Let $T>0$ and let $X$ be a separable Banach space over $\mathbb{C}$.
		Let $1\le p<q<\infty$. Then we have the continuous embeddings
		\[
			U^pX[T] \hookrightarrow V^p_{\mathrm{rc}}X[T] \hookrightarrow U^qX[T] \hookrightarrow L^\infty_t X[T]
		\]
		whose operator norms depend only on $p,q$ and not on $T$ or $X$.
	\end{lemma}

	\begin{lemma}[Duality] \label{ZLEM:UpVpDuality}
		Let $T>0$ and let $X$ be a separable Banach space over $\mathbb{C}$ such that .
		Let $p\in(1,\infty)$ and let $p':=\frac{p}{p-1}$ be the H\"{o}lder conjugate of $p$.
		Then
		\[
			\left(\mathrm{D}U^pX[T]\right)^* = V^{p'}_{\mathrm{rc}}X^\ast[T]
		\]
		in the sense that, for $f\in\mathrm{D}U^pX[T]$,
		\[
			\left\|f\right\|_{\mathrm{D}U^pX[T]}
			\le
			C(p)
			\sup\left\{
				\left|\int_0^T \langle v(t),f(t)\rangle_{X^*,X}\,\mathrm{d}t\right|
				\;\bigg|\;
				v\in V^{p'}_{\mathrm{rc}}X^\ast[T],
				\left\|v\right\|_{V^{p'}_{\mathrm{rc}}X^\ast[T]} \le 1
			\right\}
		\;.\]
	\end{lemma}

\section{The Modified Principal Operator} \label{ZSECT:TheModPrOp}

	The goal of this section is to establish the basic properties of solutions to the linear equation
	\begin{equation} \label{ZEQN:NewPrOpHomgCopy}
		\left(\partial_t-\mathrm{i}\triangle+\mathfrak{P}_{B}\right)u = 0
	\;,\end{equation}
	and then use these properties to define function spaces for constructing the iterates in the iteration scheme (\ref{ZEQN:IntroItrtnScheme}).
	The hypotheses we require on $B$ are summarised in the following definition.

	\begin{definition}
		Let $B=B_1(t,x)\mathrm{d}x_1+B_2(t,x)\mathrm{d}x_2$ be a time-dependent spatial 1-form
		defined on $[0,1)\times\mathbb{R}^2_x$.
		We say that $B$ is an {\em admissible form},
		if $B\in L^\infty_{t,x}[1]$, $\nabla B\in L^1_tL^\infty_x[1]$, and $\mathrm{div}\,B\equiv 0$.
	\end{definition}

	The first basic question is that of whether (\ref{ZEQN:NewPrOpHomgCopy}) gives a well-defined evolution
	in the generalised Sobolev spaces $H^{\mathfrak{m}}$.
	The following key Proposition will be proved in the next two subsections.

	\begin{proposition} \label{ZPROP:NewPrOpHmLWP}
		Let $B$ be an admissible form and $\mathfrak{m}$ be a Sobolev weight.
		Let $T\in(0,1]$ and $t_0\in[0,T)$.
		Then, given $u^{\mathrm{in}}\in H^{\mathfrak{m}}$,
		there exists a unique solution $u\in L^\infty_tH^{\mathfrak{m}}[T]$ to (\ref{ZEQN:NewPrOpHomgCopy}) with $u(t_0)=u^{\mathrm{in}}$.
		Moreover, this solution satisfies
		\begin{equation} \label{ZEQN:NewPrOpHmLWP_EQN01}
			\left\|u\right\|_{L^\infty_tH^{\mathfrak{m}}[T]}
			\le
			C\mathrm{e}^{C_1\|\nabla B\|_{L^1_tL^\infty_x[1]}}
			\left\|u^{\mathrm{in}}\right\|_{H^{\mathfrak{m}}}
		\end{equation}
		where $C,C_1>0$ are constants depending only on $[\mathfrak{m}]$.
	\end{proposition}
	
	\begin{remark}
		Eventually, when establishing the continuity of the solution map in Theorem \ref{ZTHM:MainThm},
		we will choose $\mathfrak{m}$ depending on the profile of the initial data.
		In Proposition \ref{ZPROP:NewPrOpHmLWP} and other results in this section,
		the fact that the various constants depend only on $[\mathfrak{m}]_\star$ and $[\mathfrak{m}]^\star$,
		and not on the finer details of $\mathfrak{m}$,
		will be crucial for the fact that the existence time in Theorem \ref{ZTHM:MainThm}
		depends only on the size of the initial data, and not on its profile.
	\end{remark}

\subsection{Uniqueness}
	We first address the issue of uniqueness.
	Of course, if $H^{\mathfrak{m}}$ is a sufficiently regular space, say if $[\mathfrak{m}]_\star\ge 2$,
	then unconditional uniqueness in $L^\infty_tH^{\mathfrak{m}}[T]$ is immediate from a simple energy argument.
	For lower regularity $H^{\mathfrak{m}}$ spaces, we have to work harder.

	\begin{lemma} \label{ZLEM:UniqLem1}
		Let $B$ be an admissible form and $\mathfrak{m}$ be a Sobolev weight.
		Then any solution $u\in L^\infty_tH^{\mathfrak{m}}[T]$ to (\ref{ZEQN:NewPrOpHomgCopy})
		satisfies the differential inequality
		\begin{equation}\label{ZEQN:UniqLem1_EQN01}
			\partial_t\left\|\mathrm{P}_\mu u(t)\right\|_{L^2_x}
			\le
			C\left\|\nabla B(t)\right\|_{L^\infty_x}
			\sum_{\lambda\;:\;\left|\log_2\left(\frac{\lambda}{\mu}\right)\right|\le 5} \left\|\mathrm{P}_\lambda u(t)\right\|_{L^2_x}
		\end{equation}
		where $C$ is a universal constant independent of $\mathfrak{m}$.
	\end{lemma}
	\begin{proof}
		Since $u$ solves (\ref{ZEQN:NewPrOpHomgCopy}), we have
		\begin{equation}\label{ZEQN:UniqLem1_EQN11}
			\left(\partial_t-\mathrm{i}\triangle+\mathfrak{P}_B\right)\mathrm{P}_\mu u
			=
			\left(\mathfrak{P}_B\mathrm{P}_\mu - \mathrm{P}_\mu\mathfrak{P}_B \right)u
		\;.\end{equation}
		Now, from the definition we have
		\begin{equation*}\begin{split}
			\left(\mathfrak{P}_B\mathrm{P}_\mu - \mathrm{P}_\mu\mathfrak{P}_B \right)u
		=\;&
			\sum_{\lambda\;:\; \left|\log_2\left(\frac{\lambda}{\mu}\right)\right|\le 5}
			\left[
				\mathrm{P}_{<2^{-5}\lambda}B_i\;\mathrm{P}_\lambda\mathrm{P}_\mu\partial_iu
				- \mathrm{P}_\mu\left(\mathrm{P}_{<2^{-5}\lambda}B_i\;\mathrm{P}_\lambda\partial_iu\right)
			\right]
		\\&
			+
			\sum_{\lambda\;:\; \left|\log_2\left(\frac{\lambda}{\mu}\right)\right|\le 5}
			\mathrm{P}_\lambda
			\left[
				\mathrm{P}_{<2^{-5}\lambda}B_i\;\mathrm{P}_\mu\partial_iu
				- \mathrm{P}_\mu\left(\mathrm{P}_{<2^{-5}\lambda}B_i\; \partial_iu\right)
			\right]
		\\=:\;&
			\mathrm{I} + \mathrm{II}
		\;.\end{split}\end{equation*}

		We claim the estimate
		\begin{equation}\label{ZEQN:UniqLem1_EQN12}
			\left\|\mathrm{I}(t)\right\|_{L^2_x}
		\lesssim
			\left\|\nabla B(t)\right\|_{L^\infty_x}
			\sum_{\lambda\;:\;\left|\log_2\left(\frac{\lambda}{\mu}\right)\right|\le 5} \left\|\mathrm{P}_\lambda u(t)\right\|_{L^2_x}
		\;.\end{equation}
		Indeed, recalling that $\mathrm{div}\,B=0$, we have
		\begin{equation*}\begin{split}
			\Big|\mathrm{P}_{<2^{-5}\lambda}&B_i\;\mathrm{P}_\lambda\mathrm{P}_\mu\partial_i u
			-
			\mathrm{P}_\mu\big(\mathrm{P}_{<2^{-5}\lambda}B_i\;\mathrm{P}_\lambda\partial_i u\big)\Big|(t,x)
		\\=\;&
			\left|\int_{\mathbb{R}^2} \widecheck{\varphi_\mu}(x-y)
			\left(\mathrm{P}_{<2^{-5}\lambda}B_i(t,x)
			- \mathrm{P}_{<2^{-5}\lambda}B_i(t,y)\right)\partial_i\mathrm{P}_\lambda u(t,y)\;\mathrm{d}y\right|
		\\=\;&
			\left|\int_{\mathbb{R}^2} \partial_i\widecheck{\varphi_\mu}(x-y) \left(\mathrm{P}_{<2^{-5}\lambda}B_i(t,x)
			- \mathrm{P}_{<2^{-5}\lambda}B_i(t,y)\right)\mathrm{P}_\lambda u(t,y)\;\mathrm{d}y\right|
		\\\lesssim\;&
			\left\|\nabla B(t)\right\|_{L^\infty_x}
			\int_{\mathbb{R}^2} \left|x-y\right|\left|\nabla\widecheck{\varphi_\mu}(x-y)\right|\left|\mathrm{P}_\lambda u(t,y)\right|\,\mathrm{d}y
		\;.\end{split}\end{equation*}
		Applying Young's convolution inequality, and noting that $\left\|x\nabla\widecheck{\varphi_\mu}\right\|_{L^1_x}$ is a constant independent of $\mu$,
		we obtain
		\[
			\left\|\mathrm{P}_{<2^{-5}\lambda}B_i(t)\;\mathrm{P}_\lambda\mathrm{P}_\mu\partial_i u(t)
			-
			\mathrm{P}_\mu\big(\mathrm{P}_{<2^{-5}\lambda}B_i(t)\;\mathrm{P}_\lambda\partial_i u(t)\big)\right\|_{L^2_x}
		\lesssim
			\left\|\nabla B(t)\right\|_{L^\infty_x}
			\left\|\mathrm{P}_\lambda u(t)\right\|_{L^2_x}
		\;.\]
		Summing up over $\lambda$ gives the desired estimate (\ref{ZEQN:UniqLem1_EQN12}).

		We can prove a similar estimate for $\mathrm{II}(t)$. Precisely, we have
		\begin{equation}\label{ZEQN:UniqLem1_EQN13}
			\left\|\mathrm{II}(t)\right\|_{L^2_x}
		\lesssim
			\left\|\nabla B(t)\right\|_{L^\infty_x}
			\sum_{\lambda\;:\;\left|\log_2\left(\frac{\lambda}{\mu}\right)\right|\le 5} \left\|\mathrm{P}_\lambda u(t)\right\|_{L^2_x}
		\;.\end{equation}
		Indeed, observe that only frequency components of $u$ near $\mu$ will make a nonzero contribution to the sum defining $\mathrm{II}(t)$.
		Therefore, we have
		\[
			\mathrm{II}(t) =
			\sum_{\rho\;:\;\left|\log_2\left(\frac{\rho}{\mu}\right)\right|\le 5} \mathrm{P}_\rho
			\left(
				\sum_{\lambda\;:\;\left|\log_2\left(\frac{\lambda}{\mu}\right)\right|\le 5}
				\left[
					\mathrm{P}_{<2^{-5}\rho}B_i\;\mathrm{P}_\lambda\mathrm{P}_\mu\partial_i u
					- \mathrm{P}_\mu\left(\mathrm{P}_{<2^{-5}\rho}B_i\;\mathrm{P}_\lambda\partial_i u\right)
				\right]
			\right)
		\;.\]
		For each $\rho$, the expression $\mathrm{P}_\rho(\cdots)$ above can be estimated in the exact same manner as our estimate of $\mathrm{I}(t)$.
		Then, since we sum only over finitely many $\rho$, we obtain (\ref{ZEQN:UniqLem1_EQN13}) as a result.

		By combining the estimates (\ref{ZEQN:UniqLem1_EQN12}), (\ref{ZEQN:UniqLem1_EQN13}),
		we obtain
		\[
			\left\|\left(\left(\mathfrak{P}_B\mathrm{P}_\mu - \mathrm{P}_\mu\mathfrak{P}_B \right)u\right)(t)\right\|_{L^2_x}
			\lesssim
			\left\|\nabla B(t)\right\|_{L^\infty_x}
			\sum_{\lambda\;:\;\left|\log_2\left(\frac{\lambda}{\mu}\right)\right|\le 5} \left\|\mathrm{P}_\lambda u(t)\right\|_{L^2_x}
		\;.\]
		Hence, multiplying (\ref{ZEQN:UniqLem1_EQN11}) by $\overline{\mathrm{P}_\mu u}$ and integrating by parts, 
		which is justified since the terms in (\ref{ZEQN:UniqLem1_EQN11}) are smooth, we obtain
		\[
			\partial_t\left\|\mathrm{P}_\mu u(t)\right\|_{L^2_x}^2
			\lesssim
			\left\|\nabla B(t)\right\|_{L^\infty_x}\left\|\mathrm{P}_\mu u(t)\right\|_{L^2_x}
			\sum_{\lambda\;:\;\left|\log_2\left(\frac{\lambda}{\mu}\right)\right|\le 5} \left\|\mathrm{P}_\lambda u(t)\right\|_{L^2_x}
		\]
		which gives (\ref{ZEQN:UniqLem1_EQN01}).
	\end{proof}

	\begin{proof}[Proof of the uniqueness statement of Proposition \ref{ZPROP:NewPrOpHmLWP}]
		By linearity, we only need to prove that any $L^\infty_tH^{\mathfrak{m}}[T]$ solution to (\ref{ZEQN:NewPrOpHomgCopy})
		with initial data $u(t_0)=0$ must necessarily be zero.

		Let $\varepsilon_0=\varepsilon_0([\mathfrak{m}])>0$ be a small constant to be chosen later.
		Choose a sufficiently large positive integer $K$ such that,
		for any interval $I\subset[0,T)$ of length $\le 2TK^{-1}$,
		we have $\|\nabla B\|_{L^1_t(I,L^\infty_x)}\le\varepsilon_0$.
		Write $[0,T)$ as the union of the $K-1$ overlapping small intervals $[kTK^{-1}, (k+2)TK^{-1})$ with $0\le k\le K-2$.
		Therefore it suffices to show,
		if $J$ is one of these small intervals and there exists $t_J\in J$ such that $u(t_J)=0$,
		then $u$ is zero on $J$.

		For $t\in J$, integrating (\ref{ZEQN:UniqLem1_EQN01}) from $t_J$ to $t$ gives
		\[
			\mathfrak{m}(\mu)\left\|\mathrm{P}_\mu u(t)\right\|_{L^2_x}
			\le
			C([\mathfrak{m}])
			\int_J \left\|\nabla B(t')\right\|_{L^\infty_x}
			\sum_{\lambda\;:\;\left|\log_2\left(\frac{\lambda}{\mu}\right)\right|\le 5} \mathfrak{m}(\lambda)\left\|\mathrm{P}_\lambda u(t')\right\|_{L^2_x}
			\,\mathrm{d}t'
		\]
		where the constant $C([\mathfrak{m}])$ comes from (\ref{ZEQN:SobWtRmk_EQN01}).
		Squaring both sides and applying Cauchy-Schwarz, we obtain
		\[
			\mathfrak{m}(\mu)^2\left\|\mathrm{P}_\mu u(t)\right\|_{L^2_x}^2
			\le
			C([\mathfrak{m}])
			\varepsilon_0
			\int_J \left\|\nabla B(t')\right\|_{L^\infty_x}
			\sum_{\lambda\;:\;\left|\log_2\left(\frac{\lambda}{\mu}\right)\right|\le 5} \mathfrak{m}(\lambda)^2\left\|\mathrm{P}_\lambda u(t')\right\|_{L^2_x}^2
			\,\mathrm{d}t'
		\;.\]
		By summing over $\mu$ and taking the supremum over $t\in J$, we deduce
		\[
			\left\|u\right\|_{L^\infty_tH^{\mathfrak{m}}[J]}^2
			\le
			C([\mathfrak{m}])
			\varepsilon_0^2
			\left\|u\right\|_{L^\infty_tH^{\mathfrak{m}}[J]}^2
		\;.\]
		Hence, if $\varepsilon_0$ were chosen small enough so that $C([\mathfrak{m}])\varepsilon_0^2<1$,
		then $\|u\|_{L^\infty_tH^{\mathfrak{m}}[J]}=0$ as required.
	\end{proof}

\subsection{Existence}
	We now turn our attention to the existence statement of Proposition \ref{ZPROP:NewPrOpHmLWP}.
	We first prove existence of solutions in the special case $H^{\mathfrak{m}}=L^2_x$.
	This is accomplished in Lemma \ref{ZLEM:NewPrOpL2Exstnce} by extracting a weak-star limit of solutions to regularised equations,
	which is possible due to the condition $\mathrm{div}\,B=0$.

	\begin{lemma} \label{ZLEM:NewPrOpL2Exstnce}
		Let $B$ be an admissible form. Let $T\in(0,1]$ and $t_0\in[0,t_0)$.
		Then, given $u^{\mathrm{in}}\in L^2_x$, there exists $u\in C_tL^2_x[T]$ solving (\ref{ZEQN:NewPrOpHomgCopy})
		such that $u$ is the unique $L^\infty_tL^2_x[T]$ weak-star limit of solutions to the regularised equations
		\begin{equation} \label{ZEQN:NewPrOpL2LWP_EQN01}
			\left\{
			\begin{aligned}
				\left(\partial_t-\mathrm{i}\triangle\right)u_\mu &= \chi_\mu(\mathrm{D}_x)\mathfrak{P}_Bu_\mu
			\;,\\
				u_\mu(t_0) &= \chi_\mu(\mathrm{D}_x)u^{\mathrm{in}}
			\end{aligned}\right.
		\end{equation}
		as $\mathfrak{D}\in\mu\to\infty$, where $\chi_\mu$ is the indicator function of the ball of radius $\mu$ in $\mathbb{R}^2$.
		Furthermore,
		\begin{equation} \label{ZEQN:NewPrOpL2LWP_EQN02}
			\left\|u(t)\right\|_{L^2_x} = \left\|u^{\mathrm{in}}\right\|_{L^2_x}
			\quad\mbox{for all } t \in [0,T)
		\;.\end{equation}
	\end{lemma}
	\begin{proof}
		The proof is a standard application of the energy method.
		The point is, since $\mathrm{div}\,B=0$, the operator $\mathfrak{P}_B$ is formally symmetric on $L^2_x$,
		and so the evolution of $(\partial_t-\mathrm{i}\triangle+\mathfrak{P}_B)$ conserves the $L^2_x$ norm.
		We provide the details for the sake of completeness.
		
		For ease of exposition, we assume that $t_0=0$ and remark that the proof below immediately generalises to any other
		initial time in $[0,1)$.
	
		For every $\mu\in\mathfrak{D}$, the right-hand side of the evolution equation in (\ref{ZEQN:NewPrOpL2LWP_EQN01})
		is continuous linear on $L^2_x$ with norm $\lesssim\mu\|B(t)\|_{L^\infty_x}$.
		Hence, (\ref{ZEQN:NewPrOpL2LWP_EQN01}) has a unique solution for given initial data $u^{\mathrm{in}}\in L^2_x$.
		This solution has compact frequency support and is thus smooth.
		Therefore, we may multiply by $\overline{u_\mu}$ and integrate by parts to obtain $\partial_t\|u_\mu\|_{L^2_x}^2=0$.
		We conclude $\|u_\mu\|_{L^\infty_tL^2_x[T]}\le\|u^{\mathrm{in}}\|_{L^2_x}$.
		
		By weak-star sequential compactness we may extract a subsequence $u_{\mu_k}\stackrel{\star}{\rightharpoonup}u$
		in $L^\infty_tL^2_x[T]$.
		Then we have
		\[
			\left\|u\right\|_{L^\infty_tL^2_x[T]}\le \left\|u^{\mathrm{in}}\right\|_{L^2_x}
		\;.\]
		In particular, by linearity, this limit is unique: If $\|u^{\mathrm{in}}\|_{L^2_x}=0$ then $u=0$.
		
		Since $\mathfrak{P}_B$ is formally symmetric, for any $v\in C_{\mathrm{b}}([0,T],H^1)$ we have
		\begin{equation*}\begin{split}
			\int_0^T \left(u(t),v(t)\right)_{L^2_x}\mathrm{d}t
		=\;&
			\lim_{k\to\infty} \int_0^T \left(u_{\mu_k}(t),v(t)\right)_{L^2_x}\mathrm{d}t
		\\=\;&
			\lim_{k\to\infty} \int_0^T\left(\mathrm{e}^{\mathrm{i}t\triangle}u^{\mathrm{in}}
				+ \int_0^t\mathrm{e}^{\mathrm{i}(t-t')\triangle}\mathfrak{P}_{B(t')}u_{\mu_k}(t')\,\mathrm{d}t' \;,\;
				\chi_{\mu_k}(\mathrm{D}_x)v(t)\right)_{L^2_x}\mathrm{d}t
		\\=\;&
			\lim_{k\to\infty}
				\left(u^{\mathrm{in}} \;,\; \chi_{\mu_k}(\mathrm{D}_x)\int_0^T \mathrm{e}^{-\mathrm{i}t\triangle}v(t)\,\mathrm{d}t\right)_{L^2_x}
		\\&
			+\lim_{k\to\infty}
				\int_0^T \left(u_{\mu_k}(t') \;,\; \mathfrak{P}_{B(t')}\int_{t'}^T
					\mathrm{e}^{-\mathrm{i}(t-t')\triangle}\chi_{\mu_k}(\mathrm{D}_x)v(t')\,\mathrm{d}t\right)_{L^2_x}\mathrm{d}t'
		\\=\;&
			\left(u^{\mathrm{in}},\int_0^T \mathrm{e}^{-\mathrm{i}t\triangle}v(t)\,\mathrm{d}t\right)_{L^2_x}
		\\&+
			\int_0^T \left(u(t'),\mathfrak{P}_{B(t')}\int_{t'}^T\mathrm{e}^{-\mathrm{i}(t-t')\triangle}v(t)\,\mathrm{d}t'\right)_{L^2_x}\mathrm{d}t'
		\\=\;&
			\int_0^T \left\langle
				\mathrm{e}^{\mathrm{i}t\triangle}u^{\mathrm{in}}
				+ \int_0^t\mathrm{e}^{\mathrm{i}(t-t')\triangle}\mathfrak{P}_{B(t')}u(t')\,\mathrm{d}t'
				\;,\; v(t)\right\rangle_{H^{-1},H^1}\mathrm{d}t
		\;.\end{split}\end{equation*}
		This verifies that
		\[
			u(t) = 
			\mathrm{e}^{\mathrm{i}t\triangle}u^{\mathrm{in}}
				+ \int_0^t\mathrm{e}^{\mathrm{i}(t-t')\triangle}\mathfrak{P}_{B(t')}u(t')\,\mathrm{d}t'
		\]
		as Bochner integrals into $H^{-1}$. In particular, $u$ solves (\ref{ZEQN:NewPrOpHomgCopy}) with initial data $u^{\mathrm{in}}$.
		
		Now, we may also solve (\ref{ZEQN:NewPrOpL2LWP_EQN01})
		backwards from any time in $[0,T)$.
		By applying the same argument above, we have
		\[
			\left\|u(0)\right\|_{L^2_x}\le \left\|u(t)\right\|_{L^2_x}
		\;.\]
		This verifies (\ref{ZEQN:NewPrOpL2LWP_EQN02}).
		
		Finally, as $\partial_tu\in L^\infty_tH^{-2}[T]$, we have $u\in C_{\mathrm{b}}H^{-2}[T]$.
		Since $\|u(t)\|_{L^2_x}$ is conserved and $L^2_x$ is a uniformly convex space, we deduce that $u\in C_{\mathrm{b}}L^2_x[T]$.
	\end{proof}
	
	We must now upgrade our $L^2_x$ existence result to other $H^{\mathfrak{m}}$ spaces.
	It is natural to split the given initial data into its frequency components $\mathrm{P}_\nu u^{\mathrm{in}}$ and
	solve (\ref{ZEQN:NewPrOpHomgCopy}) to get an $L^2_x$ solution $u_\nu$ with initial data $u_\nu(t_0)=\mathrm{P}_\nu u^{\mathrm{in}}$ for each $\nu$.
	Then, by linearity, an obvious candidate for the solution with initial data $u^{\mathrm{in}}$ is $u=\sum_\nu u_\nu$.
	However, since the evolution of (\ref{ZEQN:NewPrOpHomgCopy}) does not preserve the frequency support,
	it is not immediately obvious that the sum $\sum_\nu u_\nu$ converges in $L^\infty_tH^{\mathfrak{m}}[T]$.

	The fact that $B$ is an admissible form will be sufficient to guarantee this convergence.
	The key idea is that initial data, localised about a frequency scale $\nu$, will launch a solution
	which, within a fixed time interval, transfers only a very small amount of mass to frequency scales vastly different from $\nu$.
	The following lemma contains the precise, quantitative formulation of this idea.

	\begin{lemma} \label{ZLEM:NewPrOpMassXfer}
		Let $B$ be any admissible form.
		Let $T\in(0,1]$ and let $t_0\in[0,T)$.
		Let $\nu\in\mathfrak{D}$
		and let $v$ be a solution on $[0,T)$ to (\ref{ZEQN:NewPrOpHomgCopy}),
		whose initial data $v(t_0)\in L^2_x$ is frequency supported in $\{\frac{1}{2}\nu\le|\xi|\le 2\nu\}$.
		Then for $\ell\in\mathbb{Z}_{\ge 0}$, we have
		\begin{equation}\label{ZEQN:NewPrOpMassXfer_EQN01}
			\left\|\mathrm{P}_\mu v\right\|_{L^\infty_tL^2_x[T]}
			\le
				\frac{\left(C_0\|\nabla B\|_{L^1_tL^\infty_x[1]}\right)^{\ell}}{\ell!} \left\|v(t_0)\right\|_{L^2_x}
			\quad\mbox{whenever }
			\left|\log_2\left(\frac{\mu}{\nu}\right)\right|\ge 5\ell
		\;.\end{equation}
		Here $C_0>0$ is a universal constant independent of $T, \nu$ or $\ell$.
	\end{lemma}
	\begin{proof}
		For ease of exposition, we shall assume $t_0=0$ and remark that the proof for general $t_0$ is similar.
		Put $C_0:=20C$ where $C$ is the constant appearing in (\ref{ZEQN:UniqLem1_EQN01}).
		It suffices to prove the stronger estimate
		\begin{equation}\label{ZEQN:NewPrOpMassXfer_EQN11}\begin{split}
			\left\|\mathrm{P}_\mu v(t)\right\|_{L^2_x}
		\le
			C_0^\ell \int_0^t\int_0^{t_\ell}\cdots\int_0^{t_2}
				&\prod_{m=1}^\ell \left\|\nabla B(t_m)\right\|_{L^\infty_x}
				\mathrm{d}t_1\cdots\mathrm{d}t_\ell\,
				\left\|v(0)\right\|_{L^2_x}
		\\&\mbox{whenever }
			\left|\log_2\left(\frac{\mu}{\nu}\right)\right|\ge 5\ell
		\;,\end{split}\end{equation}
		for $\ell\in\mathbb{Z}_{\ge 0}$, where, when $\ell=0$, the integral is defined to be $1$.
		
		We establish (\ref{ZEQN:NewPrOpMassXfer_EQN11}) by induction on $\ell$.
		The conservation of $L^2_x$ norm, from Lemma \ref{ZLEM:NewPrOpL2Exstnce}, gives the base case $\ell=0$.
		For $\ell\ge 1$, plugging the induction hypothesis for $\ell-1$ into every summand on
		the right-hand side of (\ref{ZEQN:UniqLem1_EQN01}), we obtain
		\begin{equation}\label{ZEQN:NewPrOpMassXfer_EQN12}
			\partial_t\left\|\mathrm{P}_\mu v(t)\right\|_{L^2_x}
		\le
			C_0
			\left\|\nabla B(t)\right\|_{L^\infty_x}
			C_0^{\ell-1}\int_0^{t}\cdots\int_0^{t_2}
				\prod_{m=1}^{\ell-1}\left\|\nabla B(t_m)\right\|_{L^\infty_x}\mathrm{d}t_1\cdots\mathrm{d}t_{\ell-1}
			\left\|v(0)\right\|_{L^2_x}
		\;.\end{equation}
		Since $|\log_2(\mu/\nu)|\ge 5\ell\ge 5$, we have by definition that $\mathrm{P}_\mu v(0)=0$.
		Therefore, a direct integration of (\ref{ZEQN:NewPrOpMassXfer_EQN12}) yields
		\[
			\left\|\mathrm{P}_\mu v(t)\right\|_{L^2_x}
			\le
			C_0^\ell
			\int_0^t
			\int_0^{t_{\ell}}\cdots\int_0^{t_2}
				\prod_{m=1}^\ell\left\|\nabla B(t_m)\right\|_{L^\infty_x}\mathrm{d}t_1\cdots\mathrm{d}t_{\ell}
			\left\|v(0)\right\|_{L^2_x}
		\]
		which completes the induction step.
	\end{proof}
	
	\begin{proof}[Proof of the existence statement of Proposition \ref{ZPROP:NewPrOpHmLWP} and of (\ref{ZEQN:NewPrOpHmLWP_EQN01})]
		Let $w\in H^{\mathfrak{m}}$ be given. Let $u_\nu$
		be the solution of (\ref{ZEQN:NewPrOpHomgCopy}) with initial data $u_\nu(t_0)=\mathrm{P}_\nu w$.
		To complete the proof of Proposition \ref{ZPROP:NewPrOpHmLWP}, it suffices to prove
		\begin{equation}\label{ZEQN:NewPrOpHmLWP_EQN11}
			\left\|\left\{\mathfrak{m}(\mu)\left\|\mathrm{P}_\mu u_\nu(t)\right\|_{L^2_x}\right\}_{\mu,\nu}\right\|_{\ell^2_\mu\ell^1_\nu}
			\le
			C\mathrm{e}^{C_1\|\nabla B\|_{L^1_tL^\infty_x[1]}}\left\|w\right\|_{H^{\mathfrak{m}}}
		\end{equation}
		for $C,C_1$ as in the statement of Proposition \ref{ZPROP:NewPrOpHmLWP}.
		Indeed,
		\[
			\sum_\mu \mathfrak{m}(\mu)^2\left\|\mathrm{P}_\mu\sum_\nu u_\nu(t)\right\|_{L^2_x}^2
			\le
			\left\|\left\{\mathfrak{m}(\mu)\left\|\mathrm{P}_\mu u_\nu(t)\right\|_{L^2_x}\right\}_{\mu,\nu}\right\|_{\ell^2_\mu\ell^1_\nu}^2
		\]
		which shows that the desired solution $u=\sum_\nu u_\nu$ belongs to $L^\infty_tH^{\mathfrak{m}}[T]$
		and satisfies the claimed estimate (\ref{ZEQN:NewPrOpHmLWP_EQN01}).
		
		Now, recall that from the definitions, we have
		\[
			\mathfrak{m}(\mu) \le 2^{5(\ell+1)[\mathfrak{m}]}\mathfrak{m}(\nu)
			\quad\mbox{whenever }
			5\ell\le\left|\log_2\left(\frac{\mu}{\nu}\right)\right|< 5(\ell+1)
		\;.\]
		Therefore, using Lemma \ref{ZLEM:NewPrOpMassXfer}, we have
		\[
			\mathfrak{m}(\mu)\sum_\nu\left\|\mathrm{P}_\mu u_\nu(t)\right\|_{L^2_x}
			\le
			2^{5[\mathfrak{m}]}\sum_{\ell=0}^\infty \sum_{\nu\;:\; 5\ell\le\left|\log_2\left(\frac{\nu}{\mu}\right)\right|< 5(\ell+1)}
				\frac{\left(C_0\left\|\nabla B\right\|_{L^1_tL^\infty_x[1]}\right)^\ell}{\ell!}
				2^{5\ell [\mathfrak{m}]}\mathfrak{m}(\nu)\left\|\mathrm{P}_\nu w\right\|_{L^2_x}
		\;.\]
		We set $C_1 = C_1([\mathfrak{m}]) := C_02^{5[\mathfrak{m}]}$ once and for all.
		Then, by Cauchy-Schwarz,
		\begin{equation*}\begin{split}
			\bigg(\mathfrak{m}(\mu)&\sum_\nu\left\|\mathrm{P}_\mu u_\nu(t)\right\|_{L^2_x}\bigg)^2
		\\\le\;&
			C([\mathfrak{m}])\mathrm{e}^{C_1\|\nabla B\|_{L^1_tL^\infty_x[1]}}
			\sum_{\ell=0}^\infty
			\left(
				\sum_{\nu\;:\; 5\ell\le\left|\log_2\left(\frac{\nu}{\mu}\right)\right|< 5(\ell+1)}
				\frac{\left(C_1\left\|\nabla B\right\|_{L^1_tL^\infty_x[1]}\right)^\ell}{\ell!}
				\mathfrak{m}(\nu)^2\left\|\mathrm{P}_\nu w\right\|_{L^2_x}^2
			\right)
		\;.\end{split}\end{equation*}
		Summing over $\mu$ then gives (\ref{ZEQN:NewPrOpHmLWP_EQN11}).
	\end{proof}
	
\subsection{Strichartz estimates}
	Having proved Proposition \ref{ZPROP:NewPrOpHmLWP} in the preceding two sections,
	we now show that the corresponding solutions enjoy local-in-time Strichartz estimates with loss of derivatives.
	
	\begin{proposition} \label{ZPROP:NewPrOpStrichartz}
		Let $B$ be an admissible form and $\mathfrak{m}$ be a Sobolev weight.
		Let $T\in(0,1]$ and $t_0\in[0,T)$, and let $w\in H^{\mathfrak{m}}$.
		Let $u$ be the solution to (\ref{ZEQN:NewPrOpHomgCopy}) with initial data $u(t_0)=w$.
		Let $(q,r)$ be a Strichartz pair.
		Then the estimate
		\begin{equation} \label{ZEQN:NewPrOpStrichartz_EQN01}
			\left\|\mathrm{P}_\mu u\right\|_{L^q_tL^r_x[T]}
			\le
			C\left(1+\left\|B\right\|_{L^\infty_tL^\infty_x[1]}\right)
			\mathrm{e}^{C_1\|\nabla B\|_{L^1_tL^\infty_x[1]}}
			\mu^{\frac{1}{q}}
			\mathfrak{m}(\mu)^{-1}
			\left\|u\right\|_{H^{\mathfrak{m}}}
		\end{equation}
		holds for some constants $C=C([\mathfrak{m}],q)>0$ and $C_1=C_1([\mathfrak{m}])>0$.
	\end{proposition}
	\begin{proof}
		Following the strategy of \cite{burq_gerard_tzvetkov_AmerJMath.2004},
		we divide $[0,T)$ into disjoint intervals each of length $\le \mu^{-1}$,
		so that there are $\le \mu$ such intervals.
		Consider one such interval $J=[t_1,t_2)$.
		Applying the usual Strichartz estimate to
		\[
			\left(\partial_t-\mathrm{i}\triangle\right)\mathrm{P}_\mu u = -\mathrm{P}_\mu\mathfrak{P}_B u
		\]
		over the interval $J$, we obtain
		\begin{equation*}\begin{split}
			\left\|\mathrm{P}_\mu u\right\|_{L^q_t(J,L^r_x)}
		\lesssim\;&
			\left\|\mathrm{P}_\mu u(t_1)\right\|_{L^2_x}
			+
			|J|\mu\left\|\mathrm{P}_\mu\mathfrak{P}_B u\right\|_{L^\infty_t(J,L^2_x)}
		\\\lesssim\;&
			\left\|\mathrm{P}_\mu u(t_1)\right\|_{L^2_x}
			+
			\left\|B\right\|_{L^\infty_{t,x}[1]}
			\sum_{\lambda\;:\;\left|\log_2\left(\frac{\lambda}{\mu}\right)\right|\le 5} \left\|\mathrm{P}_\lambda u\right\|_{L^\infty_t(J,L^2_x)}
		\;.\end{split}\end{equation*}
		Using (\ref{ZEQN:NewPrOpHmLWP_EQN01}) to bound the right-hand side, we obtain
		\begin{equation}\label{ZEQN:NewPrOpStrichartz_EQN11}
			\left\|\mathrm{P}_\mu u\right\|_{L^q_t(J,L^r_x)}
			\le
			C\left([\mathfrak{m}],q\right)
			\left(1+\left\|B\right\|_{L^\infty_{t,x}[1]}\right)
			\mathrm{e}^{C_1([\mathfrak{m}])\|\nabla B\|_{L^1_tL^\infty_x[1]}}
			\mathfrak{m}(\mu)^{-1}
			\left\|w\right\|_{H^{\mathfrak{m}}}
		\;.\end{equation}
		Note that the right-hand side of (\ref{ZEQN:NewPrOpStrichartz_EQN11}) is now independent of $J$.
		Hence, raising (\ref{ZEQN:NewPrOpStrichartz_EQN11})
		to the $q$-th power and summing over the intervals $J$,
		and recalling that there are $\le\mu$ such intervals,
		we obtain (\ref{ZEQN:NewPrOpStrichartz_EQN01}).
	\end{proof}
	
\subsection{Adapted function spaces}
	Having now established the basic properties of solutions to the linear homogeneous equation (\ref{ZEQN:NewPrOpHomgCopy}),
	we define the function spaces which we will use to construct the iteration scheme (\ref{ZEQN:ItrtnSchemeSuccinct}).
	
	\begin{notation}
		Let $B$ be an admissible form and $\mathfrak{m}$ be a Sobolev weight.
		For $t,t_0\in[0,1)$, denote
		\[
			\mathfrak{S}_B(t,t_0)w := U(t)
		\]
		where $U$ solves (\ref{ZEQN:NewPrOpHomgCopy}) on $[0,1)$ with initial data $U(t_0)=w\in H^{\mathfrak{m}}$.
	\end{notation}
	
	\begin{definition}
		Let $B$ be an admissible form and $\mathfrak{m}$ be a Sobolev weight.
		Let $T\in(0,1]$.
		
		Let $p\in[1,\infty)$. Define $U^p_BH^{\mathfrak{m}}[T]$
		to be the Banach space of functions $u:[0,T)\rightarrow H^{\mathfrak{m}}$ such that
		$\mathfrak{S}_B(0,t)u(t)$ belongs to $U^pH^{\mathfrak{m}}[T]$.
		The $U^p_BH^{\mathfrak{m}}[T]$ norm is given by
		\[
			\left\|u\right\|_{U^p_BH^{\mathfrak{m}}[T]} := \left\|\mathfrak{S}_B(0,t)u(t)\right\|_{U^pH^{\mathfrak{m}}[T]}
		\;.\]
		Define $\mathrm{D}U^p_BH^{\mathfrak{m}}[T]$ to consist of functions $f:[0,T)\times\mathbb{R}^2_x\rightarrow\mathbb{C}$
		such that $\mathfrak{S}_B(0,t)f(t)\in \mathrm{D}U^pH^{\mathfrak{m}}[T]$, equipped with the norm
		\[
			\left\|f\right\|_{\mathrm{D}U^p_BH^{\mathfrak{m}}[T]} := \left\|\mathfrak{S}_B(0,t)f(t)\right\|_{\mathrm{D}U^pH^{\mathfrak{m}}[T]}
			= \left\|\int_0^t\mathfrak{S}_B(0,t')f(t')\,\mathrm{d}t'\right\|_{U^pH^{\mathfrak{m}}[T]}
		\;.\]
		Lastly, define $V^p_BH^{\mathfrak{m}}[T]$ to be the Banach space of functions $v:[0,T)\rightarrow H^{\mathfrak{m}}$
		such that
		$\mathfrak{S}_B(0,t)v(t)$ belongs to $V^p_{\mathrm{rc}}H^{\mathfrak{m}}[T]$, equipped with the norm
		\[
			\left\|v\right\|_{V^p_BH^{\mathfrak{m}}[T]} := \left\|\mathfrak{S}_B(0,t)v(t)\right\|_{V^pH^{\mathfrak{m}}[T]}
		\;.\]
	\end{definition}
	
	As a first consequence of the definitions, of the uniqueness statement in Proposition \ref{ZPROP:NewPrOpHmLWP},
	and of Duhamel's formula, we have the following result.
	
	\begin{lemma} \label{ZLEM:DuhamelBd}
		Let $B$ be an admissible form and $\mathfrak{m},\mathfrak{n}$ be Sobolev weights with
		$\mathfrak{n}\le\mathfrak{m}$, so that $H^{\mathfrak{m}}\hookrightarrow H^{\mathfrak{n}}$.
		Let $T\in(0,1]$ and $p\in[1,\infty)$.
		
		Suppose $u\in L^\infty_tH^{\mathfrak{n}}[T]$ and $u(0)=u^{\mathrm{in}}\in H^{\mathfrak{m}}$ and
		\[
			\left(\partial_t-\mathrm{i}\triangle+\mathfrak{P}_B\right)u = f
		\]
		with $f\in\mathrm{D}U^p_BH^{\mathfrak{m}}[T]$.
		Then, in fact, $u$ must be given by
		\begin{equation} \label{ZEQN:DuhamelBd_EQN00}
			u(t) = \mathfrak{S}_B(t,0)u^{\mathrm{in}} + \int_0^t \mathfrak{S}_B(t,t')f(t')\,\mathrm{d}t'
		\;,\end{equation}
		and in particular, $u\in U^p_BH^{\mathfrak{m}}[T]$ and
		\begin{equation} \label{ZEQN:DuhamelBd_EQN01}
			\left\|u\right\|_{U^p_BH^{\mathfrak{m}}[T]}
			\lesssim
			\left\|u^{\mathrm{in}}\right\|_{H^{\mathfrak{m}}}
			+
			\left\|f\right\|_{\mathrm{D}U^p_BH^{\mathfrak{m}}[T]}
		\;.\end{equation}
	\end{lemma}
	\begin{proof}
		Let $v$ be given by the right-hand isde of (\ref{ZEQN:DuhamelBd_EQN00}).
		Clearly, $v\in U^p_BH^{\mathfrak{m}}[T]$ and satisfies
		\[
			\left\|v\right\|_{U^p_BH^{\mathfrak{m}}[T]}
			\lesssim
			\left\|u^{\mathrm{in}}\right\|_{H^{\mathfrak{m}}}
			+
			\left\|f\right\|_{\mathrm{D}U^p_BH^{\mathfrak{m}}[T]}
		\;.\]
		Now, by Proposition \ref{ZPROP:NewPrOpHmLWP} and the atomic structure of $U^p_BH^{\mathfrak{m}}[T]$,
		we have $U^p_BH^{\mathfrak{m}}[T]\hookrightarrow L^\infty_tH^{\mathfrak{m}}[T]$.
		Thus, $u-v\in L^\infty_tH^{\mathfrak{n}}[T]$.
		But $u-v$ is a solution to (\ref{ZEQN:NewPrOpHomgCopy}) with $(u-v)(0)=0$.
		Hence, by the uniqueness statement in Proposition \ref{ZPROP:NewPrOpHmLWP},
		we have $u-v=0$.
	\end{proof}
	
	Observe that Lemma \ref{ZLEM:UpVpEmbedding} generalises immediately to the above function spaces.
	More precisely, we have the following embedding result.
	
	\begin{lemma}[Embeddings] \label{ZLEM:UpVpEmbeddingB}
		Let $B$ be an admissible form and $\mathfrak{m}$ be a Sobolev weight.
		Let $T\in(0,1]$ and let $1\le p<q<\infty$. Then we have the continuous embeddings
		\[
			U^p_BH^{\mathfrak{m}}[T] \hookrightarrow V^p_BH^{\mathfrak{m}}[T] \hookrightarrow U^q_BH^{\mathfrak{m}}[T]
		\]
		whose operator norms depend on $p,q$ and not on $T$ or $B$ or $\mathfrak{m}$.
	\end{lemma}

	To use the Duhamel formula in Lemma \ref{ZLEM:DuhamelBd},
	we will need to estimate the $\mathrm{D}U^p_BH^{\mathfrak{m}}[T]$ norm of the various nonlinearities we encounter.
	Such estimates can be efficiently obtained
	using the following duality result,
	which is the obvious generalisation of Lemma \ref{ZLEM:UpVpDuality}.
	
	\begin{lemma}[Duality] \label{ZLEM:UpVpDualityB}
		Let $B$ be an admissible form and $\mathfrak{m}$ be a Sobolev weight.
		Let $T\in (0,1]$ and $p\in(1,\infty)$, and let $p':=\frac{p}{p-1}$.
		Then
		\[
			\left(\mathrm{D}U^p_BH^{m}[T]\right)^\ast
			=
			V^{p'}_BH^{\mathfrak{m}^{-1}}[T]
		\]
		in the sense that
		\[
			\left\|f\right\|_{\mathrm{D}U^pH^{\mathfrak{m}}[T]}
			\le
			C\left(p,[\mathfrak{m}]\right)
			\sup_v\left|\int_0^T \int_{\mathbb{R}^2_x} \overline{v(t,x)}\,f(t,x)\,\mathrm{d}x\right|
		\]
		where the supremum is taken over all $v\in V^{p'}_BH^{\mathfrak{m}^{-1}}[T]$ with
		$\|v\|_{V^{p'}_BH^{\mathfrak{m}^{-1}}[T]}\le 1$.
	\end{lemma}
	\begin{proof}
		From Lemma \ref{ZLEM:NewPrOpL2Exstnce} we have that $\mathfrak{S}_B(t_1,t_0)$ are unitary maps on $L^2_x$.
		Moreover, Lemma \ref{ZLEM:HmDuality} guarantees that $H^{\mathfrak{m}^{-1}}$ and $(H^{\mathfrak{m}})^\ast$
		are isomorphic with equivalent norms.
		Hence, Lemma \ref{ZLEM:UpVpDualityB} follows immediately from Lemma \ref{ZLEM:UpVpDuality}.
	\end{proof}
	
	Our next Lemma shows that generalises the energy and Strichartz estimates,
	established earlier for free solutions to (\ref{ZEQN:NewPrOpHomgCopy}),
	to arbitrary $U^p_BH^{\mathfrak{m}}[T]$ functions.
	
	\begin{lemma} \label{ZLEM:OldLemma4.9}
		Let $B$ be an admissible form and $\mathfrak{m}$ be a Sobolev weight.
		Let $T\in (0,1]$ and $p\in[1,\infty)$, and let $(q,r)$ be a Strichartz pair.
		Then we have the estimates
		\begin{equation}\label{ZEQN:OldLemma4.9_EQN01}
			\left\|u\right\|_{L^\infty_tH^{\mathfrak{m}}[T]}
			\le
			C\left([\mathfrak{m}]\right)\mathrm{e}^{C_1([\mathfrak{m}])\|\nabla B\|_{L^1_tL^\infty_x[1]}}
			\left\|u\right\|_{U^p_BH^{\mathfrak{m}}[T]}
		\end{equation}
		and
		\begin{equation}\label{ZEQN:OldLemma4.9_EQN02}
			\left\|\mathrm{P}_\mu u\right\|_{L^q_tL^r_x[T]}
			\le
			C\left([\mathfrak{m}],q\right)
			\left(1+\left\|B\right\|_{L^\infty_{t,x}[1]}\right)
			\mathrm{e}^{C_1([\mathfrak{m}])\|\nabla B\|_{L^1_tL^\infty_x[1]}}
			\mu^{\frac{1}{q}}\mathfrak{m}(\mu)^{-1}
			\left\|u\right\|_{U^q_BH^{\mathfrak{m}}[T]}
		\end{equation}
	\end{lemma}
	\begin{proof}
		Due to the atomic structure of the $U^p_BH^{\mathfrak{m}}[T]$ spaces,
		the asserted estimates are immediate consequences of Propositions \ref{ZPROP:NewPrOpHmLWP}
		and \ref{ZPROP:NewPrOpStrichartz}.
	\end{proof}
	
	With the above machinery, the following result, which lets us compare $U^p$ spaces associated to different admissible forms,
	is now straightforward.
	
	\begin{proposition} \label{ZPROP:ChangeOfAdmForm}
		Let $\mathfrak{m}$ be a Sobolev weight and $B,\varGamma$ be admissible forms.
		Let $T\in(0,1]$. Let $p\in(1,\infty)$.
		Then we have the embedding $U^p_BH^{\mathfrak{m}}[T]\hookrightarrow U^p_\varGamma H^{\lambda^{-1}\mathfrak{m}}[T]$ with
		\[
			\left\|u\right\|_{U^p_\varGamma H^{\lambda^{-1}\mathfrak{m}}[T]}
			\le
			C([\mathfrak{m}])\mathrm{e}^{C_1([\mathfrak{m}])\left(\|\nabla B\|_{L^1_tL^\infty_x[1]}+\|\nabla\varGamma\|_{L^1_tL^\infty_x[1]}\right)}
			T
			\left\|B-\varGamma\right\|_{L^\infty_{t,x}[T]}
			\left\|u\right\|_{U^p_BH^{\mathfrak{m}}[T]}
		\;.\]
	\end{proposition}
	\begin{proof}
		Suppose first that $u=\mathfrak{S}_B(t,t_0)w$ is a free solution on $[0,T)$ to (\ref{ZEQN:NewPrOpHomgCopy})
		with $w\in H^{\mathfrak{m}}$, so that
		\begin{equation}\label{ZEQN:ChangeOfAdmForm_EQN01}
			\left(\partial_t-\mathrm{i}\triangle+\mathfrak{P}_\varGamma\right)u = \mathfrak{P}_{\varGamma-B}u
		\;.\end{equation}
		Now, observe that $\|\mathfrak{P}_bw\|_{H^{\lambda^{-1}\mathfrak{m}}[T]}\le C([\mathfrak{m}]) \|b\|_{L^\infty_x}\|w\|_{H^{\mathfrak{m}}[T]}$.
		Therefore, for $v\in V^{p'}_\varGamma H^{\lambda\mathfrak{m}^{-1}}[T]$, we have
		\[
			\left|\int_0^T\int_{\mathbb{R}^2_x}
				\overline{v}\,\mathfrak{P}_{\varGamma-B}u\,\mathrm{d}x\,\mathrm{d}t\right|
			\le
				C([\mathfrak{m}])
				T
				\left\|v\right\|_{L^\infty_tH^{\lambda\mathfrak{m}^{-1}}[T]}
				\left\|B-\varGamma\right\|_{L^\infty_{t,x}[T]}
				\left\|u\right\|_{L^\infty_tH^{\mathfrak{m}}[T]}
		\;.\]
		Using Lemma \ref{ZLEM:OldLemma4.9}, we obtain
		\begin{equation*}\begin{split}
			\left|\int_0^T\int_{\mathbb{R}^2_x}
				\overline{v}\,\mathfrak{P}_{\varGamma-B}u\,\mathrm{d}x\,\mathrm{d}t\right|
			\le\;&
				C([\mathfrak{m}])
				\mathrm{e}^{C_1([\mathfrak{m}])\left(\|\nabla B\|_{L^1_tL^\infty_x[1]}+\|\nabla\varGamma\|_{L^1_tL^\infty_x[1]}\right)}
		\\&
			\cdot
				T
				\left\|v\right\|_{V^{p'}_\varGamma H^{\lambda\mathfrak{m}^{-1}}[T]}
				\left\|B-\varGamma\right\|_{L^\infty_{t,x}[1]}
				\left\|w\right\|_{H^{\mathfrak{m}}}
		\;.\end{split}\end{equation*}
		Thus, by the duality principle of Lemma \ref{ZLEM:UpVpDualityB},
		\[
			\left\|\mathfrak{P}_{\varGamma-B}u\right\|_{\mathrm{D}U^p_BH^{\lambda^{-1}\mathfrak{m}}[T]}
			\le
				C([\mathfrak{m}])
				\mathrm{e}^{C_1([\mathfrak{m}])\left(\|\nabla B\|_{L^1_tL^\infty_x[1]}+\|\nabla\varGamma\|_{L^1_tL^\infty_x[1]}\right)}
				T
				\left\|B-\varGamma\right\|_{L^\infty_{t,x}[1]}
				\left\|w\right\|_{H^{\mathfrak{m}}}
		\;.\]
		Plugging into the Duhamel formula in Lemma \ref{ZLEM:DuhamelBd}, we find
		\[
			\left\|u\right\|_{U^p_BH^{\lambda^{-1}\mathfrak{m}}[T]}
			\le
				C([\mathfrak{m}])
				\mathrm{e}^{C_1([\mathfrak{m}])\left(\|\nabla B\|_{L^1_tL^\infty_x[1]}+\|\nabla\varGamma\|_{L^1_tL^\infty_x[1]}\right)}
				T
				\left\|B-\varGamma\right\|_{L^\infty_{t,x}[1]}
				\left\|w\right\|_{H^{\mathfrak{m}}}
		\;.\]
		This proves Proposition \ref{ZPROP:ChangeOfAdmForm} in the special case when $u$ is a free solution to (\ref{ZEQN:NewPrOpHomgCopy}).
		
		In particular, if now $u$ is a $U^p_BH^{\mathfrak{m}}[T]$ atom, then
		\[
			\left\|u\right\|_{U^p_BH^{\lambda^{-1}\mathfrak{m}}[T]}
			\le
				C([\mathfrak{m}])
				\mathrm{e}^{C_1([\mathfrak{m}])\left(\|\nabla B\|_{L^1_tL^\infty_x[1]}+\|\nabla\varGamma\|_{L^1_tL^\infty_x[1]}\right)}
				T
				\left\|B-\varGamma\right\|_{L^\infty_{t,x}[1]}
		\;.\]
		The assertion of Proposition \ref{ZPROP:ChangeOfAdmForm} now follow from the atomic structure of $U^p_BH^{\mathfrak{m}}[T]$.
	\end{proof}
	
\section{Construction of the Iteration Scheme} \label{ZSECT:ConstructItrtnScheme}

	The goal of the present section is to set up the iteration scheme (\ref{ZEQN:IntroItrtnScheme}), and show that the iterates $\phi^{[n]}$
	exist on a common time interval $T=T(\|\phi^{\mathrm{in}}\|_{H^s})$.
	The convergence of the iteration scheme to a solution of the Chern-Simons-Schr\"odinger system in the Coulomb gauge, (\ref{ZEQN:CSSCoul}), will be addressed in the next section.

\subsection{Setting up the iteration scheme}
	For convenience, we introduce the following notation.
	Define the bilinear maps $\mathcal{N}^2_\alpha$ for $\alpha=0,1,2$,
	and the quadrilinear maps $\mathcal{N}^4_t, \mathcal{N}^4_x$ on $\mathcal{S}(\mathbb{R}^2_x)$, by
	\begin{equation*}\begin{split}
		\mathcal{N}^2_i[u_1,u_2] :=&\; \epsilon_{ij}\frac{\partial_j}{(-\triangle)}\left(u_1u_2\right)
	\;,\\
		\mathcal{N}^2_0[u_1,u_2] :=&\; \left(-\triangle\right)^{-1}\left(\nabla u_1\wedge \nabla u_2\right)
	\;,\\
		\mathcal{N}^4_t[u_1,u_2,u_3,u_4] :=&\; \frac{\mathrm{rot}}{(-\triangle)}\left(\mathcal{N}^2_x[u_1,u_2]u_3u_4\right)
	\;,\\
		\mathcal{N}^4_x[u_1,u_2,u_3,u_4] :=&\; -\mathrm{i}\mathcal{N}^2_x[u_1,u_2]\cdot\mathcal{N}^2_x[u_3,u_4]
	\;,\end{split}\end{equation*}
	where we have also denoted $\mathcal{N}^2_x:=(\mathcal{N}^2_1,\mathcal{N}^2_2)$.
	We warn the reader that $\mathcal{N}^2_x$ is $\mathbb{R}^2$-valued while $\mathcal{N}^4_x$ is real-valued.
	Define also the trilinear map
	\[
		\mathcal{Q}\left[u_1,u_2,u_3\right] 
		:=
		\sum_\lambda \left[\mathrm{P}_\lambda\mathcal{N}^2_x\left[u_1,u_2\right]\cdot\nabla\mathrm{P}_{<2^5\lambda}u_3
			+ \mathrm{P}_{<2^5\lambda}\left(\mathrm{P}_\lambda\mathcal{N}^2_x\left[u_1,u_2\right]\cdot\nabla u_3\right)\right]
	\;,\]
	so that, in the notation of the Introduction,
	$\mathcal{Q}[\overline{\phi},\phi,\phi] = \mathfrak{Q}_{A_x}\phi$ for a solution $\phi$ to (\ref{ZEQN:CSSCoul}).

	Then the Chern-Simons-Schr\"odinger system in the Coulomb gauge, (\ref{ZEQN:CSSCoul}), can be written as
	\begin{equation}\label{ZEQN:CSSCoulSuccinct2}
	\left\{
		\begin{aligned}
			\left(\partial_t-\mathrm{i}\triangle+\mathfrak{P}_{A_x}\right)\phi
		=\;&
			\mathcal{Q}\left[\overline{\phi},\phi,\phi\right]
			+ \mathcal{N}^2_0\left[\overline{\phi},\phi\right]\phi
			+ \mathcal{N}^4_t\left[\overline{\phi},\phi,\overline{\phi},\phi\right]\phi
		\\&
			+ \mathcal{N}^4_x\left[\overline{\phi},\phi,\overline{\phi},\phi\right]\phi
			- \mathrm{i}\kappa\left|\phi\right|^2\phi
		\;,\\
			A_x =\;& -\mbox{$\frac{1}{2}$}\mathcal{N}^2_x\left[\overline{\phi},\phi\right]
		\;.
		\end{aligned}
	\right.
	\end{equation}
	Similarly, the iteration scheme (\ref{ZEQN:IntroItrtnScheme}) can be written succinctly as
	\begin{equation}\label{ZEQN:ItrtnSchemeSuccinct}
	\left\{
		\begin{aligned}
			\left(\partial_t-\mathrm{i}\triangle+\mathfrak{P}_{A_x^{[n-1]}}\right)\phi^{[n]}
		=\;&
			\mathcal{Q}\left[\overline{\phi^{[n]}},\phi^{[n]},\phi^{[n]}\right]
			+ \mathcal{N}^2_0\left[\overline{\phi^{[n]}},\phi^{[n]}\right]\phi^{[n]}
		\\&
			+ \mathcal{N}^4_t\left[\overline{\phi^{[n]}},\phi^{[n]},\overline{\phi^{[n]}},\phi^{[n]}\right]\phi^{[n]}
		\\&
			+ \mathcal{N}^4_x\left[\overline{\phi^{[n]}},\phi^{[n]},\overline{\phi^{[n]}},\phi^{[n]}\right]\phi^{[n]}
			- \mathrm{i}\kappa\left|\phi^{[n]}\right|^2\phi^{[n]}
		\;,\\
			A_x^{[n]} =\;& -\mbox{$\frac{1}{2}$}\mathcal{N}^2_x\left[\overline{\phi^{[n]}},\phi^{[n]}\right]
		\;,\\
			\phi^{[n]}(0) =\;& \phi^{\mathrm{in}}
		\;.
		\end{aligned}
	\right.
	\end{equation}

	We record the following easy estimate, which will play a key role in formulating the existence result for the iteration scheme, Theorem \ref{ZTHM:Sect5MainThm}.

	\begin{lemma} \label{ZLEM:EnergyEstmt_N2x}
		We have the estimate
		\begin{equation} \label{ZEQN:EnergyEstmt_N2x_EQN02}
			\left\|\mathcal{N}^2_x\left[u_1,u_2\right]\right\|_{L^\infty_x} \le C\left\|u_1\right\|_{H^1}\left\|u_2\right\|_{H^1}
		\;.\end{equation}
	\end{lemma}
	\begin{proof}
		By Bernstein's inequality, it suffices to prove the stronger estimate
		\begin{equation} \label{ZEQN:EnergyEstmt_N2x_EQN11}
			\left\|\mathcal{N}^2_x\left[u_1,u_2\right]\right\|_{B^1_{4,\infty}} \le C\left\|u_1\right\|_{H^1}\left\|u_2\right\|_{H^1}
		\;.\end{equation}
		By Hardy-Littlewood-Sobolev,
		\[
			\left\|\mathrm{P}_\mu\mathcal{N}^2_x\left[\mathrm{P}_\lambda u_1,\mathrm{P}_{\le\lambda}u_2\right]\right\|_{L^4_x}
			\le
			C\left\|\mathrm{P}_\lambda u_1\,\mathrm{P}_{\le\lambda}u_2\right\|_{L^{\frac{4}{3}}_x}
			\le
			C\lambda^{-1}\left\|u_1\right\|_{H^1}\left\|u_2\right\|_{H^1}
		\;.\]
		Summing up over $\lambda\gtrsim\mu$, and noting that $\mathcal{N}^2_x[u_1,u_2]$ is symmetric in $u_1,u_2$, we obtain
		\[
			\left\|\mathrm{P}_\mu\mathcal{N}^2_x\left[u_1,u_2\right]\right\|_{L^4_x}
			\le
			C\mu^{-1}\left\|u_1\right\|_{H^1}\left\|u_2\right\|_{H^1}
		\]
		which is (\ref{ZEQN:EnergyEstmt_N2x_EQN11}).
	\end{proof}

\subsection{Statement of the existence result}
	We will construct the iterates to (\ref{ZEQN:ItrtnSchemeSuccinct}) by solving the more general initial value problem
	\begin{equation}\label{ZEQN:MoreGeneralIVP}
		\left\{
			\begin{aligned}
				\left(\partial_t-\mathrm{i}\triangle+\mathfrak{P}_B\right)\psi
			=\;&
				\mathcal{Q}\left[\overline{\psi},\psi,\psi\right]
				+ \mathcal{N}^2_0\left[\overline{\psi},\psi\right]\psi
				+ \mathcal{N}^4_t\left[\overline{\psi},\psi,\overline{\psi},\psi\right]\psi
			\\&
				+\mathcal{N}^4_x\left[\overline{\psi},\psi,\overline{\psi},\psi\right]\psi
				- \mathrm{i}\kappa\left|\psi\right|^2\psi
			\;,\\
				\psi(0) =\;& \psi^{\mathrm{in}}\in H^{\mathfrak{m}}
			\;.
			\end{aligned}
		\right.
	\end{equation}
	Throughout this section, $s\ge 1$ is fixed. We impose the following hypotheses.
	\begin{enumerate}[(I)]
		\item \label{ZHYP:SobWt}
			$\mathfrak{m}$ is a Sobolev weight satisfying
			\begin{equation} \label{ZEQN:Impose_SobWt}
				s\le[\mathfrak{m}]_\star \le [\mathfrak{m}]^\star
				\le s+\frac{1}{8}
			\;.\end{equation}
			Note that, in particular, this implies
			\[
				\lambda^s \le \mathfrak{m}(\lambda) \le \lambda^{s+\frac{1}{8}}
			\;,\]
			and more generally
			\[
				\left(\frac{\lambda}{\mu}\right)^s \le \frac{\mathfrak{m}(\lambda)}{\mathfrak{m}(\mu)} \le
					\left(\frac{\lambda}{\mu}\right)^{s+\frac{1}{8}}
				\quad\mbox{ whenever }
				\lambda\ge\mu
			\;.\]

		\item \label{ZHYP:nablaB}
			$B$ is an admissible form which satisfies
			\begin{equation} \label{ZEQN:Impose_nablaB}
				\left\|\nabla B\right\|_{L^1_tL^\infty_x[1]} \le 1
			\;.\end{equation}
	\end{enumerate}
	Under these hypotheses, Lemma \ref{ZLEM:EnergyEstmt_N2x} and (\ref{ZEQN:OldLemma4.9_EQN01})
	guarantee the existence of a constant $K_1=K_1>0$, which we fix once and for all, such that
	\begin{equation} \label{ZEQN:K1_defn}
		\left\|\mathcal{N}^2_x\left[\psi_1,\psi_2\right]\right\|_{L^\infty_{t,x}[T]}
		\le
		\frac{K_1}{2}
		\left\|\psi_1\right\|_{U^2_BH^1[T]}
		\left\|\psi_2\right\|_{U^2_BH^1[T]}
	\;.\end{equation}
	The main result of this section is that the iterates to the iteration scheme (\ref{ZEQN:ItrtnSchemeSuccinct}) can be constructed,
	and they satisfy certain useful bounds.
	More precisely, we have the following.

	\begin{theorem} \label{ZTHM:Sect5MainThm}
		There exists a small constant $\delta_1=\delta_1(s)\in(0,1]$ such that the following holds.

		Assume the hypotheses (\ref{ZHYP:SobWt}), (\ref{ZHYP:nablaB}) above.
		Let $M>0$ and
		let $\psi^{\mathrm{in}}\in H^{\mathfrak{m}}$ with $\|\psi^{\mathrm{in}}\|_{H^{\mathfrak{m}}}\le M$.
		Suppose additionally that
		\begin{equation} \label{ZEQN:Impose_B}
			\left\|B\right\|_{L^\infty_{t,x}[1]} \le K_1M^2
		\end{equation}
		where $K_1$ is the constant appearing in (\ref{ZEQN:K1_defn}).

		Then, with the existence time $T:=\delta_1(1+M)^{-28}\le 1$,
		there exists a unique solution $\psi\in U^2_BH^{\mathfrak{m}}[T]$ to the initial value problem (\ref{ZEQN:MoreGeneralIVP}).
		This solution satisfies
		\[
			\left\|\psi\right\|_{U^2_BH^{\mathfrak{m}}[T]} \le 2M
		\;.\]
		Moreover, letting $\varGamma$ be the extension by zero of $-\frac{1}{2}\mathcal{N}^2_x[\overline{\psi},\psi]$
		from $(0,T]$ to $(0,1]$,
		we have that $\varGamma$ is an admissible form which also verifies
		hypothesis (\ref{ZHYP:nablaB}) and (\ref{ZEQN:Impose_B}).
	\end{theorem}
	
	The basic idea of the proof of Theorem \ref{ZTHM:Sect5MainThm} is to choose $T$ so that an appropriate contraction map can be set up in the same
	\[
		\mathfrak{E}_{M,T} := \left\{\psi\in U^2_BH^{\mathfrak{m}}[T] \;\Big|\; \left\|\psi\right\|_{U^2_BH^{\mathfrak{m}}[T]}\le 2M\right\}
	\;.\]
	The task of proving Theorem \ref{ZTHM:Sect5MainThm} thus reduces to establishing multilinear estimates
	for each nonlinearity on the right-hand side of (\ref{ZEQN:MoreGeneralIVP}).

\subsection{Preliminary bounds}
	In proving our multilinear estimates we will heavily rely on the
	estimates in Lemma \ref{ZLEM:OldLemma4.9}. Due to our hypotheses (\ref{ZHYP:SobWt}) and (\ref{ZHYP:nablaB}),
	and also because of (\ref{ZEQN:Impose_B}),
	the estimates provided by Lemma 4.13 simplify considerably. For ease
	of exposition we will re-state these estimates here.
	
	\begin{definition}
		Let $\mathfrak{m}$ be a Sobolev weight.
		Let $T\in(0,1]$.
		We define the seminorm $\|\cdot\|_{\Psi^{\mathfrak{m}}[T]}$ on functions $\psi:[0,T)\times\mathbb{R}^2_x\rightarrow\mathbb{C}$ by
		\[
			\left\|\psi\right\|_{\Psi^{\mathfrak{m}}[T]}
			:=
			\left\|\psi\right\|_{L^\infty_tH^{\mathfrak{m}}[T]}
			+
			\sup_\mu \mu^{-\frac{1}{4}}\mathfrak{m}(\mu)\left\|\mathrm{P}_\mu\psi\right\|_{L^4_{t,x}[T]}
		\;.\]
		For $\sigma\in\mathbb{R}$, we define $\Psi^\sigma$ to be $\Psi^{\mathfrak{m}}$
		corresponding to $\mathfrak{m}(\lambda)=\lambda^\sigma$.
	\end{definition}
	
	\begin{lemma} \label{ZLEM:OldLemma4.9Restated}
		Let $\mathfrak{m}$ be a Sobolev weight such that $[\mathfrak{m}]\le C(s)$.
		Assume the hypothesis (\ref{ZHYP:nablaB})
		and assume $B$ satisfies (\ref{ZEQN:Impose_B}).
		Let $T\in(0,1]$.
		If $\psi\in V^2_BH^{\mathfrak{m}}[T]$, and $\widetilde{\psi}$ is either $\psi$ or $\overline{\psi}$, then we have
		\begin{equation}\label{ZEQN:OldLemma4.9Restated_EQN01}
			\left\|\widetilde{\psi}\right\|_{\Psi^{\mathfrak{m}}[T]}
			\le
			C(s)\left(1+M\right)^2\left\|\psi\right\|_{V^2_BH^{\mathfrak{m}}[T]}
		\;.\end{equation}
	\end{lemma}
	\begin{proof}
		Due to the $V^2_{\mathrm{rc}}\hookrightarrow U^4$ embedding, 
		(\ref{ZEQN:OldLemma4.9Restated_EQN01}) is simply a restatement of Lemma \ref{ZLEM:OldLemma4.9}.
	\end{proof}

	\begin{lemma} \label{ZLEM:DispEstmt_LinftyCtrl}
		Let $\mathfrak{m}$ be a Sobolev weight such that $[\mathfrak{m}]\le C(s)$. Let $T\in(0,1]$.
		Then
		\[
			\left\|\mathrm{P}_\mu\psi\right\|_{L^4_tL^\infty_x[T]}
			\le
			C(s)\mu^{\frac{3}{4}}\mathfrak{m}(\mu)^{-1}\left\|\psi\right\|_{\Psi^{\mathfrak{m}}[T]}
		\;.\]
		In particular,
		if $\psi\in\Psi^1[T]$ then $\psi\in L^4_tL^\infty_x[T]$ with
		\[
			\left\|\psi\right\|_{L^4_tL^\infty_x[T]}
			\le
			C\left\|\psi\right\|_{\Psi^1[T]}
		\;.\]
	\end{lemma}
	\begin{proof}
		This is trivial from Bernstein's inequality.
	\end{proof}
	
\subsection{Estimates for $\mathcal{N}^2_0$, $\mathcal{N}^2_x$, $\mathcal{N}^4_t$}
	In this subsection, we collect various space-time estimates for $\mathcal{N}^2_0,\mathcal{N}^2_x,\mathcal{N}^4_t$,
	which we will need for our multilinear estimates, and also for our difference estimates in Section \ref{ZSECT:CnvgItrtnScheme}.
	
	\begin{lemma} \label{ZLEM:DispEstmt_N2x}
		Assume the hypothesis (\ref{ZHYP:SobWt}).
		Let $T\in(0,1]$.
		Then
		\begin{equation}\label{ZEQN:DispEstmt_N2x_EQN02}
			\left\|\mathrm{P}_\mu\mathcal{N}^2_x\left[\psi_1,\psi_2\right]\right\|_{L^2_tL^\infty_x[T]}
			\le
			C(s)
			\mu^{-\frac{1}{4}}
			\mathfrak{m}(\mu)^{-1}
			\left(
				\left\|\psi_1\right\|_{\Psi^{\mathfrak{m}}[T]}
				\left\|\psi_2\right\|_{\Psi^1[T]}
				+
				\left\|\psi_1\right\|_{\Psi^1[T]}
				\left\|\psi_2\right\|_{\Psi^{\mathfrak{m}}[T]}
			\right)
		\;.\end{equation}
	\end{lemma}
	\begin{proof}
		For the case $\mu=1$, the Bernstein and Hardy-Littlewood-Sobolev inequalities give
		\[
			\left\|\mathrm{P}_1\mathcal{N}^2_x\left[\psi_1,\psi_2\right]\right\|_{L^2_tL^\infty_x[T]}
			\lesssim
			T^{\frac{1}{2}}\left\|\psi_1\psi_2\right\|_{L^\infty_tL^1_x[T]}
			\le
			C\left\|\psi_1\right\|_{\Psi^1[T]}\left\|\psi_2\right\|_{\Psi^1[T]}
		\;.\]
		For $\mu\ge 2$, Bernstein's inequality and Lemma \ref{ZLEM:DispEstmt_LinftyCtrl} give
		\begin{equation*}\begin{split}
			\left\|\mathrm{P}_\mu\mathcal{N}^2_x\left[\mathrm{P}_\lambda\psi_1,\mathrm{P}_{\le\lambda}\psi_2\right]\right\|_{L^2_tL^\infty_x[T]}
		\lesssim\;&
			\mu^{-1}\left\|\mathrm{P}_\lambda\psi_1\right\|_{L^4_tL^\infty_x[T]}\left\|\psi_2\right\|_{L^4_tL^\infty_x[T]}
		\\\lesssim\;&
			\mu^{-1}\lambda^{\frac{3}{4}}\mathfrak{m}(\lambda)^{-1}\left\|\psi_1\right\|_{\Psi^{\mathfrak{m}}[T]}\left\|\psi_2\right\|_{\Psi^1[T]}
		\\\le\;&
			C(s)\mu^{-1+s}\lambda^{\frac{3}{4}-s}\mathfrak{m}(\mu)^{-1}\left\|\psi_1\right\|_{\Psi^{\mathfrak{m}}[T]}\left\|\psi_2\right\|_{\Psi^1[T]}
		\end{split}\end{equation*}
		where the last inequality is due to hypothesis (\ref{ZHYP:SobWt}).
		Summing over $\lambda\gtrsim\mu$ and noting the symmetry of $\mathcal{N}^2_x[\psi_1,\psi_2]$ in $\psi_1,\psi_2$,
		we obtain (\ref{ZEQN:DispEstmt_N2x_EQN02}).
	\end{proof}
	
	\begin{lemma} \label{ZLEM:DispEstmt_N2t}
		Assume the hypothesis (\ref{ZHYP:SobWt}).
		Let $T\in(0,1]$.
		Then, for $\mu\ge 2$,
		\begin{equation}\label{ZEQN:DispEstmt_N2t_EQN05}
			\left\|\mathrm{P}_\mu\mathcal{N}^2_0\left[\psi_1,\psi_2\right]\right\|_{L^\infty_tL^1_x[T]}
		\le
			C(s)\mu^{-1}\mathfrak{m}(\mu)^{-1}
			\left(
				\left\|\psi_1\right\|_{\Psi^{\mathfrak{m}}[T]}
				\left\|\psi_2\right\|_{\Psi^1[T]}
				+
				\left\|\psi_1\right\|_{\Psi^1[T]}
				\left\|\psi_2\right\|_{\Psi^{\mathfrak{m}}[T]}
			\right)
		\;.\end{equation}
		We also have the estimate
		\begin{equation}\label{ZEQN:DispEstmt_N2t_EQN06}
			\left\|\mathcal{N}^2_0\left[\psi_1,\psi_2\right]\right\|_{L^4_tL^\infty_x[T]}
		\le
			C\left\|\psi_1\right\|_{\Psi^1[T]}\left\|\psi_2\right\|_{\Psi^1[T]}
		\;.\end{equation}
	\end{lemma}
	\begin{proof}
		For the proof of (\ref{ZEQN:DispEstmt_N2t_EQN05}), we have
		\begin{equation*}\begin{split}
			\left\|\mathrm{P}_\mu\mathcal{N}^2_0\left[\mathrm{P}_\lambda\psi_1,\mathrm{P}_{\le\lambda}\psi_2\right]\right\|_{L^\infty_tL^1_x[T]}
		\lesssim\;&
			\mu^{-1}
			\left\|\mathrm{P}_\lambda\psi_1\right\|_{L^\infty_tL^2_x[T]}
			\left\|\nabla\psi_2\right\|_{L^\infty_tL^2_x[T]}
		\\\lesssim\;&
			\mu^{-1}
			\mathfrak{m}(\lambda)^{-1}
			\left\|\psi_1\right\|_{\Psi^{\mathfrak{m}}[T]}
			\left\|\psi_2\right\|_{\Psi^1[T]}
		\\\le\;&
			C(s)\mu^{-1}\mathfrak{m}(\mu)^{-1}\left(\frac{\mu}{\lambda}\right)^s
			\left\|\psi_1\right\|_{\Psi^{\mathfrak{m}}[T]}
			\left\|\psi_2\right\|_{\Psi^1[T]}
		\end{split}\end{equation*}
		where the last inequality follows from the hypothesis (\ref{ZHYP:SobWt}).
		Summing over $\lambda\gtrsim\mu$, and noting that $\mathcal{N}^2_0[\psi_1,\psi_2]$ is skew-symmetric in $\psi_1,\psi_2$,
		we obtain (\ref{ZEQN:DispEstmt_N2t_EQN05}).
		
		We turn to the proof of (\ref{ZEQN:DispEstmt_N2t_EQN06}). By Bernstein (and also Hardy-Littlewood-Sobolev for $\mu=1$)
		and Lemma \ref{ZLEM:DispEstmt_LinftyCtrl}, we have
		\begin{equation*}\begin{split}
			\left\|\mathrm{P}_\mu\mathcal{N}^2_0\left[\mathrm{P}_\lambda\psi_1,\mathrm{P}_{\le\lambda}\psi_2\right]\right\|_{L^4_tL^\infty_x[T]}
		\lesssim\;&
			\mu^{\frac{1}{2}}\left\|\mathrm{P}_\lambda\psi_1\,\nabla\mathrm{P}_{\le\lambda}\psi_2\right\|_{L^4_tL^{\frac{4}{3}}_x[T]}
		\\\lesssim\;&
			\mu^{\frac{1}{2}}\lambda^{-\frac{3}{4}}\left\|\psi_1\right\|_{\Psi^1[T]}\left\|\psi_2\right\|_{\Psi^1[T]}
		\end{split}\end{equation*}
		The right-hand side is summable over $\lambda\gtrsim\mu$.
		Since $\mathcal{N}^2_0$ is skew-symmetric, we have (\ref{ZEQN:DispEstmt_N2t_EQN06}).
	\end{proof}	

	\begin{lemma} \label{ZLEM:DispEstmt_N4t}
		Assume the hypothesis (\ref{ZHYP:SobWt}).
		Let $T\in(0,1]$.
		Then
		\begin{equation} \label{ZEQN:DispEstmt_N4t_EQN01}
			\left\|\mathrm{P}_\mu\mathcal{N}^4_t\left[\psi_1,\psi_2,\psi_3,\psi_4\right]\right\|_{L^2_tL^\infty_x[T]}
		\le
			C(s)\mu^{-\frac{1}{4}}\mathfrak{m}(\mu)^{-1}
			\sum_{\ell=1}^4 \left\|\psi_\ell\right\|_{\Psi^{\mathfrak{m}}[T]}
			\prod_{\substack{k=1\\k\neq\ell}}^4 \left\|\psi_k\right\|_{\Psi^1[T]}
		\;.\end{equation}
	\end{lemma}
	\begin{proof}
		We first deal with the case $\mu=1$. By Bernstein, Hardy-Littlewood-Sobolev, and Lemma \ref{ZLEM:EnergyEstmt_N2x}, we find
		\begin{equation*}\begin{split}
			\left\|\mathrm{P}_1\mathcal{N}^4_t\left[\psi_1,\psi_2,\psi_3,\psi_4\right]\right\|_{L^2_tL^\infty_x[T]}
		\lesssim\;&
			\left\|\mathcal{N}^2_x\left[\psi_1,\psi_2\right]\psi_3\psi_4\right\|_{L^\infty_tL^1_x[T]}
		\lesssim
			\prod_{\ell=1}^4\left\|\psi_\ell\right\|_{\Psi^1[T]}
		\end{split}\end{equation*}
		as required.
		
		Suppose now $\mu\ge 2$. Then, by Bernstein,
		\begin{equation*}\begin{split}
			\left\|\mathrm{P}_\mu\mathcal{N}^4_t\left[\psi_1,\psi_2,\psi_3,\psi_4\right]\right\|_{L^2_tL^\infty_x[T]}
		\lesssim\;&
			\left\|\mathrm{P}_{\gtrsim\mu}\mathcal{N}^2_x\left[\psi_1,\psi_2\right]\psi_3\psi_4\right\|_{L^2_tL^2_x[T]}
		\\&+
			\mu^{-1}\left\|\mathrm{P}_{\ll\mu}\mathcal{N}^2_x\left[\psi_1,\psi_2\right]
				\mathrm{P}_{\approx\mu}\left(\psi_3\psi_4\right)\right\|_{L^2_tL^\infty_x[T]}
		\\\lesssim\;&
			\left\|\mathrm{P}_{\gtrsim\mu}\mathcal{N}^2_x\left[\psi_1,\psi_2\right]\right\|_{L^2_tL^\infty_x[T]}
			\left\|\psi_3\right\|_{L^\infty_tL^4_x[T]}
			\left\|\psi_4\right\|_{L^\infty_tL^4_x[T]}
		\\&+
			\mu^{-1}\left\|\mathcal{N}^2_x\left[\psi_1,\psi_2\right]\right\|_{L^\infty_{t,x}[T]}
			\left\|\mathrm{P}_{\approx\mu}\left(\psi_3\psi_4\right)\right\|_{L^2_tL^\infty_x[T]}
		\\=:\;& \mathrm{I} + \mathrm{II}
		\;.\end{split}\end{equation*}
		Using Lemma \ref{ZLEM:DispEstmt_N2x} and the Sobolev embedding $H^1\hookrightarrow L^4_x$,
		\[
			\mathrm{I} \le C(s)\mu^{-\frac{1}{4}}\mathfrak{m}(\mu)^{-1}
			\left(
				\left\|\psi_1\right\|_{\Psi^{\mathfrak{m}}[T]}
				\left\|\psi_2\right\|_{\Psi^1[T]}
				+
				\left\|\psi_1\right\|_{\Psi^1[T]}
				\left\|\psi_2\right\|_{\Psi^{\mathfrak{m}}[T]}
			\right)
			\left\|\psi_3\right\|_{\Psi^1[T]}
			\left\|\psi_4\right\|_{\Psi^1[T]}
		\;.\]
		Now, we have
		\begin{equation*}\begin{split}
			\left\|\mathrm{P}_\lambda\psi_3\,\mathrm{P}_{\le\lambda}\psi_4\right\|_{L^2_tL^\infty_x[T]}
		\lesssim\;&
			\left\|\mathrm{P}_\lambda\psi_3\right\|_{L^4_tL^\infty_x[T]}
			\left\|\psi_4\right\|_{L^4_tL^\infty_x[T]}
		\\\lesssim\;&
			\lambda^{\frac{3}{4}}\mathfrak{m}(\lambda)^{-1}
			\left\|\psi_3\right\|_{\Psi^{\mathfrak{m}}[T]}
			\left\|\psi_4\right\|_{\Psi^1[T]}
		\\\le\;&
			C(s)
			\lambda^{\frac{3}{4}-s}\mathfrak{m}(\mu)^{-1}\mu^s
			\left\|\psi_3\right\|_{\Psi^{\mathfrak{m}}[T]}
			\left\|\psi_4\right\|_{\Psi^1[T]}
		\end{split}\end{equation*}
		where the last inequality is due to the hypothesis (\ref{ZHYP:SobWt}).
		Summing over $\lambda\gtrsim\mu$ we obtain, by symmetry,
		\[
			\left\|\mathrm{P}_{\approx\mu}\left(\psi_3\psi_4\right)\right\|_{L^2_tL^\infty_x[T]}
		\le
			C(s)
			\mu^{\frac{3}{4}}\mathfrak{m}(\mu)^{-1}
			\left(
				\left\|\psi_3\right\|_{\Psi^{\mathfrak{m}}[T]}
				\left\|\psi_4\right\|_{\Psi^1[T]}
				+
				\left\|\psi_3\right\|_{\Psi^1[T]}
				\left\|\psi_4\right\|_{\Psi^{\mathfrak{m}}[T]}
			\right)
		\;.\]
		Hence, by Lemma \ref{ZLEM:EnergyEstmt_N2x},
		\[
			\mathrm{II}
			\le
			C(s)
			\mu^{-\frac{1}{4}}\mathfrak{m}(\mu)^{-1}
			\left\|\psi_1\right\|_{\Psi^1[T]}
			\left\|\psi_2\right\|_{\Psi^1[T]}
			\left(
				\left\|\psi_3\right\|_{\Psi^{\mathfrak{m}}[T]}
				\left\|\psi_4\right\|_{\Psi^1[T]}
				+
				\left\|\psi_3\right\|_{\Psi^1[T]}
				\left\|\psi_4\right\|_{\Psi^{\mathfrak{m}}[T]}
			\right)
		\;.\]
		The proof is complete.
	\end{proof}
	
\subsection{Multilinear estimates}
	
	We now estimate each of the nonlinearities in (\ref{ZEQN:MoreGeneralIVP}) in $\mathrm{D}U^2_BH^{\mathfrak{m}}[T]$.
	This is accomplished with the aid of the duality principle, Lemma \ref{ZLEM:UpVpDualityB}.
	
	\begin{lemma} \label{ZLEM:MLEstmt_Q}
		Assume the hypotheses (\ref{ZHYP:SobWt}), (\ref{ZHYP:nablaB}). Let $T\in(0,1]$. Then
		\[
			\left\|\mathcal{Q}\left[\psi_1,\psi_2,\psi_3\right]\right\|_{\mathrm{D}U^2_BH^{\mathfrak{m}}[T]}
			\le
			C(s)T^{\frac{1}{2}}
			\sum_{\ell=1}^3 \left\|\psi_\ell\right\|_{\Psi^{\mathfrak{m}}[T]}
				\prod_{\substack{k=1\\k\neq\ell}}^3  \left\|\psi_k\right\|_{\Psi^1[T]}
		\;.\]
	\end{lemma}
	\begin{proof}
		By duality, it suffices to prove the estimate
		\begin{equation}\label{ZEQN:MLEstmt_Q_EQN11}
			\left|\int_0^T\int_{\mathbb{R}^2}
				\overline{\psi_0}\,\mathcal{Q}\left[\psi_1,\psi_2,\psi_3\right]\,\mathrm{d}x\,\mathrm{d}t\right|
			\le
			C(s)T^{\frac{1}{2}}
			\left\|\psi_0\right\|_{V^2_BH^{\mathfrak{m}^{-1}}[T]}
			\sum_{\ell=1}^3 \left\|\psi_\ell\right\|_{\Psi^{\mathfrak{m}}[T]}
				\prod_{\substack{k=1\\k\neq\ell}}^3 \left\|\psi_k\right\|_{\Psi^1[T]}
		\;.\end{equation}
		Using Lemma \ref{ZLEM:DispEstmt_N2x}, we have
		\begin{equation*}\begin{split}
			\bigg|\int_0^T\int_{\mathbb{R}^2}
		&
			\mathrm{P}_\nu\overline{\psi_0}\,
			\mathrm{P}_\mu\mathcal{N}^2_x\left[\psi_1,\psi_2\right]
			\cdot\nabla\mathrm{P}_\lambda\psi_3
			\,\mathrm{d}x\,\mathrm{d}t
			\bigg|
		\\\lesssim\;&
			T^{\frac{1}{2}}
			\left\|\mathrm{P}_\nu\psi_0\right\|_{L^\infty_tL^2_x[T]}
			\left\|\mathrm{P}_\mu\mathcal{N}^2_x\left[\psi_1,\psi_2\right]\right\|_{L^2_tL^\infty_x[T]}
			\left\|\mathrm{P}_\lambda\psi_3\right\|_{L^\infty_tH^1[T]}
		\\\le\;&
			C(s)T^{\frac{1}{2}}
			\mathfrak{m}(\nu)
			\mu^{-\frac{1}{4}}\mathfrak{m}(\mu)^{-1}
			\left\|\psi_0\right\|_{V^2_BH^{\mathfrak{m}^{-1}}[T]}
			\sum_{\ell=1}^3 \left\|\psi_\ell\right\|_{\Psi^{\mathfrak{m}}[T]}
				\prod_{\substack{k=1\\k\neq\ell}}^3 \left\|\psi_k\right\|_{\Psi^1[T]}
		\\\le\;&
			C(s)T^{\frac{1}{2}}
			\nu^{s+\frac{1}{8}}
			\mu^{-\frac{1}{4}-s}
			\left\|\psi_0\right\|_{V^2_BH^{\mathfrak{m}^{-1}}[T]}
			\sum_{\ell=1}^3 \left\|\psi_\ell\right\|_{\Psi^{\mathfrak{m}}[T]}
				\prod_{\substack{k=1\\k\neq\ell}}^3 \left\|\psi_k\right\|_{\Psi^1[T]}
		\end{split}\end{equation*}
		where the last two inequalities follow from the hypotheses (\ref{ZHYP:SobWt}), (\ref{ZHYP:nablaB}).
		Now, the right-hand side is summable over $\{\mu\approx\max(\lambda,\nu)\}$ to give (\ref{ZEQN:MLEstmt_Q_EQN11}).
	\end{proof}

	\begin{lemma} \label{ZLEM:MLEstmt_N2tphi}
		Assume the hypotheses (\ref{ZHYP:SobWt}), (\ref{ZHYP:nablaB}). Assume also that $B$ satisfies (\ref{ZEQN:Impose_B}). Let $T\in(0,1]$. Then
		\[
			\left\|\mathcal{N}^2_0\left[\psi_1,\psi_2\right]\psi_3\right\|_{\mathrm{D}U^2_BH^{\mathfrak{m}}[T]}
			\le
			C(s)\left(1+M\right)^2T^{\frac{1}{2}}
			\sum_{\ell=1}^3 \left\|\psi_\ell\right\|_{\Psi^{\mathfrak{m}}[T]}
				\prod_{\substack{k=1\\k\neq\ell}}^3 \left\|\psi_k\right\|_{\Psi^1[T]}
		\;.\]
	\end{lemma}
	\begin{proof}
		By duality, it suffices to prove the estimate
		\begin{equation} \label{ZEQN:MLEstmt_N2tphi_EQN11}
			\bigg|\int_0^T\int_{\mathbb{R}^2}
				\overline{\psi_0}\mathcal{N}^2_0\left[\psi_1,\psi_2\right]\psi_3\,\mathrm{d}x\,\mathrm{d}t\bigg|
			\le
			C(s)\left(1+M\right)^2T^{\frac{1}{2}}
			\left\|\psi_0\right\|_{V^2_BH^{\mathfrak{m}^{-1}}[T]}
			\sum_{\ell=1}^3 \left\|\psi_\ell\right\|_{\Psi^{\mathfrak{m}}[T]}
				\prod_{\substack{k=1\\k\neq\ell}}^3 \left\|\psi_k\right\|_{\Psi^1[T]}
		\;.\end{equation}
		By (\ref{ZEQN:DispEstmt_N2t_EQN06}), H\"{o}lder's inequality gives
		\begin{equation*}\begin{split}
			\sum_\nu\bigg|\int_0^T\int_{\mathbb{R}^2}
		&
				\mathrm{P}_\nu\overline{\psi_0}\,
				\mathrm{P}_{\lesssim\nu}\mathcal{N}^2_0\left[\psi_1,\psi_2\right]\,
				\mathrm{P}_{\approx\nu}\psi_3\,\mathrm{d}x\,\mathrm{d}t\bigg|
		\\\lesssim\;&
			\int_0^T \left\|\mathcal{N}^2_0\left[\psi_1,\psi_2\right]\right\|_{L^\infty_x}
				\sum_\nu \left\|\mathrm{P}_\nu\psi_0\right\|_{L^2_x}\left\|\mathrm{P}_{\approx\nu}\overline{\psi_3}\right\|_{L^2_x}
				\,\mathrm{d}t
		\\\lesssim\;&
			T^{\frac{3}{4}}
			\left\|\psi_0\right\|_{L^\infty_tH^{\mathfrak{m}^{-1}}[T]}
			\left\|\mathcal{N}^2_0\left[\psi_1,\psi_2\right]\right\|_{L4_tL^\infty_x[T]}
			\left\|\psi_3\right\|_{l^\infty_tH^{\mathfrak{m}}[T]}
		\\\le\;&
			C(s)
			T^{\frac{1}{2}}
			\left\|\psi_0\right\|_{V^2_BH^{\mathfrak{m}^{-1}}[T]}
			\left\|\psi_1\right\|_{\Psi^1[T]}
			\left\|\psi_2\right\|_{\Psi^1[T]}
			\left\|\psi_3\right\|_{\Psi^{\mathfrak{m}}[T]}
		\end{split}\end{equation*}
		where the last inequality follows from Lemma \ref{ZLEM:DispEstmt_N2t} and the hypotheses (\ref{ZHYP:SobWt}) and (\ref{ZHYP:nablaB}).

		By the Littlewood-Paley trichotomy, to prove (\ref{ZEQN:MLEstmt_N2tphi_EQN11}) it remains to show
		\[
			\left(\sum_{\mu\approx\lambda\gg\nu} + \sum_{\nu\approx\lambda\gg\mu}\right)
			\left|\int_0^T\int_{\mathbb{R}^2}
			\mathrm{P}_\nu\overline{\psi_0}\,\mathrm{P}_\mu\mathcal{N}^2_0\left[\psi_1,\psi_2\right]\,\mathrm{P}_\lambda\psi_3\,\mathrm{d}x\,\mathrm{d}t\right|
			\le
			\mbox{RHS(\ref{ZEQN:MLEstmt_N2tphi_EQN11})}
		\;.\]
		For this, we note using (\ref{ZEQN:DispEstmt_N2t_EQN05}) and Lemmas \ref{ZLEM:OldLemma4.9Restated}, \ref{ZLEM:DispEstmt_LinftyCtrl} that
		\begin{equation*}\begin{split}
			\bigg|\int_0^T
		&
			\int_{\mathbb{R}^2}
			\mathrm{P}_\nu\overline{\psi_0}\,\mathrm{P}_\mu\mathcal{N}^2_0\left[\psi_1,\psi_2\right]\,\mathrm{P}_\lambda\psi_3\,\mathrm{d}x\,\mathrm{d}t\bigg|
		\\\lesssim\;&
			T^{\frac{1}{2}}
			\left\|\mathrm{P}_\nu\psi_0\right\|_{L^4_tL^\infty_x[T]}
			\left\|\mathrm{P}_\mu\mathcal{N}^2_0\left[\psi_1,\psi_2\right]\right\|_{L^\infty_tL^1_x[T]}
			\left\|\mathrm{P}_\lambda\psi_3\right\|_{L^4_tL^\infty_x[T]}
		\\\le\;&
			C(s)\left(1+M\right)^2T^{\frac{1}{2}}
			\nu^{\frac{3}{4}}\mathfrak{m}(\nu)
			\mu^{-1}\mathfrak{m}(\mu)^{-1}
			\lambda^{\frac{3}{4}-1}
			\left\|\psi_0\right\|_{V^2_BH^{\mathfrak{m}^{-1}}[T]}
			\sum_{\ell=1}^3 \left\|\psi_\ell\right\|_{\Psi^{\mathfrak{m}}[T]}
				\prod_{\substack{k=1\\k\neq\ell}}^3 \left\|\psi_k\right\|_{\Psi^1[T]}
		\\\le\;&
			C(s)\left(1+M\right)^2T^{\frac{1}{2}}
			\nu^{\frac{3}{4}+s}
			\mu^{-1-s}
			\lambda^{-\frac{1}{4}}
			\left\|\psi_0\right\|_{V^2_BH^{\mathfrak{m}^{-1}}[T]}
			\sum_{\ell=1}^3 \left\|\psi_\ell\right\|_{\Psi^{\mathfrak{m}}[T]}
				\prod_{\substack{k=1\\k\neq\ell}}^3 \left\|\psi_k\right\|_{\Psi^1[T]}
		\;.\end{split}\end{equation*}
		The right-hand side is summable over $\{\mu\approx\lambda\gg\nu\}$ and over $\{\mu\approx\nu\gg\lambda\}$, as desired.
	\end{proof}

	\begin{lemma} \label{ZLEM:MLEstmt_N4tphi}
		Assume the hypotheses (\ref{ZHYP:SobWt}), (\ref{ZHYP:nablaB}).
		Let $T\in(0,1]$.
		Then
		\[
			\left\|\mathcal{N}^4_t\left[\psi_1,\psi_2,\psi_3,\psi_4\right]\psi_5\right\|_{\mathrm{D}U^2_BH^{\mathfrak{m}}[T]}
		\le
			C(s)T^{\frac{1}{2}}
			\sum_{\ell=1}^5 \left\|\psi_\ell\right\|_{\Psi^{\mathfrak{m}}[T]} \prod_{\substack{k=1\\k\neq\ell}}^5 \left\|\psi_k\right\|_{\Psi^1[T]}
		\;.\]
	\end{lemma}
	\begin{proof}
		By duality, it suffices to prove the estimate
		\begin{equation} \label{ZEQN:MLEstmt_N4tphi_EQN11}
			\bigg|\int_0^T\int_{\mathbb{R}^2}
				\overline{\psi_0}\mathcal{N}^4_t\left[\psi_1,\psi_2,\psi_3,\psi_4\right]\psi_5\,\mathrm{d}x\,\mathrm{d}t\bigg|
			\le
			C(s)T^{\frac{1}{2}}
			\left\|\psi_0\right\|_{V^2_BH^{\mathfrak{m}^{-1}}[T]}
			\sum_{\ell=1}^5 \left\|\psi_\ell\right\|_{\Psi^{\mathfrak{m}}[T]}
				\prod_{\substack{k=1\\k\neq\ell}}^5 \left\|\psi_k\right\|_{\Psi^1[T]}
		\;.\end{equation}

		Firstly, using H\"older's inequality and Lemma \ref{ZLEM:DispEstmt_N4t}, we have
		\begin{equation*}\begin{split}
			\sum_\nu \bigg|
			\int_0^T
		&	
			\int_{\mathbb{R}^2} \mathrm{P}_\nu\overline{\psi_0}
				\,\mathrm{P}_{\lesssim\nu}\mathcal{N}^4_t\left[\psi_1,\psi_2,\psi_3,\psi_4\right]
				\,\mathrm{P}_{\approx\nu}\psi_5
				\,\mathrm{d}x\,\mathrm{d}t\bigg|
		\\\lesssim\;&
			\int_0^T \left\|\mathcal{N}^4_t\left[\psi_1,\psi_2,\psi_3,\psi_4\right](t)\right\|_{L^\infty_x}
				\sum_\nu \left\|\mathrm{P}_\nu\psi_0(t)\right\|_{L^2_x} \left\|\mathrm{P}_{\approx\nu}\psi_5(t)\right\|_{L^2_x}\mathrm{d}t
		\\\le\;&
			C(s)T^{\frac{1}{2}}
			\left\|\psi_0\right\|_{L^\infty_tH^{\mathfrak{m}^{-1}}[T]}
			\left\|\mathcal{N}^4_t\left[\psi_1,\psi_2,\psi_3,\psi_4\right]\right\|_{L^2_tL^\infty_x[T]}
			\left\|\psi_5\right\|_{L^\infty_tH^{\mathfrak{m}}[T]}
		\\\le\;&
			C(s)T^{\frac{1}{2}}
			\left\|\psi_0\right\|_{V^2_BH^{\mathfrak{m}^{-1}}[T]}
			\left(\prod_{\ell=1}^4 \left\|\psi_\ell\right\|_{\Psi^1[T]}\right)
			\left\|\psi_5\right\|_{\Psi^{\mathfrak{m}}[T]}
		\;.\end{split}\end{equation*}

		By the Littlewood-Paley trichotomy, to prove (\ref{ZEQN:MLEstmt_N4tphi_EQN11}) it remains to show that
		\begin{equation}\label{ZEQN:MLEstmt_N4tphi_EQN12}
			\left(\sum_{\mu\approx\nu\gg\lambda} + \sum_{\mu\approx\lambda\gg\nu}\right)
			\left|\int_0^T\int_{\mathbb{R}^2}
				\mathrm{P}_\nu\overline{\psi_0}
				\,\mathrm{P}_\mu\mathcal{N}^4_t\left[\psi_1,\psi_2,\psi_3,\psi_4\right]
				\,\mathrm{P}_\lambda\psi_5
				\,\mathrm{d}x\,\mathrm{d}t
			\right|
			\le
			\mbox{RHS(\ref{ZEQN:MLEstmt_N4tphi_EQN11})}
		\;.\end{equation}
		For this, using H\"older's inequality and Lemma \ref{ZLEM:DispEstmt_N4t} again, we find
		\begin{equation*}\begin{split}
			\bigg|\int_0^T\int_{\mathbb{R}^2}
		&
				\mathrm{P}_\nu\overline{\psi_0}
				\,\mathrm{P}_\mu\mathcal{N}^4_t\left[\psi_1,\psi_2,\psi_3,\psi_4\right]
				\,\mathrm{P}_\lambda\psi_5
				\,\mathrm{d}x\,\mathrm{d}t
			\bigg|
		\\\lesssim\;&
			T^{\frac{1}{2}}
			\left\|\mathrm{P}_\nu\psi_0\right\|_{L^\infty_tL^2_x[T]}
			\left\|\mathrm{P}_\mu\mathcal{N}^4_t\left[\psi_1,\psi_2,\psi_3,\psi_4\right]\right\|_{L^2_tL^\infty_x[T]}
			\left\|\mathrm{P}_\lambda\psi_5\right\|_{L^\infty_tL^2_x[T]}
		\\\le\;&
			C(s)T^{\frac{1}{2}}
			\mathfrak{m}(\nu)\mu^{-1}\mathfrak{m}(\mu)^{-1}\lambda^{-1}
			\left\|\psi_0\right\|_{V^2_BH^{\mathfrak{m}^{-1}}[T]}
			\left(
				\sum_{\ell=1}^4 \left\|\psi_\ell\right\|_{\Psi^{\mathfrak{m}}[T]}
				\prod_{\substack{k=1\\k\neq\ell}}^4 \left\|\psi_k\right\|_{\Psi^1[T]}
			\right)
			\left\|\psi_5\right\|_{\Psi^1[T]}
		\;.\end{split}\end{equation*}
		Due to hypothesis (\ref{ZHYP:SobWt}), the right-hand side is summable over $\{\mu\approx\max(\nu,\lambda)\}$.
		Therefore we obtain (\ref{ZEQN:MLEstmt_N4tphi_EQN12}) as required.
	\end{proof}

	\begin{lemma} \label{ZLEM:MLEstmt_N4xphi}
		Assume the hypotheses (\ref{ZHYP:SobWt}), (\ref{ZHYP:nablaB}).
		Let $T\in(0,1]$.
		Then
		\[
			\left\|\mathcal{N}^4_x\left[\psi_1,\psi_2,\psi_3,\psi_4\right]\psi_5\right\|_{\mathrm{D}U^2_BH^{\mathfrak{m}}[T]}
		\le
			C(s)T^{\frac{1}{2}}
			\sum_{\ell=1}^5 \left\|\psi_\ell\right\|_{\Psi^{\mathfrak{m}}[T]} \prod_{\substack{k=1\\k\neq\ell}}^5 \left\|\psi_k\right\|_{\Psi^1[T]}
		\;.\]
	\end{lemma}
	\begin{proof}
		By duality, it suffices to prove the estimate
		\begin{equation} \label{ZEQN:MLEstmt_N4xphi_EQN11}
			\bigg|\int_0^T\int_{\mathbb{R}^2}
				\overline{\psi_0}\mathcal{N}^4_x\left[\psi_1,\psi_2,\psi_3,\psi_4\right]\psi_5\,\mathrm{d}x\,\mathrm{d}t\bigg|
			\le
			C(s)T^{\frac{1}{2}}
			\left\|\psi_0\right\|_{V^2_BH^{\mathfrak{m}^{-1}}[T]}
			\sum_{\ell=1}^5 \left\|\psi_\ell\right\|_{\Psi^{\mathfrak{m}}[T]}
				\prod_{\substack{k=1\\k\neq\ell}}^5 \left\|\psi_k\right\|_{\Psi^1[T]}
		\;.\end{equation}
		By the Littlewood-Paley trichotomy and symmetry, it suffices to verify the estimates
		\begin{equation} \label{ZEQN:MLEstmt_N4xphi_EQN12}
			\sum_{\mu,\nu\;:\;\mu\gtrsim\nu}
			\left|\int_0^T\int_{\mathbb{R}^2}
				\mathrm{P}_\nu\overline{\psi_0}\,
				\mathrm{P}_\mu\mathcal{N}^2_x\left[\psi_1,\psi_2\right]\cdot\mathrm{P}_{\lesssim\mu}\mathcal{N}^2_x\left[\psi_3,\psi_4\right]
				\,\psi_5
				\,\mathrm{d}x\,\mathrm{d}t
			\right|
			\le
			\mbox{RHS(\ref{ZEQN:MLEstmt_N4xphi_EQN11})}
		\end{equation}
		and
		\begin{equation} \label{ZEQN:MLEstmt_N4xphi_EQN13}
			\sum_\nu
			\left|\int_0^T\int_{\mathbb{R}^2}
				\mathrm{P}_\nu\overline{\psi_0}\,
				\mathrm{P}_{\ll\nu}\mathcal{N}^2_x\left[\psi_1,\psi_2\right]\cdot\mathrm{P}_{\ll\nu}\mathcal{N}^2_x\left[\psi_3,\psi_4\right]
				\,\mathrm{P}_{\approx\nu}\psi_5
				\,\mathrm{d}x\,\mathrm{d}t
			\right|
			\le
			\mbox{RHS(\ref{ZEQN:MLEstmt_N4xphi_EQN11})}
		\;.\end{equation}

		Using Lemmas \ref{ZLEM:EnergyEstmt_N2x} and \ref{ZLEM:DispEstmt_N2x}, a typical summand on the left-hand side of (\ref{ZEQN:MLEstmt_N4xphi_EQN12}) is controlled by
		\begin{equation*}\begin{split}
			T^{\frac{1}{2}}
		&
			\left\|\mathrm{P}_\nu\psi_0\right\|_{L^\infty_tL^2_x[T]}
			\left\|\mathrm{P}_\mu\mathcal{N}^2_x\left[\psi_1,\psi_2\right]\right\|_{L^2_tL^\infty_x[T]}
			\left\|\mathcal{N}^2_x\left[\psi_3,\psi_4\right]\right\|_{L^\infty_{t,x}[T]}
			\left\|\psi_5\right\|_{L^\infty_tL^2_x[T]}
		\\\le\;&
			C(s)T^{\frac{1}{2}}
			\mathfrak{m}(\nu)
			\mu^{-\frac{1}{4}}\mathfrak{m}(\mu)^{-1}
			\left\|\psi_0\right\|_{V^2_BH^{\mathfrak{m}^{-1}}[T]}
		\\&\cdot
			\left(
				\left\|\psi_1\right\|_{\Psi^{\mathfrak{m}}[T]}
				\left\|\psi_2\right\|_{\Psi^1[T]}
				+
				\left\|\psi_1\right\|_{\Psi^1[T]}
				\left\|\psi_2\right\|_{\Psi^{\mathfrak{m}}[T]}
			\right)
			\left\|\psi_3\right\|_{\Psi^1[T]}
			\left\|\psi_4\right\|_{\Psi^1[T]}
			\left\|\psi_5\right\|_{\Psi^1[T]}
		\;.\end{split}\end{equation*}
		Summing over $\{\mu\gtrsim\nu\}$, we obtain (\ref{ZEQN:MLEstmt_N4xphi_EQN12}) as required.

		As for (\ref{ZEQN:MLEstmt_N4xphi_EQN13}), we use Lemma \ref{ZLEM:EnergyEstmt_N2x} to obtain the estimate
		\begin{equation*}\begin{split}
			\mbox{LHS(\ref{ZEQN:MLEstmt_N4xphi_EQN13})}
		\lesssim\;&
			\int_0^T
			\left\|\mathcal{N}^2_x\left[\psi_1,\psi_2\right](t)\right\|_{L^\infty_x}
			\left\|\mathcal{N}^2_x\left[\psi_3,\psi_4\right](t)\right\|_{L^\infty_x}
			\left(\sum_\nu
				\left\|\mathrm{P}_\nu\psi_0(t)\right\|_{L^2_x}
				\left\|\mathrm{P}_{\approx\nu}\psi_5(t)\right\|_{L^2_x}
			\right)
			\mathrm{d}t
		\\\le\;&
			C(s)T
			\left\|\psi_0\right\|_{L^\infty_tH^{\mathfrak{m}^{-1}}[T]}
			\left\|\mathcal{N}^2_x\left[\psi_1,\psi_2\right]\right\|_{L^\infty_{t,x}[T]}
			\left\|\mathcal{N}^2_x\left[\psi_3,\psi_4\right]\right\|_{L^\infty_{t,x}[T]}
			\left\|\psi_5\right\|_{L^\infty_tH^{\mathfrak{m}}[T]}
		\\\le\;&
			C(s)T^{\frac{1}{2}}
			\left\|\psi_0\right\|_{V^2_BH^{\mathfrak{m}^{-1}}[T]}
			\left(\prod_{k=1}^4 \left\|\psi_k\right\|_{\Psi^1[T]}\right)
			\left\|\psi_5\right\|_{\Psi^{\mathfrak{m}}[T]}
		\;.\end{split}\end{equation*}
		This verifies (\ref{ZEQN:MLEstmt_N4xphi_EQN13}) and hence completes the proof.
	\end{proof}
	
	\begin{lemma} \label{ZLEM:MLEstmt_phi3}
		Assume the hypotheses (\ref{ZHYP:SobWt}), (\ref{ZHYP:nablaB}).
		Assume also that $B$ satisfies (\ref{ZEQN:Impose_B}).
		Let $T\in(0,1]$.
		Then
		\[
			\left\|\psi_1\psi_2\psi_3\right\|_{\mathrm{D}U^2_BH^{\mathfrak{m}}[T]}
		\le
			C(s)\left(1+M\right)^2T^{\frac{1}{2}}
			\sum_{\ell=1}^3 \left\|\psi_\ell\right\|_{\Psi^{\mathfrak{m}}[T]} \prod_{\substack{k=1\\k\neq\ell}}^3 \left\|\psi_k\right\|_{\Psi^1[T]}
		\;.\]
	\end{lemma}
	\begin{proof}
		By duality, it suffices to prove the estimate
		\begin{equation} \label{ZEQN:MLEstmt_phi3_EQN11}
			\left|\int_0^T\int_{\mathbb{R}^2}
				\overline{\psi_0}\psi_1\psi_2\psi_3\,\mathrm{d}x\,\mathrm{d}t\right|
			\le
			C(s)\left(1+M\right)^2T^{\frac{1}{2}}
			\left\|\psi_0\right\|_{V^2_BH^{\mathfrak{m}^{-1}}[T]}
			\sum_{\ell=1}^3 \left\|\psi_\ell\right\|_{\Psi^{\mathfrak{m}}[T]} \prod_{\substack{k=1\\k\neq\ell}}^3 \left\|\psi_k\right\|_{\Psi^1[T]}
		\;.\end{equation}
		By the Littlewood-Paley trichotomy and symmetry, (\ref{ZEQN:MLEstmt_phi3_EQN11}) follows from the two estimates
		\begin{equation} \label{ZEQN:MLEstmt_phi3_EQN12}
			\sum_\nu \left|\int_0^T\int_{\mathbb{R}^2}
				\mathrm{P}_\nu\overline{\psi_0}\,\mathrm{P}_{\approx\nu}\psi_1
				\,\mathrm{P}_{\lesssim\nu}\psi_2\,\mathrm{P}_{\lesssim\nu}\psi_3
				\,\mathrm{d}x\,\mathrm{d}t\right|
			\le
			\mbox{RHS(\ref{ZEQN:MLEstmt_phi3_EQN11})}
		\end{equation}
		and
		\begin{equation} \label{ZEQN:MLEstmt_phi3_EQN13}
			\sum_{\lambda,\nu\;:\;\lambda\gg\nu} \left|\int_0^T\int_{\mathbb{R}^2}
				\mathrm{P}_\nu\overline{\psi_0}\,\mathrm{P}_\lambda\psi_1
				\,\mathrm{P}_{\approx\lambda}\psi_2\,\mathrm{P}_{\lesssim\lambda}\psi_3
				\,\mathrm{d}x\,\mathrm{d}t\right|
			\le
			\mbox{RHS(\ref{ZEQN:MLEstmt_phi3_EQN11})}
		\;.\end{equation}
		
		Using Lemma \ref{ZLEM:DispEstmt_LinftyCtrl}, we easily obtain (\ref{ZEQN:MLEstmt_phi3_EQN12}) as follows,
		\begin{equation*}\begin{split}
			\mbox{LHS(\ref{ZEQN:MLEstmt_phi3_EQN12})}
		\lesssim\;&
			\int_0^T \left(\sum_\nu \left\|\mathrm{P}_\nu\psi_0(t)\right\|_{L^2_x} \left\|\mathrm{P}_{\approx\nu}\psi_1(t)\right\|_{L^2_x}\right)
			\left\|\psi_2(t)\right\|_{L^\infty_x} \left\|\psi_3(t)\right\|_{L^\infty_x}\mathrm{d}t
		\\\le\;&
			C(s)T^{\frac{1}{2}}
			\left\|\psi_0\right\|_{L^\infty_tH^{\mathfrak{m}^{-1}}[T]}
			\left\|\psi_1\right\|_{L^\infty_tH^{\mathfrak{m}}[T]}
			\left\|\psi_2\right\|_{L^4_tL^\infty_x[T]}
			\left\|\psi_3\right\|_{L^4_tL^\infty_x[T]}
		\\\le\;&
			C(s)T^{\frac{1}{2}}
			\left\|\psi_0\right\|_{V^2_BH^{\mathfrak{m}^{-1}}[T]}
			\left\|\psi_1\right\|_{\Psi^{\mathfrak{m}}[T]}
			\left\|\psi_2\right\|_{\Psi^1[T]}
			\left\|\psi_3\right\|_{\Psi^1[T]}
		\;.\end{split}\end{equation*}
		
		We now turn to the proof of (\ref{ZEQN:MLEstmt_phi3_EQN13}).
		Using Lemma \ref{ZLEM:OldLemma4.9}, a typical summand on the left-hand side of (\ref{ZEQN:MLEstmt_phi3_EQN13}) is controlled by
		\begin{equation*}\begin{split}
			C(s)T^{\frac{1}{2}}
		&
			\left\|\mathrm{P}_\nu\psi_0\right\|_{L^4_tL^\infty_x[T]}
			\left\|\mathrm{P}_\lambda\psi_1\right\|_{L^\infty_tL^2_x[T]}
			\left\|\mathrm{P}_{\approx\lambda}\psi_2\right\|_{L^\infty_tL^2_x[T]}
			\left\|\psi_3\right\|_{L^4_tL^\infty_x[T]}
		\\\le\;&
			C(s)T^{\frac{1}{2}}\left(1+M\right)^2
			\nu^{\frac{3}{4}}\mathfrak{m}(\nu)
			\lambda^{-1}\mathfrak{m}(\lambda)^{-1}
			\left\|\psi_0\right\|_{V^2_BH^{\mathfrak{m}^{-1}}[T]}
			\left\|\psi_1\right\|_{\Psi^{\mathfrak{m}}[T]}
			\left\|\psi_2\right\|_{\Psi^1[T]}
			\left\|\psi_3\right\|_{\Psi^1[T]}
		\;.\end{split}\end{equation*}
		Summing up over $\{\lambda\gg\nu\}$, we obtain (\ref{ZEQN:MLEstmt_phi3_EQN13}) as required.
	\end{proof}
	
\subsection{Proof of Theorem \ref{ZTHM:Sect5MainThm}}
	Let $T\in(0,1]$ be fixed later. Define $\varSigma:\mathfrak{E}_{M,T}\rightarrow U^2_BH^{\mathfrak{m}}[T]$ by
	\begin{equation*}\begin{split}
		\varSigma(\psi)(t)
	:=\;&
		\mathfrak{S}_B(t,0)\psi^{\mathrm{in}} + \int_0^t \mathfrak{S}_B(t,t')\bigg[
			\mathcal{Q}\left[\overline{\psi},\psi,\psi\right](t')
			+ \left(\mathcal{N}^2_0\left[\overline{\psi},\psi\right]\psi\right)(t')
	\\&
		+ \left(\mathcal{N}^4_t\left[\overline{\psi},\psi,\overline{\psi},\psi\right]\psi\right)(t')
		+ \left(\mathcal{N}^4_x\left[\overline{\psi},\psi,\overline{\psi},\psi\right]\psi\right)(t')
		- \mathrm{i}\kappa\left(|\psi|^2\psi\right)(t')
		\bigg]\,\mathrm{d}t'
	\;.\end{split}\end{equation*}
	Our goal is to show that $\varSigma$ defines a contraction map $\mathfrak{E}_{M,T}\rightarrow\mathfrak{E}_{M,T}$.
	
	Suppose $\psi_1\in\mathfrak{E}_{M,T}$.
	By Lemma \ref{ZLEM:OldLemma4.9Restated}, we have $\|\psi_1\|_{\Psi^{\mathfrak{m}}[T]}\le C(s)(1+M)^2M$ for every $\psi\in\mathfrak{E}_{M,T}$.
	Now, applying Lemmas \ref{ZLEM:MLEstmt_Q}, \ref{ZLEM:MLEstmt_N2tphi}, \ref{ZLEM:MLEstmt_N4tphi}, \ref{ZLEM:MLEstmt_N4xphi}, \ref{ZLEM:MLEstmt_phi3},
	replacing the $\Psi^1[T]$ norms by $\Psi^{\mathfrak{m}}[T]$ norms, we have
	\begin{equation*}\begin{split}
		\left\|\mathcal{Q}\left[\overline{\psi_1},\psi_1,\psi_1\right]\right\|_{\mathrm{D}U^2_BH^{\mathfrak{m}}[T]}
	\le\;&
		C(s)T^{\frac{1}{2}}\left(1+M\right)^6M^3
	\;,\\
		\left\|\mathcal{N}^2_0\left[\overline{\psi_1},\psi_1\right]\psi_1\right\|_{\mathrm{D}U^2_BH^{\mathfrak{m}}[T]}
	\le\;&
		C(s)T^{\frac{1}{2}}\left(1+M\right)^8M^3
	\;,\\
		\left\|\mathcal{N}^4_t\left[\overline{\psi_1},\psi_1\,\overline{\psi_1},\psi_1\right]\psi_1\right\|_{\mathrm{D}U^2_BH^{\mathfrak{m}}[T]}
	\le\;&
		C(s)T^{\frac{1}{2}}\left(1+M\right)^{10}M^5
	\;,\\
		\left\|\mathcal{N}^4_x\left[\overline{\psi_1},\psi_1\,\overline{\psi_1},\psi_1\right]\psi_1\right\|_{\mathrm{D}U^2_BH^{\mathfrak{m}}[T]}
	\le\;&
		C(s)T^{\frac{1}{2}}\left(1+M\right)^{10}M^5
	\;,\\
		\left\|\left|\psi_1\right|^2\psi_1\right\|_{\mathrm{D}U^2_BH^{\mathfrak{m}}[T]}
	\le\;&
		C(s)T^{\frac{1}{2}}\left(1+M\right)^8M^5
	\;.\end{split}\end{equation*}
	Summing these up, we obtain
	\[
		\left\|\varSigma(\psi_1)\right\|_{U^2_BH^{\mathfrak{m}}[T]}
	\le
		M + C(s)T^{\frac{1}{2}}\left(1+M\right)^{14}M
	\;.\]
	If $\psi_2$ is another element of $\mathfrak{E}_{M,T}$, then a similar argument shows
	\[
		\left\|\varSigma(\psi_1)-\varSigma(\psi_2)\right\|_{U^2_BH^{\mathfrak{m}}[T]}
	\le
		C(s)T^{\frac{1}{2}}\left(1+M\right)^{14}\left\|\psi_1-\psi_2\right\|_{U^2_BH^{\mathfrak{m}}[T]}
	\;.\]
	Hence, we see that by choosing $T=\delta_1(s)(1+M)^{-28}$ for sufficiently small $\delta_1=\delta_1(s)\in(0,1]$,
	we could ensure that $\varSigma$ indeed defines a contraction map $\mathfrak{E}_{M,T}\rightarrow\mathfrak{E}_{M,T}$.

	The unique fixed point $\psi\in\mathfrak{E}_{M,T}$ is then the desired solution to (\ref{ZEQN:MoreGeneralIVP}).
	Moreover, by (\ref{ZEQN:K1_defn}),
	\[
		\left\|\mathcal{N}^2_x\left[\overline{\psi},\psi\right]\right\|_{L^\infty_{t,x}[1]} \le 2K_1M^2
	\;.\]
	Thus, letting $\varGamma$ be the extension by zero of $-\frac{1}{2}\mathcal{N}^2_x[\overline{\psi},\psi]$ to $[0,1)$,
	we have that $\varGamma$ satisfies (\ref{ZEQN:Impose_B}).

	It remains to check that $\varGamma$ verifies the hypothesis (\ref{ZHYP:nablaB}),
	provided we choose $\delta_1(s)$ smaller if necessary.
	For this, we need the following estimate.

	\begin{lemma} \label{ZLEM:DispEstmt_nablaN2x}
		Let $T\in(0,1]$. Then
		\[
			\left\|\nabla\mathcal{N}^2_x\left[\psi_1,\psi_2\right]\right\|_{L^1_tL^\infty_x[T]}
			\lesssim
			T^{\frac{1}{2}}
			\left\|\psi_1\right\|_{\Psi^1[T]}
			\left\|\psi_2\right\|_{\Psi^1[T]}
		\;.\]
	\end{lemma}
	\begin{proof}
		Recalling that $\mathfrak{R}$ denotes the Riesz transform, observe that $\nabla\mathcal{N}^2_x[v_1,v_2]$ and $\mathfrak{R}^2(v_1v_2)$ have the same components.
		By Bernstein's inequality, the boundedness of the Riesz transform on $L^4_x$, and Lemma \ref{ZLEM:DispEstmt_LinftyCtrl}, we have
		\begin{equation*}\begin{split}
			\left\|\nabla\mathcal{N}^2_x\left[\mathrm{P}_\lambda\psi_1,\mathrm{P}_{\le\lambda}\psi_2\right]\right\|_{L^1_tL^\infty_x[T]}
		\lesssim\;&
			T^{\frac{1}{2}}\mu^{\frac{1}{2}}
			\left\|\mathrm{P}_\lambda\psi_1\,\mathrm{P}_{\le\lambda}\psi_2\right\|_{L^2_tL^4_x[T]}
		\\\lesssim\;&
			T^{\frac{1}{2}}\mu^{\frac{1}{2}}
			\left\|\mathrm{P}_\lambda\psi_1\right\|_{L^4_{t,x}[T]}
			\left\|\psi_2\right\|_{L^4_tL^\infty_x[T]}
		\\\lesssim\;&
			T^{\frac{1}{2}}\mu^{\frac{1}{2}}\lambda^{-\frac{3}{4}}
			\left\|\psi_1\right\|_{\Psi^1[T]}
			\left\|\psi_2\right\|_{\Psi^1[T]}
		\;.\end{split}\end{equation*}
		Summing up over $\lambda\gtrsim\mu$ and noting the symmetry of $\mathcal{N}^2_x[v_1,v_2]$ in $v_1$ and $v_2$,
		we obtain the desired estimate.
	\end{proof}

	By Lemma \ref{ZLEM:DispEstmt_nablaN2x}, noting that $\|\psi\|_{\Psi^s[T]} \le C(s)(1+M)^2M$, we have
	\[
		\left\|\varGamma\right\|_{L^1_tL^\infty_x[T]}
		\le
		C(s)T^{\frac{1}{2}}\mu^{\frac{1}{2}}\left(1+M\right)^4M^2
	\;.\]
	Hence, indeed, by choosing $\delta_1(s)$ smaller if necessary and setting $T=\delta_1(s)(1+M)^{-28}$ as before,
	we can ensure that the right-hand side is $\le 1$ and, consequently, (\ref{ZHYP:nablaB}) holds for $\varGamma$.

	The proof of Theorem \ref{ZTHM:Sect5MainThm} is complete.

\section{Convergence of the Iteration Scheme} \label{ZSECT:CnvgItrtnScheme}

	Using Theorem \ref{ZTHM:Sect5MainThm},
	we may inductively construct the iterates $\phi^{[n]}$ of the iteration scheme (\ref{ZEQN:ItrtnSchemeSuccinct}),
	initialised with $A_x^{[0]}=0$.
	For the proof of Theorem \ref{ZTHM:MainThm}, it remains
	to show that the iterates $\phi^{[n]}$ converge to a solution $\phi$ of the Chern-Simons-Schr\"{o}dinger system in the Coulomb gauge,
	(\ref{ZEQN:CSSCoulSuccinct2}),
	and to verify the $H^s$ continuity of the solution map and the weak Lipschitz estimate (\ref{ZEQN:WkLipBd}).
	
	The technical core of both tasks is that of estimating $\|\psi-\psi'\|_{L^\infty_tH^{s-1}[T]}$
	where both $\psi,\psi'$ solve (\ref{ZEQN:MoreGeneralIVP}) with possibly different admissible forms $B,B'$ respectively,
	and possibly different initial data.
	This is provided for by the following result.
	
	\begin{theorem}\label{ZTHM:Sect6MainThm}
		Given a small $\varepsilon\in(0,1]$,
		there exists $\delta_2=\delta(s,\varepsilon)\in(0,\delta_1(s)]$, where $\delta_1(s)$ is as in Theorem \ref{ZTHM:Sect5MainThm},
		such that the following is true.
	
		Let $M>0$ and suppose $B,B',B^\dagger$ are admissible forms satisfying the hypothesis (\ref{ZHYP:nablaB}) and (\ref{ZEQN:Impose_B}).
		Assume that $\mathfrak{m}(\lambda)=\lambda^s$, and $T=\delta_2(s,\varepsilon)(1+M)^{-28}$.
		Let $\psi\in U^2_BH^s[T]$ and $\psi'\in U^2_{B'}H^s[T]$ are solutions given by Theorem \ref{ZTHM:Sect5MainThm} to (\ref{ZEQN:MoreGeneralIVP}),
		with admissible forms $B,B'$ respectively,
		such that $\|\psi(0)\|_{H^s}\le M$, $\|\psi'(0)\|_{H^s}\le M$.
		Then
		\begin{equation}\label{ZEQN:Sect6MainThm_EQN01}
			\left\|\psi-\psi'\right\|_{U^2_{B^\dagger} H^{s-1}[T]}
		\le
			\varepsilon\left(
				\left\|B-B^\dagger\right\|_{L^2_tL^\infty_x[T]}
				+
				\left\|B'-B^\dagger\right\|_{L^2_tL^\infty_x[T]}
			\right)
			+
			C(s)\left\|\psi(0)-\psi'(0)\right\|_{H^{s-1}}
		\;.\end{equation}
	\end{theorem}
	
	Note that Proposition \ref{ZPROP:ChangeOfAdmForm} already guarantees that $\psi,\psi'\in U^2_{B^\dagger} H^{s-1}[T]$.
	Thus, the left-hand side of (\ref{ZEQN:Sect6MainThm_EQN01}) is finite.

	The proof of Theorem \ref{ZTHM:Sect6MainThm} is straightforward but rather labourious.
	The rest of this section will be devoted to this proof.
	
	Explicitly, the difference equation for $\psi-\psi'$ can be written
	\begin{equation}\label{ZEQN:DifferenceEqn}\begin{split}
		\left(\partial_t-\mathrm{i}\triangle+\mathfrak{P}_{B^\dagger}\right)\left(\psi-\psi'\right)
	=\;&
		\mathfrak{P}_{B^\dagger-B}\psi + \mathfrak{P}_{B'-B^\dagger}\psi'
	\\&
		+ \left(\mathcal{Q}\left[\overline{\psi},\psi,\psi\right] - \mathcal{Q}\left[\overline{\psi'},\psi',\psi'\right]\right)
	\\&
		+ \left(\mathcal{N}^2_0\left[\overline{\psi},\psi\right]\psi - \mathcal{N}^2_0\left[\overline{\psi'},\psi'\right]\psi'\right)
	\\&
		+ \left(
			\mathcal{N}^4_t\left[\overline{\psi},\psi,\overline{\psi},\psi\right]\psi
			- \mathcal{N}^4_t\left[\overline{\psi'},\psi',\overline{\psi'},\psi'\right]\psi'
		\right)
	\\&
		+ \left(
			\mathcal{N}^4_x\left[\overline{\psi},\psi,\overline{\psi},\psi\right]\psi
			- \mathcal{N}^4_x\left[\overline{\psi'},\psi',\overline{\psi'},\psi'\right]\psi'
		\right)
	\\&
		- \mathrm{i}\kappa\left(\left|\psi\right|^2\psi - \left|\psi'\right|^2\psi'\right)
	\;.\end{split}\end{equation}
	The proof of Theorem \ref{ZTHM:Sect6MainThm} proceeds in the exact same manner as that of Theorem \ref{ZTHM:Sect5MainThm}.
	We estimate the $\mathrm{D}U^2_{B^\dagger}H^{s-1}[T]$ norm of each term of the right-hand side of (\ref{ZEQN:DifferenceEqn}),
	by testing against a $V^2_{B^\dagger}H^{-(s-1)}[T]$ function and invoking the duality principle of Lemma \ref{ZLEM:UpVpDualityB}.
	
\subsection{Difference estimates for $\mathcal{N}^2_0,\mathcal{N}^2_x,\mathcal{N}^4_t$}
	First we need the following preliminary estimates, which are analogous to those of Lemmas
	\ref{ZLEM:DispEstmt_N2x}, \ref{ZLEM:DispEstmt_N2t}, \ref{ZLEM:DispEstmt_N4t}.
	Eventually, the indeterminate function $\omega$ will be substituted with $\psi-\psi'$
	or its complex conjugate.

	\begin{lemma} \label{ZLEM:DiffEstmt_N2x}
		Let $T\in (0,1]$. Then
		\begin{equation} \label{ZEQN:DiffEstmt_N2x_EQN02}
			\left\|\mathrm{P}_\mu\mathcal{N}^2_x\left[\psi_1,\omega\right]\right\|_{L^2_tL^\infty_x[T]}
			\le
			C(s)\mu^{\frac{3}{4}-s}
			\left\|\psi_1\right\|_{\Psi^s[T]}
			\left\|\omega\right\|_{\Psi^{s-1}[T]}
		\;.\end{equation}
	\end{lemma}
	\begin{proof}
		For $\mu=1$, Bernstein and Hardy-Littlewood-Sobolev give us
		\begin{equation*}\begin{split}
			\left\|\mathrm{P}_1\mathcal{N}^2_x\left[\psi_1,\omega\right]\right\|_{L^2_tL^\infty_x[T]}
		\lesssim\;&
			T^{\frac{1}{4}}\left\|\psi_1\omega\right\|_{L^4_tL^{\frac{4}{3}}[T]}
		\\\lesssim\;&
			\left\|\psi_1\right\|_{L^4_{t,x}[T]}\left\|\omega\right\|_{L^\infty_tL^2_x[T]}
		\\\le\;&
			C(s)\left\|\psi_1\right\|_{\Psi^s[T]}\left\|\omega\right\|_{\Psi^{s-1}[T]}
		\;.\end{split}\end{equation*}
		Now, suppose instead that $2\le\mu\in\mathfrak{D}$. By Bernstein,
		\begin{equation*}\begin{split}
			\left\|\mathrm{P}_\mu\mathcal{N}^2_x\left[\mathrm{P}_\lambda\psi_1,\mathrm{P}_\rho\omega\right]\right\|_{L^2_tL^\infty_x[T]}
		\lesssim\;&
			\mu^{-\frac{1}{2}}\left\|\mathrm{P}_\lambda\psi_1\,\mathrm{P}_\rho\omega\right\|_{L^2_tL^4_x[T]}
		\\\lesssim\;&
			\mu^{-\frac{1}{2}}
			\left\|\mathrm{P}_\lambda\psi_1\right\|_{L^4_tL^\infty_x[T]}
			\left\|\mathrm{P}_\rho\omega\right\|_{L^4_{t,x}[T]}
		\\\le\;&
			C(s)\mu^{-\frac{1}{2}}\lambda^{\frac{1}{4}-s}\rho^{\frac{1}{4}-(s-1)}
			\left\|\psi_1\right\|_{\Psi^s[T]}
			\left\|\omega\right\|_{\Psi^{s-1}[T]}
		\;.\end{split}\end{equation*}
		Performing the relevant summations, we obtain
		\begin{equation*}\begin{split}
			\left\|\mathrm{P}_\mu\mathcal{N}^2_x\left[\mathrm{P}_{\lesssim\mu}\psi_1,\mathrm{P}_{\approx\mu}\omega\right]\right\|_{L^2_tL^\infty_x[T]}
		+\;&
			\left\|\mathrm{P}_\mu\mathcal{N}^2_x\left[\mathrm{P}_{\approx\mu}\psi_1,\mathrm{P}_{\lesssim\mu}\omega\right]\right\|_{L^2_tL^\infty_x[T]}
		\\&\le
			C(s)\mu^{\frac{3}{4}-s}
			\left\|\psi_1\right\|_{\Psi^s[T]}
			\left\|\omega\right\|_{\Psi^{s-1}[T]}
		\;.\end{split}\end{equation*}
		By the Littlewood-Paley trichotomy, to prove (\ref{ZEQN:DiffEstmt_N2x_EQN02}) it remains to show
		\begin{equation} \label{ZEQN:DiffEstmt_N2x_EQN12}
			\sum_{\lambda\;:\;\lambda\gg\mu}
			\left\|\mathrm{P}_\mu\mathcal{N}^2_x\left[\mathrm{P}_\lambda\psi_1,\mathrm{P}_{\approx\lambda}\omega\right]\right\|_{L^2_tL^\infty_x[T]}
			\le
			C(s)\mu^{\frac{3}{4}-s}
			\left\|\psi_1\right\|_{\Psi^s[T]}
			\left\|\omega\right\|_{\Psi^{s-1}[T]}
		\;.\end{equation}
		For this, we use Bernstein to estimate
		\begin{equation*}\begin{split}
			\left\|\mathrm{P}_\mu\mathcal{N}^2_x\left[\mathrm{P}_\lambda\psi_1,\mathrm{P}_{\approx\lambda}\omega\right]\right\|_{L^2_tL^\infty_x[T]}
		\lesssim\;&
			\left\|\mathrm{P}_\lambda\psi_1\,\mathrm{P}_{\approx\lambda}\omega\right\|_{L^2_{t,x}[T]}
		\\\lesssim\;&
			\left\|\mathrm{P}_\lambda\psi_1\right\|_{L^4_{t,x}[T]}
			\left\|\mathrm{P}_{\approx\lambda}\omega\right\|_{L^4_{t,x}[T]}
		\\\le\;&
			C(s)\lambda^{\frac{1}{4}-s}\lambda^{\frac{1}{4}-(s-1)}
			\left\|\psi_1\right\|_{\Psi^s[T]}
			\left\|\omega\right\|_{\Psi^{s-1}[T]}
		\;,\end{split}\end{equation*}
		and (\ref{ZEQN:DiffEstmt_N2x_EQN12}) follows immediately.
	\end{proof}

	\begin{lemma} \label{ZLEM:DiffEstmt_N2t}
		Let $T\in (0,1]$. Then
		\begin{equation} \label{ZEQN:DiffEstmt_N2t_EQN02}
			\left\|\mathrm{P}_\mu\mathcal{N}^2_0\left[\psi_1,\omega\right]\right\|_{L^\infty_tL^1_x[T]}
			\le
			C(s)\mu^{-s}
			\left\|\psi_1\right\|_{\Psi^s[T]}
			\left\|\omega\right\|_{\Psi^{s-1}[T]}
		\quad\mbox{for }
			\mu\ge 2
		\;.\end{equation}
		We also have
		\begin{equation} \label{ZEQN:DiffEstmt_N2t_EQN03}
			\left\|\mathrm{P}_{\lesssim\mu}\mathcal{N}^2_0\left[\psi_1,\omega\right]\right\|_{L^\infty_{t,x}[T]}
			\le
			C(s)\mu
			\left\|\psi_1\right\|_{\Psi^s[T]}
			\left\|\omega\right\|_{\Psi^{s-1}[T]}
		\;.\end{equation}
	\end{lemma}
	\begin{proof}
		We first observe the preliminary estimate
		\begin{equation} \label{ZEQN:DiffEstmt_N2t_EQN11}
			\left\|\mathrm{P}_\mu\left(\nabla\mathrm{P}_{\gg\mu}\psi_1\,\mathrm{P}_{\gg\mu}\omega\right)\right\|_{L^1_x}
			\le
			C(s)\mu^{-2(s-1)}
			\left\|\psi_1\right\|_{H^s}
			\left\|\omega\right\|_{H^{s-1}}
		\;.\end{equation}
		Indeed, since $s\ge 1$, for $\lambda\gtrsim\mu$ we have
		\[
			\left\|\nabla\mathrm{P}_\lambda\psi_1\,\mathrm{P}_{\approx\lambda}\omega\right\|_{L^1_x}
			\le
			C(s)
			\mu^{-(2s-2)}\lambda^{2s-1}
			\left\|\mathrm{P}_\lambda\psi_1\right\|_{L^2_x}
			\left\|\mathrm{P}_{\approx\lambda}\omega\right\|_{L^2_x}
		\;.\]
		Summing over $\lambda\gg\mu$, (\ref{ZEQN:DiffEstmt_N2t_EQN11}) follows using Cauchy-Schwarz.

		We turn to the proof of (\ref{ZEQN:DiffEstmt_N2t_EQN02}).
		Let $\mu\ge 2$ be fixed. We have
		\begin{equation*}\begin{split}
			\left\|\mathrm{P}_\mu\mathcal{N}^2_0\left[\mathrm{P}_{\approx\mu}\psi_1,\mathrm{P}_{\lesssim\mu}\omega\right]\right\|_{L^\infty_tL^1_x[T]}
		\lesssim\;&
			\mu^{-1}\left\|\nabla\mathrm{P}_{\approx\mu}\psi_1\,\mathrm{P}_{\lesssim\mu}\omega\right\|_{L^\infty_tL^1_x[T]}
		\\\le\;&
			C(s)\mu^{-s}
			\left\|\psi_1\right\|_{L^\infty_tH^s[T]}
			\left\|\omega\right\|_{L^\infty_tH^{s-1}[T]}
		\end{split}\end{equation*}
		and similarly
		\[
			\left\|\mathrm{P}_\mu\mathcal{N}^2_0\left[\mathrm{P}_{\lesssim\mu}\psi_1,\mathrm{P}_{\approx\mu}\omega\right]\right\|_{L^\infty_tL^1_x[T]}
		\lesssim
			C(s)\mu^{-s}
			\left\|\psi_1\right\|_{L^\infty_tH^s[T]}
			\left\|\omega\right\|_{L^\infty_tH^{s-1}[T]}
		\;.\]
		On the other hand, by (\ref{ZEQN:DiffEstmt_N2t_EQN11}),
		\[
			\left\|\mathrm{P}_\mu\mathcal{N}^2_0\left[\mathrm{P}_{\gg\mu}\psi_1,\mathrm{P}_{\gg\mu}\omega\right]\right\|_{L^\infty_tL^1_x[T]}
			\le
			C(s)\mu^{1-2s}
			\left\|\psi_1\right\|_{L^\infty_tH^s[T]}
			\left\|\omega\right\|_{L^\infty_tH^{s-1}[T]}
		\;.\]
		Hence, due to the Littlewood-Paley trichotomy, we obtain (\ref{ZEQN:DiffEstmt_N2t_EQN02}).

		We now prove (\ref{ZEQN:DiffEstmt_N2t_EQN03}). By Bernstein, Hardy-Littlewood-Sobolev and H\"{o}lder, we have
		\[
			\left\|\mathrm{P}_1\mathcal{N}^2_0\left[\psi_1,\omega\right]\right\|_{L^\infty_{t,x}[T]}
			\lesssim
			\left\|\nabla\psi_1\,\omega\right\|_{L^\infty_tL^1_x[T]}
			\le
			C(s)
			\left\|\psi_1\right\|_{L^\infty_tH^s[T]}
			\left\|\omega\right\|_{L^\infty_tH^{s-1}[T]}
		\;.\]
		For $\nu\ge 2$, Bernstein's inequality and (\ref{ZEQN:DiffEstmt_N2t_EQN02}) give
		\[
			\left\|\mathrm{P}_\nu\mathcal{N}^2_0\left[\psi_1,\omega\right]\right\|_{L^\infty_{t,x}[T]}
			\le
			C(s)\nu
			\left\|\psi_1\right\|_{\Psi^s[T]}
			\left\|\omega\right\|_{\Psi^{s-1}[T]}
		\]
		since $2-s\le 1$.
		Hence, (\ref{ZEQN:DiffEstmt_N2t_EQN03}) follows by summing the preceding estimates.
	\end{proof}

	\begin{lemma} \label{ZLEM:DiffEstmt_N4t}
		Let $T\in (0,1]$. Then
		\begin{equation} \label{ZEQN:DiffEstmt_N4t_EQN01}
			\left\|\mathrm{P}_\mu\mathcal{N}^4_t\left[\psi_1,\omega,\psi_2,\psi_3\right]\right\|_{L^2_tL^\infty_x[T]}
			\le
			C(s)\mu^{\frac{3}{4}-s}
			\left\|\omega\right\|_{\Psi^{s-1}[T]}
			\prod_{\ell=1}^3
			\left\|\psi_\ell\right\|_{\Psi^s[T]}
		\end{equation}
		and
		\begin{equation} \label{ZEQN:DiffEstmt_N4t_EQN02}
			\left\|\mathrm{P}_\mu\mathcal{N}^4_t\left[\psi_1,\psi_2,\psi_3,\omega\right]\right\|_{L^2_tL^\infty_x[T]}
			\le
			C(s)\mu^{\frac{3}{4}-s}
			\left\|\omega\right\|_{\Psi^{s-1}[T]}
			\prod_{\ell=1}^3
			\left\|\psi_\ell\right\|_{\Psi^s[T]}
		\end{equation}
	\end{lemma}
	\begin{proof}
		We first prove (\ref{ZEQN:DiffEstmt_N4t_EQN01}) in the case $\mu=1$.
		By the Bernstein, Hardy-Littlewood-Sobolev and H\"{o}lder inequalities, we have
		\begin{equation*}\begin{split}
			\left\|\mathrm{P}_1\mathcal{N}^4_t\left[\psi_1,\omega,\psi_2,\psi_3\right]\right\|_{L^2_tL^\infty_x[T]}
		\lesssim\;&
			\left\|\mathcal{N}^2_x\left[\psi_1,\omega\right]\right\|_{L^2_tL^\infty_x[T]}
			\left\|\psi_2\right\|_{L^\infty_tL^2_x[T]}
			\left\|\psi_3\right\|_{L^\infty_tL^2_x[T]}
		\\\le\;&
			C(s)\left\|\omega\right\|_{\Psi^{s-1}[T]}
			\prod_{\ell=1}^3
			\left\|\psi_\ell\right\|_{\Psi^s[T]}
		\end{split}\end{equation*}
		where we have used (\ref{ZEQN:DiffEstmt_N2x_EQN02}) to estimate $\|\mathcal{N}^2_x[\psi_1,\omega]\|_{L^2_tL^\infty_x[T]}$.

		Suppose now $2\le\mu\in\mathfrak{D}$. Then, by Bernstein,
		\[
			\left\|\mathrm{P}_\mu\mathcal{N}^4_t\left[\psi_1,\omega,\psi_2,\psi_3\right]\right\|_{L^2_tL^\infty_x[T]}
			\lesssim
			\mu^{-1}
			\left\|\mathrm{P}_\mu\left(\mathcal{N}^2_x\left[\psi_1,\omega\right]\psi_2\psi_3\right)\right\|_{L^2_tL^\infty_x[T]}
		\;.\]
		By Bernstein, H\"{o}lder and (\ref{ZEQN:DiffEstmt_N2x_EQN02}),
		\begin{equation*}\begin{split}
			\mu^{-1}
			\left\|\mathrm{P}_\mu\left(\mathrm{P}_{\gtrsim\mu}\mathcal{N}^2_x\left[\psi_1,\omega\right]\,\psi_2\psi_3\right)\right\|_{L^2_tL^\infty_x[T]}
		\lesssim\;&
			\left\|\mathrm{P}_{\gtrsim\mu}\mathcal{N}^2_x\left[\psi_1,\omega\right]\right\|_{L^2_tL^\infty_x[T]}
			\left\|\psi_2\right\|_{L^\infty_tL^4_x[T]}
			\left\|\psi_3\right\|_{L^\infty_tL^4_x[T]}
		\\\le\;&
			C(s)
			\mu^{\frac{3}{4}-s}
			\left\|\omega\right\|_{\Psi^{s-1}[T]}
			\prod_{\ell=1}^3
			\left\|\psi_\ell\right\|_{\Psi^s[T]}
		\;.\end{split}\end{equation*}
		On the other hand, for $\lambda\gtrsim\mu$, we have
		\begin{equation*}\begin{split}
			\mu^{-1}\Big\|\mathrm{P}_\mu
		&
			\left(\mathrm{P}_{\ll\mu}\mathcal{N}^2_x\left[\psi_1,\omega\right]\,\mathrm{P}_\lambda\psi_2\,\mathrm{P}_{\le\lambda}\psi_3\right)\Big\|_{L^2_tL^\infty_x[T]}
		\\\lesssim\;&
			\left\|\mathcal{N}^2_x\left[\psi_1,\omega\right]\right\|_{L^2_tL^\infty_x[T]}
			\left\|\mathrm{P}_\lambda\psi_2\right\|_{L^\infty_tL^4_x[T]}
			\left\|\psi_3\right\|_{L^\infty_tL^4_x[T]}
		\\\le\;&
			C(s)
			\lambda^{\frac{1}{2}-s}
			\left\|\omega\right\|_{\Psi^{s-1}[T]}
			\prod_{\ell=1}^3
			\left\|\psi_\ell\right\|_{\Psi^s[T]}
		\;.\end{split}\end{equation*}
		By summing over $\lambda\gtrsim\mu$ and noting the symmetry in $\psi_2,\psi_3$, we have
		\[
			\mu^{-1}
			\left\|\mathrm{P}_\mu\left(\mathrm{P}_{\ll\mu}\mathcal{N}^2_x\left[\psi_1,\omega\right]\,\psi_2\psi_3\right)\right\|_{L^2_tL^\infty_x[T]}
			\le
			C(s)
			\mu^{\frac{1}{2}-s}
			\left\|\omega\right\|_{\Psi^{s-1}[T]}
			\prod_{\ell=1}^3
			\left\|\psi_\ell\right\|_{\Psi^s[T]}
		\;.\]
		This completes the proof of (\ref{ZEQN:DiffEstmt_N4t_EQN01}).

		We turn to the proof of (\ref{ZEQN:DiffEstmt_N4t_EQN02}).
		The case $\mu=1$ is handled in exactly the same fashion as above.
		Suppose now $2\le\mu\in\mathfrak{D}$.
		Then, by Bernstein,
		\[
			\left\|\mathrm{P}_\mu\mathcal{N}^4_t\left[\psi_1,\psi_2,\psi_3,\omega\right]\right\|_{L^2_tL^\infty_x[T]}
			\lesssim
			\mu^{-1}
			\left\|\mathrm{P}_\mu\left(\mathcal{N}^2_x\left[\psi_1,\psi_2\right]\psi_3\omega\right)\right\|_{L^2_tL^\infty_x[T]}
		\;.\]
		By Bernstein, H\"{o}lder and (\ref{ZEQN:DispEstmt_N2x_EQN02}),
		\begin{equation*}\begin{split}
			\mu^{-1}
			\left\|\mathrm{P}_\mu\left(\mathrm{P}_{\gtrsim\mu}\mathcal{N}^2_x\left[\psi_1,\psi_2\right]\,\psi_3\omega\right)\right\|_{L^2_tL^\infty_x[T]}
		\lesssim\;&
			\mu^{\frac{1}{2}}
			\left\|\mathrm{P}_{\gtrsim\mu}\mathcal{N}^2_x\left[\psi_1,\psi_2\right]\right\|_{L^2_tL^\infty_x[T]}
			\left\|\psi_3\right\|_{L^\infty_tL^4_x[T]}
			\left\|\omega\right\|_{L^\infty_tL^2_x[T]}
		\\\le\;&
			C(s)
			\mu^{\frac{1}{4}-s}
			\left\|\omega\right\|_{\Psi^{s-1}[T]}
			\prod_{\ell=1}^3
			\left\|\psi_\ell\right\|_{\Psi^s[T]}
		\;.\end{split}\end{equation*}
		On the other hand, by Bernstein, H\"{o}lder and (\ref{ZEQN:EnergyEstmt_N2x_EQN02}) we have
		\begin{equation*}\begin{split}
			\mu^{-1}\Big\|\mathrm{P}_\mu
		&
			\left(\mathrm{P}_{\ll\mu}\mathcal{N}^2_x\left[\psi_1,\psi_2\right]\,\mathrm{P}_\lambda\psi_3\,\mathrm{P}_{\lesssim\lambda}\omega\right)\Big\|_{L^2_tL^\infty_x[T]}
		\\\lesssim\;&
			\left\|\mathcal{N}^2_x\left[\psi_1,\psi_2\right]\right\|_{L^\infty_{t,x}[T]}
			\left\|\mathrm{P}_\lambda\psi_3\right\|_{L^4_tL^\infty_x[T]}
			\left\|\omega\right\|_{L^\infty_tL^2_x[T]}
		\\\le\;&
			C(s)
			\lambda^{\frac{3}{4}-s}
			\left\|\omega\right\|_{\Psi^{s-1}[T]}
			\prod_{\ell=1}^3
			\left\|\psi_\ell\right\|_{\Psi^s[T]}
		\end{split}\end{equation*}
		which is summable over $\lambda\gtrsim\mu$; and also we have
		\begin{equation*}\begin{split}
			\mu^{-1}\Big\|\mathrm{P}_\mu
		&
			\left(\mathrm{P}_{\ll\mu}\mathcal{N}^2_x\left[\psi_1,\psi_2\right]\,\mathrm{P}_{\ll\mu}\psi_3\,\mathrm{P}_{\approx\mu}\omega\right)\Big\|_{L^2_tL^\infty_x[T]}
		\\\lesssim\;&
			\mu^{-1}
			\left\|\mathcal{N}^2_x\left[\psi_1,\psi_2\right]\right\|_{L^\infty_{t,x}[T]}
			\left\|\psi_3\right\|_{L^4_tL^\infty_x[T]}
			\left\|\omega\right\|_{L^4_tL^\infty_x[T]}
		\\\le\;&
			C(s)
			\mu^{\frac{1}{4}-s}
			\left\|\omega\right\|_{\Psi^{s-1}[T]}
			\prod_{\ell=1}^3
			\left\|\psi_\ell\right\|_{\Psi^s[T]}
		\;.\end{split}\end{equation*}
		By the Littlewood-Paley trichotomy, (\ref{ZEQN:DiffEstmt_N4t_EQN02}) is proved.
	\end{proof}
	
\subsection{Difference estimates for nonlinearities}
	We are now ready to estimate the $\mathrm{D}U^2_{B^\dagger}H^{s-1}[T]$ norm of each term of the right-hand side of (\ref{ZEQN:DifferenceEqn}).

	\begin{lemma} \label{ZLEM:DiffEstmt_PBphi}
		Assume the hypothesis (\ref{ZHYP:nablaB}).
		Let $T\in(0,1]$.
		Let $\varTheta$ be an admissible form.
		Then
		\[
			\left\|\mathfrak{P}_\varTheta\psi_1\right\|_{\mathrm{D}U^2_{B^\dagger}H^{s-1}[T]}
			\le
			C(s)
			T^{\frac{1}{2}}
			\left\|\varTheta\right\|_{L^2_tL^\infty_x[T]}
			\left\|\psi_1\right\|_{\Psi^s[T]}
		\;.\]
	\end{lemma}
	\begin{proof}
		By duality, it suffices to prove
		\begin{equation}\label{ZEQN:DiffEstmt_PBphi_EQN11}
			\left|\int_0^T\int_{\mathbb{R}^2_x}
				\overline{\omega_0}\,
				\mathfrak{P}_\varTheta\psi_1
			\,\mathrm{d}x\,\mathrm{d}t\right|
			\le
			C(s)
			T^{\frac{1}{2}}
			\left\|\omega_0\right\|_{V^2_{B^\dagger}H^{-(s-1)}[T]}
			\left\|\varTheta\right\|_{L^2_tL^\infty_x[T]}
			\left\|\psi_1\right\|_{\Psi^s[T]}
		\;.\end{equation}
		By H\"{o}lder and Bernstein, we have
		\begin{equation*}\begin{split}
			\left|\int_0^T\int_{\mathbb{R}^2_x}
				\overline{\omega_0}\,
				\mathfrak{P}_\varTheta\psi_1
			\,\mathrm{d}x\,\mathrm{d}t\right|
		\lesssim\;&
			\int_0^T
			\sum_\nu
				\left\|\mathrm{P}_\nu\omega_0(t)\right\|_{L^2_x}
				\left\|\varTheta(t)\right\|_{L^\infty_x}
				\nu\left\|\mathrm{P}_{\approx\nu}\psi_1(t)\right\|_{L^2_x}
			\,\mathrm{d}t
		\\\le\;&
			C(s)
			T^{\frac{1}{2}}
			\left\|\omega_0(t)\right\|_{L^\infty_tH^{-(s-1)}[T]}
			\left\|\varTheta\right\|_{L^2_tL^\infty_x[T]}
			\left\|\psi_1\right\|_{L^\infty_tH^s[T]}
		\;.\end{split}\end{equation*}
		Now use Lemma \ref{ZLEM:OldLemma4.9} to replace the $L^\infty_tH^{-(s-1)}[T]$
		norm of $\omega_0$ by the $V^2_{B^\dagger}H^{-(s-1)}[T]$.
		We thus obtain (\ref{ZEQN:DiffEstmt_PBphi_EQN11}).
	\end{proof}
	
	\begin{lemma} \label{ZLEM:DiffEstmt_Q}
		Assume the hypothesis (\ref{ZHYP:nablaB}).
		Let $T\in(0,1]$.
		Then we have
		\begin{equation}\label{ZEQN:DiffEstmt_Q_EQN01}
			\left\|\mathcal{Q}\left[\psi_1,\omega,\psi_2\right]\right\|_{\mathrm{D}U^2_{B^\dagger}H^{s-1}[T]}
			\le
			C(s)
			T^{\frac{1}{2}}
			\left\|\omega\right\|_{\Psi^{s-1}[T]}
			\prod_{\ell=1}^2
			\left\|\psi_\ell\right\|_{\Psi^s[T]}
		\end{equation}
		and
		\begin{equation}\label{ZEQN:DiffEstmt_Q_EQN02}
			\left\|\mathcal{Q}\left[\psi_1,\psi_2,\omega\right]\right\|_{\mathrm{D}U^2_{B^\dagger}H^{s-1}[T]}
			\le
			C(s)
			T^{\frac{1}{2}}
			\left\|\omega\right\|_{\Psi^{s-1}[T]}
			\prod_{\ell=1}^2
			\left\|\psi_\ell\right\|_{\Psi^s[T]}
		\end{equation}
	\end{lemma}
	\begin{proof}
		By duality, the proof of (\ref{ZEQN:DiffEstmt_Q_EQN01}) reduces to verifying the estimate
		\begin{equation}\label{ZEQN:DiffEstmt_Q_EQN11}
			\bigg|\int_0^T\int_{\mathbb{R}^2_x}
				\overline{\omega_0}\,
				\mathcal{Q}\left[\psi_1,\omega,\psi_2\right]
			\,\mathrm{d}x\,\mathrm{d}t\bigg|
			\le
			C(s)
			T^{\frac{1}{2}}
			\left\|\omega_0\right\|_{V^2_{B^\dagger}H^{-(s-1)}[T]}
			\left\|\omega\right\|_{\Psi^{s-1}[T]}
			\prod_{\ell=1}^2
			\left\|\psi_\ell\right\|_{\Psi^s[T]}
		\;.\end{equation}
		By Lemma \ref{ZLEM:DiffEstmt_N2x}, we have
		\begin{equation*}\begin{split}
			\bigg|\int_0^T\int_{\mathbb{R}^2_x}
		&
			\mathrm{P}_\nu\overline{\omega_0}
			\,\mathrm{P}_\mu\mathcal{N}^2_x\left[\psi_1,\omega\right]
			\cdot\nabla\mathrm{P}_\lambda\psi_2
			\,\mathrm{d}x\,\mathrm{d}t\bigg|
		\\\lesssim\;&
			T^{\frac{1}{2}}
			\left\|\mathrm{P}_\nu\omega_0\right\|_{L^\infty_tL^2_x[T]}
			\left\|\mathrm{P}_\mu\mathcal{N}^2_x\left[\psi_1,\omega\right]\right\|_{L^2_tL^\infty_x[T]}
			\lambda\left\|\mathrm{P}_\lambda\psi_2\right\|_{L^\infty_tL^2_x[T]}
		\\\le\;&
			C(s)
			T^{\frac{1}{2}}
			\nu^{s-1}\mu^{\frac{3}{4}-s}\lambda^{1-s}
			\left\|\omega_0\right\|_{V^2_{B^\dagger}H^{-(s-1)}[T]}
			\left\|\omega\right\|_{\Psi^{s-1}[T]}
			\prod_{\ell=1}^2
			\left\|\psi_\ell\right\|_{\Psi^s[T]}
		\;.\end{split}\end{equation*}
		The right-hand side is summable over $\{\mu\gtrsim\max(\lambda,\nu)\}$.
		This gives (\ref{ZEQN:DiffEstmt_Q_EQN01}).
		
		We now turn to proving (\ref{ZEQN:DiffEstmt_Q_EQN02}).
		By duality, it suffices to prove
		\begin{equation}\label{ZEQN:DiffEstmt_Q_EQN21}
			\bigg|\int_0^T\int_{\mathbb{R}^2_x}
				\overline{\omega_0}\,
				\mathcal{Q}\left[\psi_1,\psi_2,\omega\right]
			\,\mathrm{d}x\,\mathrm{d}t\bigg|
			\le
			C(s)
			T^{\frac{1}{2}}
			\left\|\omega_0\right\|_{V^2_{B^\dagger}H^{-(s-1)}[T]}
			\left\|\omega\right\|_{\Psi^{s-1}[T]}
			\prod_{\ell=1}^2
			\left\|\psi_\ell\right\|_{\Psi^s[T]}
		\;.\end{equation}
		By Lemma \ref{ZLEM:DispEstmt_N2x},
		\begin{equation*}\begin{split}
			\bigg|\int_0^T\int_{\mathbb{R}^2_x}
		&
			\mathrm{P}_\nu\overline{\omega_0}
			\,\mathrm{P}_\mu\mathcal{N}^2_x\left[\psi_1,\psi_2\right]
			\cdot\nabla\mathrm{P}_\lambda\omega
			\,\mathrm{d}x\,\mathrm{d}t\bigg|
		\\\lesssim\;&
			T^{\frac{1}{2}}
			\left\|\mathrm{P}_\nu\omega_0\right\|_{L^\infty_tL^2_x[T]}
			\left\|\mathrm{P}_\mu\mathcal{N}^2_x\left[\psi_1,\psi_2\right]\right\|_{L^2_tL^\infty_x[T]}
			\lambda\left\|\mathrm{P}_\lambda\omega\right\|_{L^\infty_tL^2_x[T]}
		\\\le\;&
			C(s)
			T^{\frac{1}{2}}
			\nu^{s-1}\mu^{-\frac{3}{4}-s}\lambda^{2-s}
			\left\|\omega_0\right\|_{V^2_{B^\dagger}H^{-(s-1)}[T]}
			\left\|\omega\right\|_{\Psi^{s-1}[T]}
			\prod_{\ell=1}^2
			\left\|\psi_\ell\right\|_{\Psi^s[T]}
		\;.\end{split}\end{equation*}
		The right-hand side is summable over $\{\mu\gtrsim\max(\lambda,\nu)\}$ to give (\ref{ZEQN:DiffEstmt_Q_EQN21}).
	\end{proof}

	\begin{lemma} \label{ZLEM:DiffEstmt_N2tphi}
		Assume the hypothesis (\ref{ZHYP:nablaB}).
		Assume also that the admissible form $B$ satisfies (\ref{ZEQN:Impose_B}).
		Let $T\in(0,1]$.
		Then we have
		\begin{equation}\label{ZEQN:DiffEstmt_N2tphi_EQN01}
			\left\|\mathcal{N}^2_0\left[\psi_1,\omega\right]\psi_2\right\|_{\mathrm{D}U^2_{B^\dagger}H^{s-1}[T]}
			\le
			C(s)\left(1+M\right)^2
			T^{\frac{1}{2}}
			\left\|\omega\right\|_{\Psi^{s-1}[T]}
			\prod_{\ell=1}^2
			\left\|\psi_\ell\right\|_{\Psi^s[T]}
		\end{equation}
		and
		\begin{equation}\label{ZEQN:DiffEstmt_N2tphi_EQN02}
			\left\|\mathcal{N}^2_0\left[\psi_1,\psi_2\right]\omega\right\|_{\mathrm{D}U^2_{B^\dagger}H^{s-1}[T]}
			\le
			C(s)\left(1+M\right)^2
			T^{\frac{1}{2}}
			\left\|\omega\right\|_{\Psi^{s-1}[T]}
			\prod_{\ell=1}^2
			\left\|\psi_\ell\right\|_{\Psi^s[T]}
		\;.\end{equation}
	\end{lemma}
	\begin{proof}
		By duality, the proof of (\ref{ZEQN:DiffEstmt_N2tphi_EQN01}) reduces to verifying the estimate
		\begin{equation}\label{ZEQN:DiffEstmt_N2tphi_EQN11}\begin{split}
			\bigg|\int_0^T\int_{\mathbb{R}^2}
		&
			\overline{\omega_0}\,\mathcal{N}^2_0\left[\psi_1,\omega\right]\psi_2\,\mathrm{d}x\,\mathrm{d}t\bigg|
		\\&\le
			C(s)\left(1+M\right)^2
			T^{\frac{1}{2}}
			\left\|\omega_0\right\|_{V^2_{B^\dagger}H^{-(s-1)}[T]}
			\left\|\omega\right\|_{\Psi^{s-1}[T]}
			\prod_{\ell=1}^2
			\left\|\psi_\ell\right\|_{\Psi^s[T]}
		\;.\end{split}\end{equation}
		Using (\ref{ZEQN:DiffEstmt_N2t_EQN03}) we have
		\begin{equation*}\begin{split}
			\sum_{\nu}\bigg|\int_0^T\int_{\mathbb{R}^2}
		&
			\mathrm{P}_\nu\overline{\omega_0}\,\mathrm{P}_{\lesssim\nu}\mathcal{N}^2_0\left[\psi_1,\omega\right]\,
			\mathrm{P}_{\approx\nu}\psi_2\,\mathrm{d}x\,\mathrm{d}t\bigg|
		\\\lesssim\;&
			\int_0^T 
				\left\|\mathcal{N}^2_0\left[\psi_1,\omega\right](t)\right\|_{L^\infty_x}
			\sum_\nu
				\left\|\mathrm{P}_\nu\omega_0(t)\right\|_{L^2_x}
				\left\|\mathrm{P}_{\approx\nu}\psi_2(t)\right\|_{L^2_x}
			\mathrm{d}t
		\\\le\;&
			C(s)
			T
			\left\|\omega_0\right\|_{L^\infty_tH^{-(s-1)}[T]}
			\left\|\psi_1\right\|_{\Psi^s[T]}
			\left\|\omega\right\|_{\Psi^{s-1}[T]}
			\left\|\psi_2\right\|_{L^\infty_tH^s[T]}
		\\\le\;&
			C(s)
			T^{\frac{1}{2}}
			\left\|\omega_0\right\|_{V^2_{B^\dagger}H^{-(s-1)}[T]}
			\left\|\omega\right\|_{\Psi^{s-1}[T]}
			\prod_{\ell=1}^2
			\left\|\psi_\ell\right\|_{\Psi^s[T]}
		\;.\end{split}\end{equation*}
		By the Littlewood-Paley trichotomy, to prove (\ref{ZEQN:DiffEstmt_N2tphi_EQN11}) it remains to prove
		\begin{equation}\label{ZEQN:DiffEstmt_N2tphi_EQN12}
			\left(\sum_{\mu\approx\lambda\gg\nu} + \sum_{\mu\approx\nu\gg\lambda}\right)
			\bigg|\int_0^T\int_{\mathbb{R}^2}
			\mathrm{P}_\nu\overline{\omega_0}\,\mathrm{P}_\mu\mathcal{N}^2_0\left[\psi_1,\omega\right]\,
			\mathrm{P}_\lambda\psi_2\,\mathrm{d}x\,\mathrm{d}t\bigg|
		\le
			\mbox{RHS(\ref{ZEQN:DiffEstmt_N2tphi_EQN11})}
		\;.\end{equation}
		For this, we have from (\ref{ZEQN:DiffEstmt_N2t_EQN02}) and the hypotheses that, for $\mu\ge 2$,
		\begin{equation*}\begin{split}
			\bigg|\int_0^T\int_{\mathbb{R}^2}
		&
			\mathrm{P}_\nu\overline{\omega_0}\,\mathrm{P}_\mu\mathcal{N}^2_0\left[\psi_1,\omega\right]\,
			\mathrm{P}_\lambda\psi_2\,\mathrm{d}x\,\mathrm{d}t\bigg|
		\\\le\;&
			C(s)
			T^{\frac{1}{2}}
			\left\|\mathrm{P}_\nu\omega_0\right\|_{L^4_tL^\infty_x[T]}
			\left\|\mathrm{P}_\mu\mathcal{N}^2_0\left[\psi_1,\omega\right]\right\|_{L^\infty_tL^1_x[T]}
			\left\|\mathrm{P}_\lambda\psi_2\right\|_{L^4_tL^\infty_x[T]}
		\\\le\;&
			C(s)
			T^{\frac{1}{2}}
			\left(1+M\right)^2
				\nu^{\frac{3}{4}+(s-1)}
				\mu^{-s}
				\lambda^{\frac{3}{4}-s}
			\left\|\omega_0\right\|_{V^2_{B^\dagger}H^{-(s-1)}[T]}
			\left\|\omega\right\|_{\Psi^{s-1}[T]}
			\prod_{\ell=1}^2
				\left\|\psi_\ell\right\|_{\Psi^s[T]}
		\;.\end{split}\end{equation*}
		Therefore we have (\ref{ZEQN:DiffEstmt_N2tphi_EQN12}). Hence we have proved (\ref{ZEQN:DiffEstmt_N2tphi_EQN01}).

		The proof of (\ref{ZEQN:DiffEstmt_N2tphi_EQN02}) is similar.
		By duality, it suffices to verify the estimate
		\begin{equation}\label{ZEQN:DiffEstmt_N2tphi_EQN21}\begin{split}
			\bigg|\int_0^T\int_{\mathbb{R}^2}
		&
			\overline{\omega_0}\,\mathcal{N}^2_0\left[\psi_1,\psi_2\right]\omega\,\mathrm{d}x\,\mathrm{d}t\bigg|
		\\&\le
			C(s)\left(1+M\right)^2
			T^{\frac{1}{2}}
			\left\|\omega_0\right\|_{V^2_{B^\dagger}H^{-(s-1)}[T]}
			\left\|\omega\right\|_{\Psi^{s-1}[T]}
			\prod_{\ell=1}^2
			\left\|\psi_\ell\right\|_{\Psi^s[T]}
		\;.\end{split}\end{equation}
		Indeed, using (\ref{ZEQN:DispEstmt_N2t_EQN06}) and arguing as above, we obtain
		\begin{equation*}\begin{split}
			\sum_{\nu}\bigg|\int_0^T\int_{\mathbb{R}^2}
		&
			\mathrm{P}_\nu\overline{\omega_0}\,\mathrm{P}_{\lesssim\nu}\mathcal{N}^2_0\left[\psi_1,\psi_2\right]\,
			\mathrm{P}_{\approx\nu}\omega\,\mathrm{d}x\,\mathrm{d}t\bigg|
		\\\le\;&
			C(s)
			T^{\frac{3}{4}}
			\left\|\omega_0\right\|_{V^2_{B^\dagger}H^{-(s-1)}[T]}
			\left\|\omega\right\|_{\Psi^{s-1}[T]}
			\prod_{\ell=1}^2
			\left\|\psi_\ell\right\|_{\Psi^s[T]}
		\;.\end{split}\end{equation*}
		By the Littlewood-Paley trichotomy, to prove (\ref{ZEQN:DiffEstmt_N2tphi_EQN21}) it remains to prove
		\begin{equation}\label{ZEQN:DiffEstmt_N2tphi_EQN22}
			\left(\sum_{\mu\approx\lambda\gg\nu} + \sum_{\mu\approx\nu\gg\lambda}\right)
			\bigg|\int_0^T\int_{\mathbb{R}^2}
			\mathrm{P}_\nu\overline{\omega_0}\,\mathrm{P}_\mu\mathcal{N}^2_0\left[\psi_1,\psi_2\right]\,
			\mathrm{P}_\lambda\omega\,\mathrm{d}x\,\mathrm{d}t\bigg|
		\le
			\mbox{RHS(\ref{ZEQN:DiffEstmt_N2tphi_EQN21})}
		\;.\end{equation}
		Arguing as before, we have from (\ref{ZEQN:DispEstmt_N2t_EQN05}) and the hypotheses that, for $\mu\ge 2$,
		\begin{equation*}\begin{split}
			\bigg|\int_0^T\int_{\mathbb{R}^2}
		&
			\mathrm{P}_\nu\overline{\omega_0}\,\mathrm{P}_\mu\mathcal{N}^2_0\left[\psi_1,\psi_2\right]\,
			\mathrm{P}_\lambda\omega\,\mathrm{d}x\,\mathrm{d}t\bigg|
		\\\lesssim\;&
			T^{\frac{1}{2}}
			\left\|\mathrm{P}_\nu\omega_0\right\|_{L^4_tL^\infty_x[T]}
			\left\|\mathrm{P}_\mu\mathcal{N}^2_0\left[\psi_1,\psi_2\right]\right\|_{L^\infty_tL^1_x[T]}
			\left\|\mathrm{P}_\lambda\omega\right\|_{L^4_tL^\infty_x[T]}
		\\\le\;&
			C(s)
			T^{\frac{1}{2}}
			\left(1+M\right)^2
				\nu^{-\frac{1}{4}+s}
				\mu^{-s-1}
				\lambda^{\frac{7}{4}-s}
			\left\|\omega_0\right\|_{V^2_{B^\dagger}H^{-(s-1)}[T]}
			\left\|\omega\right\|_{U^2_BH^{s-1}[T]}
			\prod_{\ell=1}^2
				\left\|\psi_\ell\right\|_{U^2_BH^s[T]}
		\;.\end{split}\end{equation*}
		Thus (\ref{ZEQN:DiffEstmt_N2tphi_EQN22}) is immediate, and we have completed the proof of (\ref{ZEQN:DiffEstmt_N2tphi_EQN02}).
	\end{proof}

	\begin{lemma} \label{ZLEM:DiffEstmt_N4tphi}
		Assume the hypothesis (\ref{ZHYP:nablaB}).
		Let $T\in(0,1]$.
		Then we have the estimates
		\begin{equation}\label{ZEQN:DiffEstmt_N4tphi_EQN01}
			\left\|\mathcal{N}^4_t\left[\psi_1,\omega,\psi_2,\psi_3\right]\psi_4\right\|_{\mathrm{D}U^2_{B^\dagger}H^{s-1}[T]}
			\le
			C(s)
			T^{\frac{1}{2}}
			\left\|\omega\right\|_{\Psi^{s-1}[T]}
			\prod_{\ell=1}^4
			\left\|\psi_\ell\right\|_{\Psi^s[T]}
		\;,\end{equation}
		\begin{equation}\label{ZEQN:DiffEstmt_N4tphi_EQN02}
			\left\|\mathcal{N}^4_t\left[\psi_1,\psi_2,\psi_3,\omega\right]\psi_4\right\|_{\mathrm{D}U^2_{B^\dagger}H^{s-1}[T]}
			\le
			C(s)
			T^{\frac{1}{2}}
			\left\|\omega\right\|_{\Psi^{s-1}[T]}
			\prod_{\ell=1}^4
			\left\|\psi_\ell\right\|_{\Psi^s[T]}
		\;,\end{equation}
		\begin{equation}\label{ZEQN:DiffEstmt_N4tphi_EQN03}
			\left\|\mathcal{N}^4_t\left[\psi_1,\psi_2,\psi_3,\psi_4\right]\omega\right\|_{\mathrm{D}U^2_{B^\dagger}H^{s-1}[T]}
			\le
			C(s)
			T^{\frac{1}{2}}
			\left\|\omega\right\|_{\Psi^{s-1}[T]}
			\prod_{\ell=1}^4
			\left\|\psi_\ell\right\|_{\Psi^s[T]}
		\;.\end{equation}
	\end{lemma}
	\begin{proof}
		We first prove (\ref{ZEQN:DiffEstmt_N4tphi_EQN01}). By duality, it suffices to prove the estimate
		\[
			\bigg|\int_0^T\int_{\mathbb{R}^2}
			\overline{\omega_0}\,\mathcal{N}^4_t\left[\psi_1,\omega,\psi_2,\psi_3\right]\psi_4\,\mathrm{d}x\,\mathrm{d}t\bigg|
		\le
			C(s)
			T^{\frac{1}{2}}
			\left\|\omega_0\right\|_{V^2_{B^\dagger}H^{-(s-1)}[T]}
			\left\|\omega\right\|_{\Psi^{s-1}[T]}
			\prod_{\ell=1}^4
			\left\|\psi_\ell\right\|_{\Psi^s[T]}
		\;.\]
		For this, using (\ref{ZEQN:DiffEstmt_N4t_EQN01}) from Lemma \ref{ZLEM:DiffEstmt_N4t} gives
		\begin{equation*}\begin{split}
			\bigg|\int_0^T\int_{\mathbb{R}^2}
		&
			\mathrm{P}_\nu\overline{\omega_0}\,
			\mathrm{P}_\mu\mathcal{N}^4_t\left[\psi_1,\omega,\psi_2,\psi_3\right]
			\,\mathrm{P}_\lambda\psi_4\,\mathrm{d}x\,\mathrm{d}t\bigg|
		\\\le\;&
			\left\|\mathrm{P}_\nu\omega_0\right\|_{L^\infty_tL^2_x[T]}
			\left\|\mathrm{P}_\mu\mathcal{N}^4_t\left[\psi_1,\omega,\psi_2,\psi_3\right]\right\|_{L^2_tL^\infty_x[T]}
			\left\|\mathrm{P}_\lambda\psi_4\right\|_{L^\infty_tL^2_x[T]}
		\\\le\;&
			C(s)
			T^{\frac{1}{2}}
			\nu^{s-1}
			\mu^{\frac{3}{4}-s}
			\lambda^{-s}
			\left\|\omega_0\right\|_{V^2_{B^\dagger}H^{-(s-1)}[T]}
			\left\|\omega\right\|_{\Psi^{s-1}[T]}
			\prod_{\ell=1}^4
			\left\|\psi_\ell\right\|_{\Psi^s[T]}
		\end{split}\end{equation*}
		which is certainly summable over the regime where the larger two of $\{\nu,\mu,\lambda\}$ are comparable.
		Therefore we have proved (\ref{ZEQN:DiffEstmt_N4tphi_EQN01}).
		
		The proof of (\ref{ZEQN:DiffEstmt_N4tphi_EQN02}) is exactly the same, except that (\ref{ZEQN:DiffEstmt_N4t_EQN02})
		is used in place of (\ref{ZEQN:DiffEstmt_N4t_EQN01}).
		
		We turn to the proof of (\ref{ZEQN:DiffEstmt_N4tphi_EQN03}).
		By duality, it suffices to prove the estimate
		\begin{equation} \label{ZEQN:DiffEstmt_N4tphi_EQN31}
			\bigg|\int_0^T\int_{\mathbb{R}^2}
			\overline{\omega_0}\,\mathcal{N}^4_t\left[\psi_1,\psi_2,\psi_3,\psi_4\right]\omega\,\mathrm{d}x\,\mathrm{d}t\bigg|
		\le
			C(s)
			T^{\frac{1}{2}}
			\left\|\omega_0\right\|_{V^2_{B^\dagger}H^{-(s-1)}[T]}
			\left\|\omega\right\|_{\Psi^{s-1}[T]}
			\prod_{\ell=1}^4
			\left\|\psi_\ell\right\|_{\Psi^s[T]}
		\;.\end{equation}
		Firstly, using Lemma \ref{ZLEM:DispEstmt_N4t}, we have
		\begin{equation*}\begin{split}
			\sum_\nu
			\bigg|\int_0^T\int_{\mathbb{R}^2}
		&
			\mathrm{P}_\nu\overline{\omega_0}\,
			\mathrm{P}_{\lesssim\nu}\mathcal{N}^4_t\left[\psi_1,\psi_2,\psi_3,\psi_4\right]
			\,\mathrm{P}_{\approx\nu}\omega\,\mathrm{d}x\,\mathrm{d}t\bigg|
		\\\lesssim\;&
			\int_0^T
			\sum_\nu
			\left\|\mathrm{P}_\nu\omega_0(t)\right\|_{L^2_x}
			\left\|\mathcal{N}^4_t\left[\psi_1,\psi_2,\psi_3,\psi_4\right](t)\right\|_{L^\infty_x}
			\left\|\mathrm{P}_{\approx\nu}\omega(t)\right\|_{L^2_x}
			\mathrm{d}t
		\\\le\;&
			C(s)
			T^{\frac{1}{2}}
			\left\|\omega_0\right\|_{L^\infty_tH^{-(s-1)}[T]}
			\left\|\mathcal{N}^4_t\left[\psi_1,\psi_2,\psi_3,\psi_4\right]\right\|_{L^2_tL^\infty_x[T]}
			\left\|\omega\right\|_{L^\infty_tH^{s-1}[T]}
		\\\le\;&
			C(s)
			T^{\frac{1}{2}}
			\left\|\omega_0\right\|_{V^2_{B^\dagger}H^{-(s-1)}[T]}
			\left\|\omega\right\|_{\Psi^{s-1}[T]}
			\prod_{\ell=1}^4
			\left\|\psi_\ell\right\|_{\Psi^s[T]}
		\;.\end{split}\end{equation*}
		By the Littlewood-Paley trichotomy, it remains to prove
		\begin{equation} \label{ZEQN:DiffEstmt_N4tphi_EQN32}
			\left(\sum_{\nu\approx\mu\gg\lambda} + \sum_{\lambda\approx\mu\gg\nu}\right)
			\bigg|\int_0^T\int_{\mathbb{R}^2}
			\mathrm{P}_\nu\overline{\omega_0}\,
			\mathrm{P}_\mu\mathcal{N}^4_t\left[\psi_1,\psi_2,\psi_3,\psi_4\right]
			\mathrm{P}_\lambda\omega\,\mathrm{d}x\,\mathrm{d}t\bigg|
			\le
			\mbox{RHS(\ref{ZEQN:DiffEstmt_N4tphi_EQN31})}
		\;.\end{equation}
		For this, using Lemma \ref{ZLEM:DispEstmt_N4t} we have
		\begin{equation*}\begin{split}
			\bigg|\int_0^T\int_{\mathbb{R}^2}
		&
			\mathrm{P}_\nu\overline{\omega_0}\,
			\mathrm{P}_\mu\mathcal{N}^4_t\left[\psi_1,\psi_2,\psi_3,\psi_4\right]
			\mathrm{P}_\lambda\omega\,\mathrm{d}x\,\mathrm{d}t\bigg|
		\\\le\;&
			C(s)
			T^{\frac{1}{2}}
			\left\|\mathrm{P}_\nu\omega_0\right\|_{L^\infty_tL^2_x[T]}
			\left\|\mathrm{P}_\mu\mathcal{N}^4_t\left[\psi_1,\psi_2,\psi_3,\psi_4\right]\right\|_{L^2_tL^\infty_x[T]}
			\left\|\mathrm{P}_\lambda\omega\right\|_{L^\infty_tL^2_x[T]}
		\\\le\;&
			C(s)
			T^{\frac{1}{2}}
			\nu^{s-1}
			\mu^{-\frac{1}{4}-s}
			\lambda^{-(s-1)}
			\left\|\omega_0\right\|_{V^2_{B^\dagger}H^{-(s-1)}[T]}
			\left\|\omega\right\|_{L^\infty_tH^{s-1}[T]}
			\prod_{\ell=1}^4
			\left\|\psi_\ell\right\|_{\Psi^s[T]}
		\;.\end{split}\end{equation*}
		Therefore (\ref{ZEQN:DiffEstmt_N4tphi_EQN32}) follows immediately.
	\end{proof}
	
	\begin{lemma} \label{ZLEM:DiffEstmt_N4xphi}
		Assume the hypothesis (\ref{ZHYP:nablaB}).
		Let $T\in(0,1]$.
		Then we have the estimates
		\begin{equation}\label{ZEQN:DiffEstmt_N4xphi_EQN01}
			\left\|\mathcal{N}^4_x\left[\psi_1,\psi_2,\psi_3,\omega\right]\psi_4\right\|_{\mathrm{D}U^2_{B^\dagger}H^{s-1}[T]}
			\le
			C(s)
			T^{\frac{1}{2}}
			\left\|\omega\right\|_{\Psi^{s-1}[T]}
			\prod_{\ell=1}^4
			\left\|\psi_\ell\right\|_{\Psi^s[T]}
		\end{equation}
		and
		\begin{equation}\label{ZEQN:DiffEstmt_N4xphi_EQN02}
			\left\|\mathcal{N}^4_x\left[\psi_1,\psi_2,\psi_3,\psi_4\right]\omega\right\|_{\mathrm{D}U^2_{B^\dagger}H^{s-1}[T]}
			\le
			C(s)
			T^{\frac{1}{2}}
			\left\|\omega\right\|_{\Psi^{s-1}[T]}
			\prod_{\ell=1}^4
			\left\|\psi_\ell\right\|_{\Psi^s[T]}
		\;.\end{equation}
	\end{lemma}
	\begin{proof}
		By duality, the proof of (\ref{ZEQN:DiffEstmt_N4xphi_EQN01}) reduces to proving the estimate
		\begin{equation*}\begin{split}
			\bigg|\int_0^T\int_{\mathbb{R}^2_x}
			\overline{\omega_0}\,
			\mathcal{N}^4_x\left[\psi_1,\psi_2,\psi_3,\omega\right]\psi_4
			\,\mathrm{d}x\,\mathrm{d}t\bigg|
		\le\;&
			C(s)T^{\frac{1}{2}}
			\left\|\omega_0\right\|_{V^2_{B^\dagger}H^{-(s-1)}[T]}
			\left\|\omega\right\|_{\Psi^{s-1}[T]}
			\prod_{\ell=1}^4
			\left\|\psi_\ell\right\|_{\Psi^s[T]}
		\;.\end{split}\end{equation*}
		This, in turn, follows from using Lemmas \ref{ZLEM:EnergyEstmt_N2x} and \ref{ZLEM:DiffEstmt_N2x} to obtain
		\begin{equation*}\begin{split}
			\bigg|\int_0^T\int_{\mathbb{R}^2_x}
		&
			\mathrm{P}_\nu\overline{\omega_0}\,
			\mathrm{P}_{\mu_1}\mathcal{N}^2_x\left[\psi_1,\psi_2\right]\cdot
			\mathrm{P}_{\mu_2}\mathcal{N}^2_x\left[\psi_3,\omega\right]
			\,\mathrm{P}_\lambda\psi_4
			\,\mathrm{d}x\,\mathrm{d}t\bigg|
		\\\le\;&
			T^{\frac{1}{2}}
			\left\|\mathrm{P}_\nu\omega_0\right\|_{L^\infty_tL^2_x[T]}
			\left\|\mathrm{P}_{\mu_1}\mathcal{N}^2_x\left[\psi_1,\psi_2\right]\right\|_{L^\infty_{t,x}[T]}
			\left\|\mathrm{P}_{\mu_2}\mathcal{N}^2_x\left[\psi_3,\omega\right]\right\|_{L^2_tL^\infty_x[T]}
			\left\|\mathrm{P}_\lambda\psi_4\right\|_{L^\infty_tL^2_x[T]}
		\\\le\;&
			C(s)
			T^{\frac{1}{2}}
			\nu^{s-1}
			\mu_1^{\frac{1}{4}-s}
			\mu_2^{\frac{3}{4}-s}
			\lambda^{-s}
			\left\|\omega_0\right\|_{V^2_{B^\dagger}H^{-(s-1)}[T]}
			\left\|\omega\right\|_{\Psi^{s-1}[T]}
			\prod_{\ell=1}^4
			\left\|\psi_\ell\right\|_{\Psi^s[T]}
		\;,\end{split}\end{equation*}
		and observing that the right-hand side is summable over the regime where the two largest of $\{\nu,\mu_1,\mu_2,\lambda\}$
		are comparable.
		
		We now turn to the proof of (\ref{ZEQN:DiffEstmt_N4xphi_EQN02}). By duality, this reduces to proving
		\begin{equation*}\begin{split}
			\bigg|\int_0^T\int_{\mathbb{R}^2_x}
			\overline{\omega_0}\,
			\mathcal{N}^4_x\left[\psi_1,\psi_2,\psi_3,\psi_4\right]\omega
			\,\mathrm{d}x\,\mathrm{d}t\bigg|
		\le\;&
			C(s)T^{\frac{1}{2}}
			\left\|\omega_0\right\|_{V^2_{B^\dagger}H^{-(s-1)}[T]}
			\left\|\omega\right\|_{\Psi^{s-1}[T]}
			\prod_{\ell=1}^4
			\left\|\psi_\ell\right\|_{\Psi^s[T]}
		\;.\end{split}\end{equation*}
		Firstly, we have, using Lemma \ref{ZLEM:EnergyEstmt_N2x}, that
		\begin{equation*}\begin{split}
			\sum_\nu
			\bigg|\int_0^T\int_{\mathbb{R}^2_x}
		&
			\mathrm{P}_\nu\overline{\omega_0}\,
			\mathrm{P}_{\ll\nu}\mathcal{N}^2_x\left[\psi_1,\psi_2\right]\cdot
			\mathrm{P}_{\ll\nu}\mathcal{N}^2_x\left[\psi_3,\psi_4\right]\,
			\mathrm{P}_{\approx\nu}\omega
			\,\mathrm{d}x\,\mathrm{d}t\bigg|
		\\\le\;&
			\int_0^T \sum_\nu
			\left\|\mathrm{P}_\nu\omega_0(t)\right\|_{L^2_x}
			\left\|\mathcal{N}^2_x\left[\psi_1,\psi_2\right](t)\right\|_{L^\infty_x}
			\left\|\mathcal{N}^2_x\left[\psi_3,\psi_4\right](t)\right\|_{L^\infty_x}
			\left\|\mathrm{P}_{\approx\nu}\omega(t)\right\|_{L^2_x}
			\,\mathrm{d}t
		\\\le\;&
			C(s)T
			\left\|\omega_0\right\|_{L^\infty_tH^{-(s-1)}[T]}
			\left\|\mathcal{N}^2_x\left[\psi_1,\psi_2\right]\right\|_{L^\infty_{t,x}[T]}
			\left\|\mathcal{N}^2_x\left[\psi_3,\psi_4\right]\right\|_{L^\infty_{t,x}[T]}
			\left\|\omega\right\|_{L^\infty_tH^{s-1}[T]}
		\\\le\;&
			C(s)T^{\frac{1}{2}}
			\left\|\omega_0\right\|_{V^2_{B^\dagger}H^{-(s-1)}[T]}
			\left\|\omega\right\|_{\Psi^{s-1}[T]}
			\prod_{\ell=1}^4
			\left\|\psi_\ell\right\|_{\Psi^s[T]}
		\;.\end{split}\end{equation*}
		By the Littlewood-Paley trichotomy and symmetry, it remains to show that
		\begin{equation}\label{ZEQN:DiffEstmt_N4xphi_EQN21}\begin{split}
			\sum_{\mu,\nu\;:\;\mu\gtrsim\nu}
			\bigg|\int_0^T\int_{\mathbb{R}^2_x}
		&
			\mathrm{P}_\nu\overline{\omega_0}\,
			\mathrm{P}_\mu\mathcal{N}^2_x\left[\psi_1,\psi_2\right]\cdot
			\mathrm{P}_{\le\mu}\mathcal{N}^2_x\left[\psi_3,\psi_4\right]
			\,\omega
			\,\mathrm{d}x\,\mathrm{d}t\bigg|
		\\\le\;&
			C(s)T^{\frac{1}{2}}
			\left\|\omega_0\right\|_{V^2_{B^\dagger}H^{-(s-1)}[T]}
			\left\|\omega\right\|_{\Psi^{s-1}[T]}
			\prod_{\ell=1}^4
			\left\|\psi_\ell\right\|_{\Psi^s[T]}
		\;.\end{split}\end{equation}
		Indeed, by Lemmas \ref{ZLEM:EnergyEstmt_N2x} and \ref{ZLEM:DispEstmt_N2x}, we have
		\begin{equation*}\begin{split}
			\bigg|\int_0^T\int_{\mathbb{R}^2_x}
		&
			\mathrm{P}_\nu\overline{\omega_0}\,
			\mathrm{P}_\mu\mathcal{N}^2_x\left[\psi_1,\psi_2\right]\cdot
			\mathrm{P}_{\le\mu}\mathcal{N}^2_x\left[\psi_3,\psi_4\right]
			\,\omega
			\,\mathrm{d}x\,\mathrm{d}t\bigg|
		\\\le\;&
			T^{\frac{1}{2}}
			\left\|\mathrm{P}_\nu\omega_0\right\|_{L^\infty_tL^2_x[T]}
			\left\|\mathrm{P}_\mu\mathcal{N}^2_x\left[\psi_1,\psi_2\right]\right\|_{L^2_tL^\infty_x[T]}
			\left\|\mathcal{N}^2_x\left[\psi_3,\psi_4\right]\right\|_{L^\infty_{t,x}[T]}
			\left\|\omega\right\|_{L^\infty_tL^2_x[T]}
		\\\le\;&
			C(s)
			T^{\frac{1}{2}}
			\nu^{s-1}
			\mu^{-\frac{1}{4}-s}
			\left\|\omega_0\right\|_{V^2_{B^\dagger}H^{-(s-1)}[T]}
			\left\|\omega\right\|_{\Psi^{s-1}[T]}
			\prod_{\ell=1}^4
			\left\|\psi_\ell\right\|_{\Psi^s[T]}
		\;,\end{split}\end{equation*}
		and then (\ref{ZEQN:DiffEstmt_N4xphi_EQN21}) follows immediately.
	\end{proof}

	\begin{lemma} \label{ZLEM:DiffEstmt_phi3}
		Assume the hypothesis (\ref{ZHYP:nablaB}).
		Let $T\in(0,1]$.
		Then we have the estimate
		\begin{equation}\label{ZEQN:DiffEstmt_phi3_EQN01}
			\left\|\psi_1\psi_2\omega\right\|_{\mathrm{D}U^2_{B^\dagger}H^{s-1}[T]}
			\le
			C(s)
			T^{\frac{1}{2}}
			\left\|\omega\right\|_{\Psi^{s-1}[T]}
			\prod_{\ell=1}^2
			\left\|\psi_\ell\right\|_{\Psi^s[T]}
		\;.\end{equation}
	\end{lemma}
	\begin{proof}
		By duality, it suffices to prove the estimate
		\[
			\bigg|\int_0^T\int_{\mathbb{R}^2_x}
			\overline{\omega_0}\,\psi_1\psi_2\omega\,\mathrm{d}x\,\mathrm{d}t\bigg|
		\le
			C(s)
			T^{\frac{1}{2}}
			\left\|\omega_0\right\|_{\mathrm{V}^2_BH^{-(s-1)}[T]}
			\left\|\omega\right\|_{\Psi^{s-1}[T]}
			\prod_{\ell=1}^2
			\left\|\psi_\ell\right\|_{\Psi^s[T]}
		\;.\]
		For this, using Lemma \ref{ZLEM:DispEstmt_LinftyCtrl} we have
		\begin{equation*}\begin{split}
			\bigg|\int_0^T\int_{\mathbb{R}^2_x}
		&
			\mathrm{P}_\nu\overline{\omega_0}\,\mathrm{P}_{\mu_1}\psi_1\,\mathrm{P}_{\mu_2}\psi_2\,\mathrm{P}_\lambda\omega\,\mathrm{d}x\,\mathrm{d}t\bigg|
		\\\le\;&
			T^{\frac{1}{2}}
			\left\|\mathrm{P}_\nu\omega_0\right\|_{L^\infty_tL^2_x[T]}
			\left\|\mathrm{P}_{\mu_1}\psi_1\right\|_{L^4_tL^\infty_x[T]}
			\left\|\mathrm{P}_{\mu_2}\psi_2\right\|_{L^4_tL^\infty_x[T]}
			\left\|\mathrm{P}_\lambda\omega\right\|_{L^\infty_tL^2_x[T]}
		\\\le\;&
			C(s)
			T^{\frac{1}{2}}
			\nu^{s-1}
			\mu_1^{\frac{3}{4}-s}
			\mu_2^{\frac{3}{4}-s}
			\lambda^{-(s-1)}
			\left\|\omega_0\right\|_{V^2_{B^\dagger}H^{-(s-1)}[T]}
			\left\|\omega\right\|_{\Psi^{s-1}[T]}
			\prod_{\ell=1}^2
			\left\|\psi_\ell\right\|_{\Psi^s[T]}
		\;.\end{split}\end{equation*}
		Now simply observe that the right-hand side is summable over the regime where the two largest frequencies among $\{\nu,\mu_1,\mu_2,\lambda\}$
		are comparable.
		Therefore we obtain (\ref{ZEQN:DiffEstmt_phi3_EQN01}).
	\end{proof}

\subsection{Proof of Theorem \ref{ZTHM:Sect6MainThm}}
	By Theorem \ref{ZTHM:Sect5MainThm}, with $T\le\delta_1(s)(1+M)^{-28}$, we have the bounds $\|\psi\|_{U^2_BH^s[T]}\le 2M$ and $\|\psi'\|_{U^2_{B'}H^s[T]}\le 2M$.
	Thus, by Lemma \ref{ZLEM:OldLemma4.9Restated} and the $U^2\hookrightarrow V^2_{\mathrm{rc}}$ embedding, we have
	\[
		\left\|\psi\right\|_{\Psi^s[T]} \;,\; \left\|\psi'\right\|_{\Psi^s[T]}
		\le
		C(s)\left(1+M\right)^3
	\;.\]
	Similarly,
	\[
		\left\|\psi-\psi'\right\|_{\Psi^{s-1}[T]} \le C(s)\left(1+M\right)^2\left\|\psi-\psi'\right\|_{U^2_BH^{s-1}[T]}
	\;.\]
	Now, apply Lemmas \ref{ZLEM:DiffEstmt_PBphi},
	\ref{ZLEM:DiffEstmt_Q}, \ref{ZLEM:DiffEstmt_N2tphi}, \ref{ZLEM:DiffEstmt_N4tphi}, \ref{ZLEM:DiffEstmt_N4xphi}, \ref{ZLEM:DiffEstmt_phi3},
	with each $\psi_\ell$ being $\psi$ or $\psi'$ or their complex conjugates,
	and $\omega$ being $\psi-\psi'$ or its complex conjugate.
	We find, respectively,
	\begin{equation*}\begin{split}
		\left\|\mathfrak{P}_{B^\dagger-B}\psi\right\|_{\mathrm{D}U^2_{B^\dagger}H^{s-1}[T]}
	\le\;&
		C(s)T^{\frac{1}{2}}
		\left(1+M\right)^3
		\left\|B-B^\dagger\right\|_{L^2_tL^\infty_x[T]}
	\;,\\
		\left\|\mathfrak{P}_{B'-B^\dagger}\psi'\right\|_{\mathrm{D}U^2_{B^\dagger}H^{s-1}[T]}
	\le\;&
		C(s)T^{\frac{1}{2}}
		\left(1+M\right)^3
		\left\|B'-B^\dagger\right\|_{L^2_tL^\infty_x[T]}
	\;,\\
		\left\|\mathcal{Q}\left[\overline{\psi},\psi,\psi\right] - \mathcal{Q}\left[\overline{\psi'},\psi',\psi'\right]\right\|_{\mathrm{D}U^2_{B^\dagger}H^{s-1}[T]}
	\le\;&
		C(s)T^{\frac{1}{2}}
		\left(1+M\right)^8
		\left\|\psi-\psi'\right\|_{U^2_{B^\dagger}H^{s-1}[T]}
	\;,\\
		\left\|\mathcal{N}^2_0\left[\overline{\psi},\psi\right]\psi - \mathcal{N}^2_0\left[\overline{\psi'},\psi'\right]\psi'\right\|_{\mathrm{D}U^2_{B^\dagger}H^{s-1}[T]}
	\le\;&
		C(s)T^{\frac{1}{2}}
		\left(1+M\right)^{10}
		\left\|\psi-\psi'\right\|_{U^2_{B^\dagger}H^{s-1}[T]}
	\;,\\
		\left\|\mathcal{N}^4_t\left[\overline{\psi},\psi,\overline{\psi},\psi\right]\psi
		- \mathcal{N}^4_t\left[\overline{\psi'},\psi',\overline{\psi'},\psi'\right]\psi'\right\|_{\mathrm{D}U^2_{B^\dagger}H^{s-1}[T]}
	\le\;&
		C(s)T^{\frac{1}{2}}
		\left(1+M\right)^{14}
		\left\|\psi-\psi'\right\|_{U^2_{B^\dagger}H^{s-1}[T]}
	\;,\\
		\left\|\mathcal{N}^4_x\left[\overline{\psi},\psi,\overline{\psi},\psi\right]\psi
		- \mathcal{N}^4_x\left[\overline{\psi'},\psi',\overline{\psi'},\psi'\right]\psi'\right\|_{\mathrm{D}U^2_{B^\dagger}H^{s-1}[T]}
	\le\;&
		C(s)T^{\frac{1}{2}}
		\left(1+M\right)^{14}
		\left\|\psi-\psi'\right\|_{U^2_{B^\dagger}H^{s-1}[T]}
	\;,\\
		\left\|\left|\psi\right|^2\psi - \left|\psi'\right|^2\psi'\right\|_{\mathrm{D}U^2_{B^\dagger}H^{s-1}[T]}
	\le\;&
		C(s)T^{\frac{1}{2}}
		\left(1+M\right)^8
		\left\|\psi-\psi'\right\|_{U^2_{B^\dagger}H^{s-1}[T]}
	\;.\end{split}\end{equation*}
	Hence, applying Duhamel's formula to (\ref{ZEQN:DifferenceEqn}) and using the above estimates, we obtain
	\begin{equation*}\begin{split}
		\left\|\psi-\psi'\right\|_{U^2_{B^\dagger}H^{s-1}[T]}
	\le\;&
		\left\|\psi(0)-\psi'(0)\right\|_{H^{s-1}}
	\\&
		+ C(s)T^{\frac{1}{2}}\left(1+M\right)^3
		\left(
			\left\|B-B^\dagger\right\|_{L^2_tL^\infty_x[T]} + \left\|B'-B^\dagger\right\|_{L^2_tL^\infty_x[T]}
		\right)
	\\&
		+ C(s)T^{\frac{1}{2}}\left(1+M\right)^{14}\left\|\psi-\psi'\right\|_{U^2_{B^\dagger}H^{s-1}[T]}
	\;.\end{split}\end{equation*}
	Therefore, if $T=\delta_2(1+M)^{-28}$ with $\delta_2=\delta_2(s,\varepsilon)$ chosen sufficiently small,
	we can subtract the last term on the right-hand side from the left, and conclude
	\[
		\left\|\psi-\psi'\right\|_{U^2_{B^\dagger}H^{s-1}[T]}
		\le
		\varepsilon
		\left(
			\left\|B-B^\dagger\right\|_{L^2_tL^\infty_x[T]} + \left\|B'-B^\dagger\right\|_{L^2_tL^\infty_x[T]}
		\right)
		+ C\left\|\psi(0)-\psi'(0)\right\|_{H^{s-1}}
	\;.\]

	The proof of Theorem \ref{ZTHM:Sect6MainThm} is complete.

\section{Completion of the Proof of the Main Theorem} \label{ZSECT:Conclusion}

	We are now ready to complete the proofs of Theorem \ref{ZTHM:MainThm} and Corollary \ref{ZCOR:MainThmCor}.
	We first prove the following auxiliary lemma,
	which will be necessary to upgrade $L^\infty_tH^{s-1}[T]$ convergence to $L^\infty_tH^s[T]$ convergence in the following proofs.

	\begin{lemma} \label{ZLEM:SobWtLem}
		Let $s\ge 1$ and let $\mathcal{K}$ be a compact subset of $H^s$.
		Then there exists a Sobolev weight $\mathfrak{m}$ satisfying the hypothesis (\ref{ZHYP:SobWt}), which additionally satisfies
		\begin{equation} \label{ZEQN:SobWtLem_EQN01}
			\lim_{\lambda\to\infty} \frac{\mathfrak{m}(\lambda)}{\lambda} = \infty
		\end{equation}
		and
		\begin{equation} \label{ZEQN:SobWtLem_EQN02}
			\sup_{w\in\mathcal{K}}\left\|w\right\|_{H^{\mathfrak{m}}} \le 2\sup_{w\in\mathcal{K}} \left\|w\right\|_{H^s}
		\;.\end{equation}
	\end{lemma}
	\begin{proof}
		By rescaling, we may assume without loss of generality that
		\[
			\sup_{w\in\mathcal{K}} \left\|w\right\|_{H^s} = 1
		\;.\]

		We claim that, for every $\varepsilon\in(0,1]$, there exists $\varLambda=\varLambda(\varepsilon)\in\mathfrak{D}$ such that
		\[
			\sup_{w\in\mathcal{K}} \left(\sum_{\lambda\ge\varLambda}
				 \lambda^{2s}\left\|\mathrm{P}_\lambda w\right\|_{L^2_x}^2\right)^{\frac{1}{2}}
			\le \varepsilon
		\;.\]
		Indeed, suppose for a contradiction that our claim was false.
		Then, for every $\mu\in\mathfrak{D}$ there exists $w_\mu\in\mathcal{K}$ such that
		\[
			\left(\sum_{\lambda\;:\;\lambda\ge\mu}
				 \lambda^{2s}\left\|\mathrm{P}_\lambda w_\mu\right\|_{L^2_x}^2\right)^{\frac{1}{2}}
			> \varepsilon
		\;.\]
		Since $\mathcal{K}$ is compact, there exists a subsequence $\mu_m\to\infty$ such that $w_{\mu_m}$ converges to some $w_\infty\in\mathcal{K}$ in $H^s$.
		As $w_\infty\in H^s$, there exists $\nu\in\mathfrak{D}$ such that
		\[
			\left(\sum_{\lambda\;:\;\lambda\ge\nu}
				 \lambda^{2s}\left\|\mathrm{P}_\lambda w_\infty\right\|_{L^2_x}^2\right)^{\frac{1}{2}}
			\le \frac{\varepsilon}{2}
		\;.\]
		However, the triangle inequality gives, for $\mu_m\ge\nu$,
		\begin{equation*}\begin{split}
			\left(\sum_{\lambda\;:\;\lambda\ge\nu}
				\lambda^{2s}\left\|\mathrm{P}_\lambda \left(w_\infty-w_{\mu_m}\right)\right\|_{L^2_x}^2\right)^{\frac{1}{2}}
		\ge\;&
			\left(\sum_{\lambda\;:\;\lambda\ge\nu}
				\lambda^{2s}\left\|\mathrm{P}_\lambda w_{\mu_m}\right\|_{L^2_x}^2\right)^{\frac{1}{2}}
			-
			\left(\sum_{\lambda\;:\;\lambda\ge\nu}
				\lambda^{2s}\left\|\mathrm{P}_\lambda w_\infty\right\|_{L^2_x}^2\right)^{\frac{1}{2}}
		\\>\;& \frac{\varepsilon}{2}
		\end{split}\end{equation*}
		which contradicts the aforementioned convergence $w_{\mu_m}\to w_\infty$.

		Set $\nu_0:=1$ and for $m\in\{1,2,3,\ldots\}$ let $\nu_m$  be the smallest element of $\mathfrak{D}$ strictly greater than $\nu_{m-1}$
		such that
		\[
			\sup_{w\in\mathcal{K}} \left(\sum_{\lambda\ge\nu_m}
				 \lambda^{2s}\left\|\mathrm{P}_\lambda w\right\|_{L^2_x}^2\right)^{\frac{1}{2}}
			\le 2^{-m}
		\;.\]
		Set $\mathfrak{m}(\lambda):=2^{\frac{1}{8}m}\lambda^s$ for $\nu_m\le\lambda<\nu_{m+1}$.
		Clearly $\mathfrak{m}$ satisfies the hypothesis (\ref{ZHYP:SobWt}) and (\ref{ZEQN:SobWtLem_EQN01}),
		and it is straightforward to verify that (\ref{ZEQN:SobWtLem_EQN02}) also holds.
	\end{proof}

\subsection{Proof of Theorem \ref{ZTHM:MainThm}: Existence of solutions}
	Let $\varepsilon=\varepsilon(s,D)\in(0,1]$ be a small constant which we will choose later.
	Let $T=\delta_2(1+2D)^{-28}$ where $\delta_2=\delta_2(s,\varepsilon)$ is given in Theorem \ref{ZTHM:Sect6MainThm}.
	Let $M:=2D$.

	Now, let the initial data $\phi^{\mathrm{in}}\in H^s$ be given with $\|\phi^{\mathrm{in}}\|_{H^s}\le D$.
	By Lemma \ref{ZLEM:SobWtLem}, we may choose a Sobolev weight $\mathfrak{m}$ satisfying the hypothesis (\ref{ZHYP:SobWt}) and (\ref{ZEQN:SobWtLem_EQN01}),
	such that $\|\phi^{\mathrm{in}}\|_{H^{\mathfrak{m}}}\le M$.
	Using Theorem \ref{ZTHM:Sect5MainThm},
	starting from $A_x^{[0]}=0$,
	we construct the iterates $\phi^{[n]}\in U^2_{A_x^{[n-1]}}H^{\mathfrak{m}}[T]$ solving the iteration scheme (\ref{ZEQN:ItrtnSchemeSuccinct}).

	We claim that, provided $\varepsilon=\varepsilon(s,D)$ was chosen small enough,
	$\{\phi^{[n]}\}_{n=1}^\infty$ will be a Cauchy sequence $L^\infty_tH^s[T]$.
	Indeed, applying Theorem \ref{ZTHM:Sect6MainThm} with $B^\dagger = A_x^{[n-1]}$ and $(\psi,B) = (\phi^{[n]},A_x^{[n-1]})$ and $(\psi',B')=(\phi^{[n+1]},A_x^{[n]})$, we find
	\[
		\left\|\phi^{[n]}-\phi^{[n+1]}\right\|_{U^2_{A_x^{[n-1]}}H^{s-1}[T]}
		\le
		\varepsilon\left\|A_x^{[n-1]} - A_x^{[n]}\right\|_{L^2_tL^\infty_x[T]}
	\;.\]
	By Lemma \ref{ZLEM:DiffEstmt_N2x} and Lemma \ref{ZLEM:OldLemma4.9Restated}, we may replace the right-hand side by
	\begin{equation*}\begin{split}
		\left\|\phi^{[n]}-\phi^{[n+1]}\right\|_{U^2_{A_x^{[n-1]}}H^{s-1}[T]}
	\le\;&
		C(s)\varepsilon\left(\left\|\phi^{[n]}\right\|_{\Psi^s[T]} + \left\|\phi^{[n-1]}\right\|_{\Psi^s[T]}\right)
		\left\|\phi^{[n-1]}-\phi^{[n]}\right\|_{\Psi^{s-1}[T]}
	\\\le\;&
		C(s)\varepsilon\left(1+M\right)^5
		\left\|\phi^{[n-1]}-\phi^{[n]}\right\|_{U^2_{A_x^{[n-2]}}H^{s-1}[T]}
	\;.\end{split}\end{equation*}
	Choose $\varepsilon=\varepsilon(s,D)$ sufficiently small so that $C(s)\varepsilon(1+M)^5<\frac{1}{2}$ on the right-hand side.
	Then
	\[
		\left\|\phi^{[n]}-\phi^{[n+1]}\right\|_{U^2_{A_x^{[n-1]}}H^{s-1}[T]}
		\le
		\frac{1}{2}
		\left\|\phi^{[n-1]}-\phi^{[n]}\right\|_{U^2_{A_x^{[n-2]}}H^{s-1}[T]}
	\;.\]
	Now, Lemma \ref{ZLEM:OldLemma4.9Restated} gives the estimate
	\[
		\left\|\phi^{[n]}-\phi^{[n+1]}\right\|_{\Psi^{s-1}[T]}\le C(s)\left(1+M\right)^2\left\|\phi^{[n]}-\phi^{[n+1]}\right\|_{U^2_{A_x^{[n-1]}}H^{s-1}[T]}
	\;.\]
	Thus, $\{\phi^{[n]}\}_{n=1}^\infty$ is a Cauchy sequence in $\Psi^{s-1}[T]$ and hence in $L^\infty_tH^{s-1}[T]$.
	On the other hand, Lemma \ref{ZLEM:OldLemma4.9} also guarantees that $\{\phi^{[n]}\}_{n=1}^\infty$ is a bounded sequence in $L^\infty_tH^{\mathfrak{m}}[T]$.
	Due to (\ref{ZEQN:SobWtLem_EQN01}), we deduce that $\{\phi^{[n]}\}_{n=1}^\infty$ is also a Cauchy sequence in $L^\infty_tH^s[T]$, as claimed.
	
	Let $\phi$ be the limit of $\{\phi^{[n]}\}_{n=1}^\infty$ in $L^\infty_tH^s[T]$.
	By Lemma \ref{ZLEM:EnergyEstmt_N2x},
	\[
		A_x^{[n]} \rightarrow A_x := -\mbox{$\frac{1}{2}$}\mathcal{N}^2_x\left[\overline{\phi},\phi\right]
		\quad\mbox{ in }
		L^\infty_{t,x}[T]
	\;.\]
	Moreover, since $H^s$ controls the $L^4_x$ norm,
	\[
		\left\|\phi\right\|_{\Psi^s[T]} \le \liminf_{n\to\infty} \left\|\phi^{[n]}\right\|_{\Psi^s[T]}
	\;.\]
	In particular, $\|\phi\|_{\Psi^s[T]}\le C(s)(1+M)^3$.
	By our choice of $T=\delta_2(\varepsilon,s)(1+M)^{-28}$,
	Lemma \ref{ZLEM:DispEstmt_nablaN2x} guarantees that, $\|\nabla A_x\|_{L^1_tL^\infty_x[T]}\le 1$.
	Hence $A_x$ is an admissible form satisfying the hypothesis (\ref{ZHYP:nablaB}) and also (\ref{ZEQN:Impose_B}).
	
	Since $A_x^{[n]} \rightarrow A_x$ in $L^\infty_{t,x}[T]$ and $\phi^{[n]}\rightarrow\phi$ in $L^\infty_tH^s[T]$, we have
	\[
		\mathfrak{P}_{A_x^{[n-1]}}\phi^{[n]} \rightarrow \mathfrak{P}_{A_x}\phi
		\quad\mbox{ in }
		L^\infty_tH^{s-1}[T]
	\;.\]
	Since $\phi^{[n]}\rightarrow\phi$ in $\Psi^{s-1}[T]$, Lemmas
	\ref{ZLEM:DiffEstmt_Q}, \ref{ZLEM:DiffEstmt_N2tphi}, \ref{ZLEM:DiffEstmt_N4tphi}, \ref{ZLEM:DiffEstmt_N4xphi}, \ref{ZLEM:DiffEstmt_phi3}
	guarantee that
	\begin{equation*}\begin{split}
		\mathcal{Q}\left[\overline{\phi^{[n]}},\phi^{[n]},\phi^{[n]}\right]
	&\rightarrow
		\mathcal{Q}\left[\overline{\phi},\phi,\phi\right]
	\;,\\
		\mathcal{N}^2_0\left[\overline{\phi^{[n]}},\phi^{[n]}\right]\phi^{[n]}
	&\rightarrow
		\mathcal{N}^2_0\left[\overline{\phi},\phi\right]\phi
	\;,\\
		\mathcal{N}^4_t\left[\overline{\phi^{[n]}},\phi^{[n]},\overline{\phi^{[n]}},\phi^{[n]}\right]\phi^{[n]}
	&\rightarrow
		\mathcal{N}^4_t\left[\overline{\phi},\phi,\overline{\phi},\phi\right]\phi
	\;,\\
		\mathcal{N}^4_x\left[\overline{\phi^{[n]}},\phi^{[n]},\overline{\phi^{[n]}},\phi^{[n]}\right]\phi^{[n]}
	&\rightarrow
		\mathcal{N}^4_x\left[\overline{\phi},\phi,\overline{\phi},\phi\right]\phi
	\;,\\
		\left|\phi^{[n]}\right|^2\phi^{[n]}
	&\rightarrow
		\left|\phi\right|^2\phi
	\end{split}\end{equation*}
	in $\mathrm{D}U^2_{B^\dagger}H^{s-1}[T]$ for any admissible form $B^\dagger$.
	
	Hence, $\phi\in U^2_{A_x}H^{s-1}[T]$ is indeed a solution to the Chern-Simons-Schr\"odinger system in the Coulomb guage, (\ref{ZEQN:CSSCoulSuccinct2}).
	Furthermore, since $\phi\in\Psi^s[T]$,
	Lemmas \ref{ZLEM:MLEstmt_Q}, \ref{ZLEM:MLEstmt_N2tphi}, \ref{ZLEM:MLEstmt_N4tphi}, \ref{ZLEM:MLEstmt_N4xphi}, \ref{ZLEM:MLEstmt_phi3}
	guarantee that the right-hand side of (\ref{ZEQN:CSSCoulSuccinct2}) belongs to $\mathrm{D}U^2_{A_x}H^s[T]$.
	In particular, $\phi\in U^2_{A_x}H^s[T]$ and $\|\phi\|_{U^2_{A_x}H^s[T]}\le 2M$.
	
	This concludes the proof of the existence of solutions.

\subsection{Proof of Theorem \ref{ZTHM:MainThm}: Uniqueness of solutions, continuity of the solution map, regularity}
	
	The uniqueness of a solution, given initial data, is a consequence of the weak Lipschitz bound (\ref{ZEQN:WkLipBd}).
	In turn, the weak Lipschitz bound (\ref{ZEQN:WkLipBd}) is a straightforward consequence of Theorem \ref{ZTHM:Sect6MainThm}.
	Indeed, let $\varepsilon=\varepsilon(s,D)\in(0,1]$ a small constant (possibly smaller than the one chosen before) which we will choose later,
	and let $T=\delta_2(s,\varepsilon)(1+D)^{-28}$ be given by Theorem \ref{ZTHM:Sect6MainThm}.
	Then,
	for two solutions $(\phi,A_x)$ and $(\phi',A_x')$ to the Chern-Simons-Schr\"{o}dinger system (\ref{ZEQN:CSSCoulSuccinct2}) with $\phi(0),\phi'(0)\in\mathbb{B}_{H^s}(D)$,
	we have the estimate
	\[
		\left\|\phi-\phi'\right\|_{U^2_{A_x}H^{s-1}[T]} \le
		\varepsilon\left\|A_x-A_x'\right\|_{L^2_tL^\infty_x[T]}
		+
		\left\|\phi(0)-\phi'(0)\right\|_{H^{s-1}}
	\;.\]
	Arguing as before using Lemma \ref{ZLEM:DiffEstmt_N2x}, we have
	\[
		\left\|\phi-\phi'\right\|_{U^2_{A_x}H^{s-1}[T]} \le
		C(s)\varepsilon(1+D)^5\varepsilon\left\|\phi-\phi'\right\|_{U^2_{A_x}H^{s-1}[T]}
		+
		\left\|\phi(0)-\phi'(0)\right\|_{H^{s-1}}
	\;,\]
	so that, with $\varepsilon=\varepsilon(s,D)$ chosen sufficiently small, we obtain by Lemma \ref{ZLEM:OldLemma4.9},
	\[
		\left\|\phi-\phi'\right\|_{L^\infty_tH^{s-1}[T]}
		\le C(s)\left\|\phi-\phi'\right\|_{U^2_{A_x}H^{s-1}[T]}
		\le C(s)\left\|\phi(0)-\phi'(0)\right\|_{H^{s-1}}
	\;.\]
	This completes the proof of (\ref{ZEQN:WkLipBd}), which also implies the uniqueness statement for solutions.

	Next, we address the issue of the continuity of the solution map into $L^\infty_tH^s[T]$.
	Let $\phi^{\mathrm{in},[n]}$ be a sequence of initial data converging to $\phi^{\mathrm{in}}$ in $H^s$.
	By Lemma \ref{ZLEM:SobWtLem}, we may pick a Sobolev weight $\mathfrak{m}$ for
	$\mathcal{K}:=\{\phi^{\mathrm{in},[n]}\}_{n=1}^\infty\cup\{\phi^{\mathrm{in}}\}$.
	By the weak Lipschitz bound (\ref{ZEQN:WkLipBd}), the solutions $\phi^{[n]}$ converge to $\phi$ in $L^\infty_tH^{s-1}[T]$.
	On the other hand, $\{\phi^{[n]}\}_{n=1}^\infty$ is bounded in $L^\infty_tH^{\mathfrak{m}}[T]$.
	Hence, (\ref{ZEQN:SobWtLem_EQN02}) guarantees that $\{\phi^{[n]}\}_{n=1}^\infty$ is a Cauchy sequence in $L^\infty_tH^s[T]$.
	Thus, the solution map is continuous.
	
	It remains to prove the last statement and the norm growth estimate (\ref{ZEQN:NormGrowthEstmt}).
	By Theorem \ref{ZTHM:Sect5MainThm}, there exists $T_1=T_1(D_1)>0$ such that any $H^1$ solution $\phi$ to (\ref{ZEQN:CSSCoul})
	with $\|\phi(0)\|_{H^1}\le D_1$ exists up to $[0,T_1)$, and satisfies $\|\phi\|_{U^2_{A_x}H^1[T_1]}\le 2D_1$
	and thus, by Lemma \ref{ZLEM:OldLemma4.9Restated},
	\[
		\|\phi\|_{\Psi^1[T_1]}\lesssim (1+D_1)^2D_1
	\;.\]
	Moreover $A_x$ satisfies (\ref{ZEQN:Impose_B}) with $M=D_1$.
	If additionally $\phi\in C_{\mathrm{b}}([0,T_s),H^s)$ with $T_s\le T_1$,
	then using Lemmas \ref{ZLEM:MLEstmt_Q}, \ref{ZLEM:MLEstmt_N2tphi}, \ref{ZLEM:MLEstmt_N4tphi}, \ref{ZLEM:MLEstmt_N4xphi}, \ref{ZLEM:MLEstmt_phi3}
	respectively, we obtain
	\begin{equation*}\begin{split}
		\left\|\mathcal{Q}\left[\overline{\phi},\phi,\phi\right]\right\|_{\mathrm{D}U^2_{A_x}H^s[T_s]}
	\le\;&
		C(s)T_s^{\frac{1}{2}}\left(1+D_1\right)^6D_1^2\left\|\phi\right\|_{U^2_{A_x}H^s[T_s]}
	\;,\\
		\left\|\mathcal{N}^2_0\left[\overline{\phi},\phi\right]\phi\right\|_{\mathrm{D}U^2_{A_x}H^s[T_s]}
	\le\;&
		C(s)T_s^{\frac{1}{2}}\left(1+D_1\right)^8D_1^2\left\|\phi\right\|_{U^2_{A_x}H^s[T_s]}
	\;,\\
		\left\|\mathcal{N}^4_t\left[\overline{\phi},\phi,\overline{\phi},\phi\right]\phi\right\|_{\mathrm{D}U^2_{A_x}H^s[T_s]}
	\le\;&
		C(s)T_s^{\frac{1}{2}}\left(1+D_1\right)^{10}D_1^4\left\|\phi\right\|_{U^2_{A_x}H^s[T_s]}
	\;,\\
		\left\|\mathcal{N}^4_x\left[\overline{\phi},\phi,\overline{\phi},\phi\right]\phi\right\|_{\mathrm{D}U^2_{A_x}H^s[T_s]}
	\le\;&
		C(s)T_s^{\frac{1}{2}}\left(1+D_1\right)^{10}D_1^4\left\|\phi\right\|_{U^2_{A_x}H^s[T_s]}
	\;,\\
		\left\|\left|\phi\right|^2\phi\right\|_{\mathrm{D}U^2_{A_x}H^s[T_s]}
	\le\;&
		C(s)T_s^{\frac{1}{2}}\left(1+D_1\right)^8D_1^4\left\|\phi\right\|_{U^2_{A_x}H^s[T_s]}
	\;.\end{split}\end{equation*}
	Summing the above estimates, we conclude from Duhamel's formula that there exists a constant $C_0=C_0(s)>0$ such that
	\[
		\left\|\phi\right\|_{U^2_{A_x}H^s[T_s]}
		\le
		\left\|\phi(0)\right\|_{H^s}
		+
		C_0(s)T_s^{\frac{1}{2}}\left(1+D_1\right)^{14}\left\|\phi\right\|_{U^2_{A_x}H^s[T_s]}
	\;.\]
	Choose $T_\star=T_\star(s,D_1)$ such that $T_\star\le T_1$ and
	\[
		C_0(s)T_\star^{\frac{1}{2}}\left(1+D_1\right)^{14} \le \frac{1}{2}
	\;.\]
	Therefore, if $T_s\le T_\star$ we have $\|\phi\|_{U^2_{A_x}H^s[T_s]}\le 2\|\phi(0)\|_{H^s}$ and hence, by Lemma \ref{ZLEM:OldLemma4.9Restated},
	\[
		\left\|\phi\right\|_{L^\infty_tH^s[T_s]}\le C(s)\left(1+D_1\right)^2\left\|\phi(0)\right\|_{H^s} =: C_\star(s,D_1)\left\|\phi(0)\right\|_{H^s}
	\;.\]
	
	The proof of Theorem \ref{ZTHM:MainThm} is complete.
	
\subsection{Proof of Corollary \ref{ZCOR:MainThmCor}}
	By the Hardy-Littlewood-Sobolev inequality, we have
	\[
		\left\|A_x(t)\right\|_{L^4_x} \lesssim \left\|\left|\phi(t)\right|^2\right\|_{L^{\frac{4}{3}}_x}
		\lesssim \left\|\phi(t)\right\|_{L^2_x}\left\|\phi(t)\right\|_{L^4_x}
	\;.\]
	Therefore,
	\begin{equation*}\begin{split}
		\left\|\phi(t)\right\|_{H^1}^2
	\lesssim\;&
		\left\|\phi(t)\right\|_{L^2_x}^2 + \left\|\mathbf{D}_x\phi(t)\right\|_{L^2_x}^2
		+ \left\|A_x(t)\right\|_{L^4_x}^2\left\|\phi(t)\right\|_{L^4_x}^2
	\\\lesssim\;&
		\mathcal{M}(0) + \left\|\mathbf{D}_x\phi(t)\right\|_{L^2_x}^2 + \mathcal{M}(0)\left\|\phi(t)\right\|_{L^4_x}^4
	\;.\end{split}\end{equation*}
	
	If $\kappa>0$, then $\|\mathbf{D}_x\phi(t)\|_{L^2_x}^2\lesssim\mathcal{E}(0)$ and $\|\phi(t)\|_{L^4_x}^4\lesssim\mathcal{E}(0)$, so
	\[
		\left\|\phi(t)\right\|_{H^1}^2
		\lesssim
		\mathcal{M}(0) + \mathcal{E}(0) + \mathcal{M}(0)\mathcal{E}(0)
	\]
	as desired.
	
	On the other hand, if $\kappa\le 0$, then $\|\mathbf{D}_x\phi(t)\|_{L^2_x}^2\lesssim\mathcal{E}(0)+\|\phi(t)\|_{L^4_x}^4$.
	Now, note that the Gagliardo-Nirenberg interpolation inequality gives $\|\phi(t)\|_{L^4_x}^2\lesssim \|\phi(t)\|_{L^2_x}\|\phi(t)\|_{H^1}$.
	Thus,
	\[
		\left\|\phi(t)\right\|_{H^1}^2
		\lesssim
		\mathcal{M}(0) + \left(\mathcal{E}(0) + \mathcal{M}(0)\left\|\phi(t)\right\|_{H^1}^2\right) + \mathcal{M}(0)^2\left\|\phi(t)\right\|_{H^1}^2
	\;,\]
	and hence we see that if $\mathcal{M}(0)$ were small enough, then $\|\phi(t)\|_{H^1}\le C(\mathcal{M}(0),\mathcal{E}(0))$.
	
	This completes the proof of Corollary \ref{ZCOR:MainThmCor}.


\begin{thebibliography}{99}

	\bibitem{bejenaru_ionescu_kenig_AdvMath.2007}
		I. Bejenaru, A. D. Ionescu, C. E. Kenig:	
		{\em Global existence and uniqueness of Schrödinger maps in dimensions $d\ge 4$}.
		Adv. Math. {\bf 215} (2007), no. 1, 263--291.

	\bibitem{bejenaru_tataru_CommunMathPhys.2009}
		I. Bejenaru, D. Tataru:
		{\em Global wellposedness in the energy space for the Maxwell-Schr\"{o}dinger system}.
		Commun. Math. Phys. {\bf 288} (2009), no. 1, 145--198.

	\bibitem{berge_debouard_saut_Nonlinearity.1995}
		L. Berg\'{e}, A. de Bouard, J.-C. Saut:
		{\em Blowing up time-dependent solutions of the planar, Chern-Simons gauged nonlinear Schr\"{o}dinger equation}.
		Nonlinearity {\bf 8} (1995), no. 2, 235--253.
		
	\bibitem{burq_gerard_tzvetkov_AmerJMath.2004}
		N. Burq, P. G\'{e}rard, N. Tzvetkov:
		{\em Strichartz inequalities and the nonlinear Schr\"{o}dinger equation on compact manifolds}.
		Amer. J. Math. {\bf 126} (2004), no. 3, 569--605.

	\bibitem{cazenave_weissler_ManuscriptaMath.1988}
		T. Cazenave, F. B. Weissler:
		{\em The Cauchy problem for the nonlinear Schr\"{o}dinger equation in $H^1$}.
		Manuscripta Math. {\bf 61} (1988), no. 4, 477--494. 
		
	\bibitem{hadac_herr_koch_AnnIHP.2009}
		M. Hadac, S. Herr, H. Koch:
		{\em Well-posedness and scattering for the KP-II equation in a critical space}.
		Ann. Inst. H. Poincar\'{e} Anal. Non-Lin\'{e}aire {\bf 26} (2009), no. 3, 917--941.
		Erratum:
		Ann. Inst. H. Poincar\'{e} Anal. Non-Lin\'{e}aire {\bf 27} (2010), no. 3, 971--972.
		
	\bibitem{herr_tataru_tzvetkov_DukeMathJ.2011}
		S. Herr, D. Tataru, N. Tzvetkov:
		{\em Global well-posedness of the energy-critical nonlinear Schr\"{o}dinger equation with small initial data in $H^1(\mathbb{T}^3)$}.
		Duke Math. J. {\bf 159} (2011), no. 2, 329--349.

	\bibitem{herr_tataru_tzvetkov_Crelle.2014}
		S. Herr, D. Tataru, N. Tzvetkov:
		{\em Strichartz estimates for partially periodic solutions to Schr\"{o}dinger equations in $4d$ and applications}.
		J. Reine Angew. Math. {\bf 690} (2014), 65--78.

	\bibitem{huh_AbstrApplAnal.2013}
		H. Huh:
		{\em Energy solution to the Chern-Simons-Schr\"{o}dinger equations}.
		Abstr. Appl. Anal. 2013, Article ID 590653.

	\bibitem{ionescu_kenig_DifferentialIntegralEquations.2006}
		A. D. Ionescu, C. E. Kenig:
		{\em Low-regularity Schr\"{o}dinger maps}.
		Differential Integral Equations {\bf 19} (2006), no. 11, 1271--1300.

	\bibitem{ionescu_kenig_CommunMathPhys.2007}
		A. D. Ionescu, C. E. Kenig:
		{\em Low-regularity Schr\"{o}dinger maps, II: Global well-posedness in dimensions $d\ge 3$}.
		Commun. Math. Phys. {\bf 271} (2007), no. 2, 523--559.

	\bibitem{jackiw_pi_PHYS.1990}
		R. Jackiw, S.-Y. Pi:
		{\em Soliton solutions to the gauged nonlinear Schr\"{o}dinger equation on the plane}.
		Phys. Rev. Lett. {\bf 64} (1990), no. 25, 2569--2972.

	\bibitem{jackiw_pi_PHYS.1992}
		R. Jackiw, S.-Y. Pi:
		{\em Self-dual Chern-Simons solitons}.
		In {\em Low-dimensional field theories and condensed matter physics (Kyoto, 1991)}.
		Progr. Theoret. Phys. Suppl. {\bf 107} (1992), 1--40.
		
	\bibitem{kato_AnnIHP.1987}
		T. Kato: {\em On nonlinear Schr\"{o}dinger equations}.
		Ann. Inst. H. Poincar\'{e} Phys. Th\'{e}or. {\bf 46} (1987), no. 1, 113--129. 

	\bibitem{kenig_ponce_vega_InventMath.2004}
		C. E. Kenig, G. Ponce, L. Vega:
		{\em The Cauchy problem for quasi-linear Schr\"{o}dinger equations}.
		Invent. Math. {\bf 158} (2004), no. 2, 343--388.

	\bibitem{koch_NOTES.2014}
		H. Koch: 
		{\em Nonlinear Dispersive Equations},
		Dispersive Equations and Nonlinear Waves, vol. 45 of {\em Oberwolfach Seminars}, Springer, Basel (2014), 1--137.

	\bibitem{koch_tataru_CommPureApplMath.2005}
		H. Koch, D. Tataru:
		{\em Dispersive estimates for principally normal pseudodifferential operators}.
		Comm. Pure Appl. Math. {\bf 58} (2005), no. 2, 217--284.

	\bibitem{koch_tataru_IMRN.2007}
		H. Koch, D. Tataru:
		{\em A Priori Bounds for the 1D Cubic NLS in Negative Sobolev Spaces}.
		Internat. Math. Res. Notices 2007, Article ID rnm053.
		
	\bibitem{liu_smith_RevMatIberoamer.2016}
		B. Liu, P. Smith:
		{\em Global wellposedness of the equivariant Chern-Simons-Schr\"{o}dinger equation}.
		To appear in  Rev. Mat. Iberoam.

	\bibitem{liu_smith_tataru_IMRN.2014}
		B. Liu, P. Smith, D. Tataru:
		{\em Local wellposedness of Chern-Simons-Schr\"{o}dinger}.
		Internat. Math. Res. Notices 2014, no. 23, 6341--6398.

	\bibitem{oh_pusateri_IMRN.2015}
		S.-J. Oh, F. Pusateri:
		{\em Decay and scattering for the Chern-Simons-Schr\"{o}dinger equations}.
		Internat. Math. Res. Notices 2015, Article ID rnv093.

	\bibitem{smith_JMathPhys.2014}
		P. Smith:
		{\em An unconstrained Lagrangian formulation and conservation laws for the Schr\"{o}dinger map system}.
		J. Math. Phys. {\bf 55} (2014), no. 5, 051502.

\end{thebibliography}
\end{document}